\documentclass[11pt]{amsart} 
\usepackage[mathscr]{eucal} 
\usepackage{mathabx} 
\usepackage[T1]{fontenc}
\usepackage[sc]{mathpazo} 

\usepackage{graphicx}
\usepackage{tikz}
\usetikzlibrary{tikzmark,arrows.meta,positioning}
\usetikzlibrary{decorations.pathreplacing}

\usepackage{hyperref}
\usepackage[capitalise,noabbrev]{cleveref}

\usepackage{algorithm}
\usepackage[noEnd=True,indLines=True]{algpseudocodex}

\usepackage{booktabs}
\usepackage{framed}
\usepackage{makecell}
\newcolumntype{V}{!{\vrule width 2pt}}

\usepackage{comment}  
\usepackage{fullpage} 
\numberwithin{equation}{section}

\usepackage[backend=bibtex,giveninits,block=space,sortcites=true,maxbibnames=99,style=numeric,backref=true]{biblatex}
\usepackage{doi}
\DefineBibliographyStrings{english}{%
  backrefpage = {page},
  backrefpages = {pages},
}

\renewbibmacro{in:}{}
\DeclareFieldFormat{pageref}{#1}
\DeclareFieldFormat[article]{volume}{\mkbibbold{#1}}
\DeclareFieldFormat[article]{number}{(#1)}
\DeclareFieldFormat{doi}{\href{https://dx.doi.org/#1}{\textsc{doi}}}
\AtEveryBibitem{\clearfield{issn}\clearfield{isbn}\clearlist{language}}
\renewbibmacro*{volume+number+eid}{%
  \printfield{volume}%
  \printfield{number}%
  \setunit{\bibeidpunct}%
  \printfield{eid}}
\def\abx@missing@entry#1{\abx@missing{#1??}}

\usepackage{todonotes}

\theoremstyle{plain}
\newtheorem{theorem}{Theorem}[section]
\newtheorem{corollary}[theorem]{Corollary}
\newtheorem{proposition}[theorem]{Proposition}

\newtheorem{lemma}[theorem]{Lemma}

\theoremstyle{definition}
\newtheorem{definition}[theorem]{Definition}
\newtheorem{construction}[theorem]{Construction}

\theoremstyle{remark}
\newtheorem{remark}[theorem]{Remark}
\newtheorem{problem}[theorem]{Problem}

\crefname{problem}{Problem}{Problems}
\crefname{corollary}{Corollary}{Corollaries}
\crefname{construction}{Construction}{Constructions}

\DeclareRobustCommand{\tdef}[1]{\textcolor{blue}{\textbf{#1}}}
\newcommand{\alert}[1]{\textcolor{red}{{#1}}}

\def\return{\mathsf{return}}

\def\breturn{\mathsf{breturn}}
\def\run{\mathsf{run}}
\def\pri{\mathsf{pri}}

\def\rrun{\mathsf{rrun}}

\def\brun{\mathsf{brun}}

\def\fh{\operatorname{fh}}
\def\h{\operatorname{h}}

\def\rr{\mathsf{rr}}
\def\brr{\mathsf{brr}}
\def\rbr{\mathsf{rbr}}
\def\bruf{\operatorname{bruf}}
\def\bbruf{\operatorname{bbruf}}

\def\truf{\operatorname{truf}}
\def\bref{\operatorname{bref}}
\def\trf{\operatorname{tref}}
\def\rf{\operatorname{ref}}
\def\c{\mathbf{c}}
\def\u{\mathbf{u}}
\def\vc{\mathbf{vc}}
\def\inc{\operatorname{inc}}

\def\comp{\mathbf{comp}}

\def\cons{\mathbf{cons}}
\def\P{\mathcal{P}}

\def\D{\mathcal{D}}
\def\CP{\mathcal{CP}}
\newcommand{\cCP}{\c \text{-}\CP}
\def\PF{\mathbf{PF}}
\def\PPF{\mathbf{PPF}}
\def\CPF{\mathbf{CPF}}

\newcommand{\ballot}{\mathcal{B}}
\newcommand{\psum}{\Sigma}

\title{Symmetries on vector parking functions via bounded lattice paths}

\author[W. Fang]{Wenjie Fang}
\address[Wenjie Fang]{Univ Gustave Eiffel, CNRS, LIGM, F-77454 Marne-la-Vallée, France}
\email{wenjie.fang@univ-eiffel.fr}

\author[Y. Li]{Yang Li}
\address[Yang Li]{Research Center for Mathematics and Interdisciplinary Sciences, Shandong University \& Frontiers Science Center for Nonlinear Expectations, Ministry of Education, Qingdao 266237, P.R. China}
\email{202421349@mail.sdu.edu.cn}

\author[Z. Lin]{Zhicong Lin}
\address[Zhicong Lin]{Research Center for Mathematics and Interdisciplinary Sciences, Shandong University \& Frontiers Science Center for Nonlinear Expectations, Ministry of Education, Qingdao 266237, P.R. China}
\email{linz@sdu.edu.cn}


\begin{document}

\maketitle

\begin{abstract}
  Partly motivated by enumeration of parking functions and their variants, there is a long-standing interest in lattice paths refined by several statistics, including the notable $\run$ and $\return$.
  We consider bounded lattice paths, which are bounded by a given lattice path and are related to vector parking functions, for which we generalize $\run$ to composition runs, parameterized by a composition.
  By constructing involutions on these paths, we establish symmetries relating composition runs to some generalized return statistics.
  As an application, we settle an open problem of Dai, Fu, and Qiu on rational Dyck paths.
  The symmetries between generalized $\run$ and $\return$ are then transferred to vector parking functions, which suggest a new notion of prime decomposition.
  In the special case of $(a, b)$-parking functions, we compute the generating function refined by $\run$ and $\pri$, which is new even for classical parking functions.
\end{abstract}

\section{Introduction} \label{sec1} 
Let $[n] := \{1,2,\dots,n\}$.
A \tdef{parking function} of length $n$ is a sequence $\alpha = (\alpha_1, \dots, \alpha_n) \in [n]^n$ satisfying
\[
  \forall i \in [n], \; |\{j \in [n]: \alpha_j \le i\}| \ge i.
\]
In other words, the non-decreasing rearrangement $\inc(\alpha) = (\beta_1, \dots, \beta_n)$ of $\alpha$ satisfies $\beta_i \le i$ for every $i$.
The name of parking functions originates from the following scenario: on a one-way street with $n$ parking spots, $n$ cars arrive in order, where the $i$-th one prefers the parking spot $\alpha_i$, and each car tries to park at its preferred spot, and take the next vacant spot if taken, except when no such spot exists, then it fails to park.
The sequence $\alpha$ is a parking function if all cars are able to park.

The concept of parking functions was first introduced by Konheim and Weiss~\cite{Konheim1966} in their work on computer storage.
Using analytic methods, they showed that the number of parking functions of length $n$ is $|\PF_n| = (n + 1)^{n - 1}$, where $\PF_n$ denotes the set of parking functions with length $n$.
Parking functions are closely related to labeled trees, hyperplane arrangements, noncrossing partitions, and Dyck paths.
For a thorough treatment, see the survey by Yan~\cite{Yan2014}.

For a sequence $\alpha$, let $\alpha_i$ denote its $i$-th entry, $|\alpha|$ its length.
By abuse of notation, we also regard $\alpha$ as the multiset of its elements whenever we write  $S\subseteq\alpha$.
A parking function $\alpha = (\alpha_1, \dots, \alpha_n)$ is \tdef{prime} if 
\[
  \forall i > 1, \; |\{j\in [n]: \alpha_j < i\}| \ge i.
\]
Equivalently, $\alpha$ remains a parking function after deleting an entry equal to $1$.
We denote by $\PPF_n$ the set of prime parking functions of length $n$; its counting formula (see~\cite[Ex.~5.49(b)]{Stanley1999}) is $|\PPF_n| = (n - 1)^{n - 1}$.
The notion of prime parking functions was introduced by Gessel~\cite{Stanley1999}, and they are intimately related with minimal transitive cycle factorizations of permutations~\cite{Kim2003}.
Gessel observed that every parking function can be decomposed into prime parking functions, and he defined the following \tdef{prime decomposition} procedure:
\begin{definition}
  Let $\alpha = (\alpha_1, \dots, \alpha_n) \in \PF_n$.
  Take an index subset $L \subseteq [n]$ of maximum size such that the subsequence $\alpha' := (\alpha_i)_{i \in L}$ is a prime parking function.
  Remove $\alpha'$ from $\alpha$ and reduce each remaining entry by $|L|$.
  The resulting tuple $(\alpha_i - |L|)_{i \in [n] \setminus L}$ is again a parking function.
  Thus, $\alpha$ splits into a prime parking function $\alpha'$ and a shorter parking function.
  Iterating this reduction until empty yields a decomposition of $\alpha$ into a sequence of prime parking functions.
\end{definition}

For example, the parking function
$\textcolor{blue}{7}%
\textcolor{red}{1}%
\textcolor{teal}{5}%
\textcolor{blue}{8}%
\textcolor{red}{2}%
\textcolor{red}{1}%
\textcolor{blue}{7}%
\textcolor{teal}{5}%
\textcolor{blue}{9}%
\textcolor{red}{2}\in\PF_{10}$
is decomposed into three prime parking functions $\textcolor{red}{1212}$, $\textcolor{teal}{11}$, and $\textcolor{blue}{1213}$. 

The \tdef{prime} statistic on a parking function $\alpha$, denoted by $\pri(\alpha)$, is the number of prime parking functions in its prime decomposition.
Using the symbolic method~\cite{flajolet}, the prime decomposition of parking functions implies
\begin{equation}
  \sum_{n \geqslant 0} \sum_{\alpha \in \PF_n} t^{\pri(\alpha)} \frac{x^n}{n!} = \left(1-t \sum_{n \geqslant 1} (n-1)^{n-1} \frac{x^n}{n!}\right)^{-1}.
\end{equation}

Motivated by the Pak--Stanley labeling of Shi arrangements, Duarte and Guedes de Oliveira~\cite{Duarte2019} introduced the run statistic on parking functions.
For a parking function $\alpha = (\alpha_1, \dots, \alpha_n)$, the \tdef{run} of $\alpha$, denoted by $\run(\alpha)$, is defined by
\[
  \run(\alpha) = \max \{k \in [n] : [k] \subseteq\alpha\}.
\]
Du~\cite{Du2025} proved that $\run$ and $\pri$ are equidistributed over $\PF_n$:
\begin{equation}
  \sum_{\alpha \in \PF_n} t^{\pri(\alpha)} = \sum_{\alpha \in \PF_n} t^{\run(\alpha)}.
\end{equation}
Then, Li and Lin~\cite{Li2023} used Dyck paths to establish the following joint symmetry
\begin{equation}\label{eq:Li-Lin}
  \sum_{\alpha \in \PF_n} t^{\pri(\alpha)} q^{\run(\alpha)} = \sum_{\alpha \in \PF_n} t^{\run(\alpha)} q^{\pri(\alpha)}.
\end{equation}

Vector parking functions generalize classical parking functions by generalizing the cardinality constraint to be parameterized by a weakly increasing vector $\u$.
See \Cref{sec5} for precise definitions.
The concept of vector parking functions can be traced back to the work of Stanley and Pitman~\cite{Stanley2002} in 2002, and later formalized and named by Kung and Yan~\cite{Kung2003} in 2003.
It was also shown in~\cite{Kung2003} that the total number of vector parking functions of length $n$ is given by a specific evaluation of the $n$-th Gon\v{c}arov polynomial.

The first motivation of the current article is to generalize~\eqref{eq:Li-Lin} to vector parking functions.
Our approach relies on \emph{bounded lattice paths}, which are directed lattice paths weakly under a fixed boundary.
Related definitions are postponed to the next section.
There are few results on joint symmetric distributions of statistics on bounded lattice paths.
A notable one is due to Elizalde and Rubey~\cite{Elizalde2016}, who proved the symmetry of the pairs (\#\textsf{top contacts}, \#\textsf{bottom contacts}) and (\#\textsf{left contacts}, \#\textsf{right contacts}).

In contrast, joint symmetric distributions of various statistics on Dyck paths have already attracted much attention.
Using bijections, Deutsch~\cite{Deu98} proved the
symmetry of the pairs (\#\textsf{return}, \textsf{height of the first peak})
and (\#\textsf{double rise}, \#\textsf{valley}) on Dyck paths.
The symmetry of the pair $(\mathsf{area}, \mathsf{bounce})$ on Dyck paths follows directly from the famous Shuffle Theorem established by Carlsson and Mellit~\cite{Carlsson2018}; the zeta map~\cite{zeta-map} was shown by Haglund~\cite{Hag03} that it transforms the pair $(\mathsf{area}, \mathsf{bounce})$ to $(\mathsf{dinv}, \mathsf{area})$ on Dyck paths.
Nevertheless, finding a combinatorial proof for the symmetry of $(\mathsf{area}, \mathsf{bounce})$ or $(\mathsf{dinv}, \mathsf{area})$ remains a long-standing open problem in algebraic combinatorics, although it is trivial in the algebraic interpretation of the related generating polynomials as bigraded Hilbert series of the subspace of alternants in the diagonal coinvariant space with two sets of variables~\cite{Haglund-qt}.
Using an involution on plane trees, Pappe, Paul and Schilling~\cite{PPS22} proved the symmetry of $(\mathsf{area}, \mathsf{depth})$, with $\mathsf{depth}$ a newly defined statistic.
Li and Lin~\cite{Li2023} recast the parking function statistics $\run$ and $\pri$ on Dyck path and proved combinatorially their joint symmetry:
\begin{equation}\label{eq:Li-Lin2}
  \sum_{D \in \D_n} t^{\run(D)} q^{\return(D)} = \sum_{D \in \D_n} t^{\return(D)} q^{\run(D)}.
\end{equation}

There has also been recent interest in symmetries of rational Dyck paths.
As a natural extension of classical Dyck paths, an $m \times n$ \emph{rational Dyck path} is a lattice path from $(0, 0)$ to $(m, n)$ weakly below the diagonal $my = nx$.
We write $\D_{m,n}$ for the set of all such paths.
See \Cref{sec:rational-dyck} for a formal definition.
Mellit~\cite{Mellit2021} proved the symmetric joint distribution of $(\mathsf{dinv}, \mathsf{area})$ on rational Dyck paths as a consequence of the rational shuffle theorem.
Dai, Fu, and Qiu~\cite{Dai2026} generalized the symmetry~\eqref{eq:Li-Lin2} to rational Dyck paths.
In the case $m \le n$, they defined the \emph{ratio-run} statistic $\rrun$ on $\D_{m,n}$ and established the symmetries
\begin{equation}
  \sum_{D\in\D_{m,n}} t^{\run(D)} q^{\return(D)} = \sum_{D\in\D_{m,n}} t^{\return(D)} q^{\run(D)} ,
\end{equation}
\begin{equation}
  \sum_{D\in\D_{m,n}} t^{\rrun(D)} q^{\return(D)} = \sum_{D\in\D_{m,n}} t^{\return(D)} q^{\rrun(D)} .
\end{equation}
Although their results hold only for $m \le n$, they have sought an extension valid for all $m, n$, raising the following open problem.

\begin{problem}[Dai, Fu, and Qiu {\cite{Dai2026}}]\label{prob:Dai}
  For any integers $m$ and $n$, is it possible to modify the definitions of $\run$ and $\return$ so that their joint distribution is symmetric over $\D_{m,n}$?
\end{problem}

In~\cite{Xin2023}, Xin and Zhang introduced the statistics $\mathsf{area}$, $\mathsf{bounce}$, and $\mathsf{dinv}$ on $\vec{k}$-Dyck paths and considered their symmetric statistics. 
More recently, Qu and Zhang~\cite{Qu2026} proved the symmetry of $(\textsf{area}, \textsf{depth})$ on $\vec{k}$-Dyck paths, which is later extended to Łukasiewicz paths by the first author~\cite{fang-luka}.
They observed that if a bijection exchanging $\text{stat}_1$ and $\text{stat}_2$ on Dyck paths preserves composition type (see \Cref{sec:2.2}), then it yields the symmetry of $(\text{stat}_1, \text{stat}_2)$ on $\vec{k}$-Dyck paths. 
It is the case for the involution on Dyck paths given by Li and Lin~\cite{Li2023} for the symmetry~\eqref{eq:Li-Lin2}, hence $(\run, \return)$ is symmetric on $\vec{k}$-Dyck paths.
Qu and Zhang~\cite{Qu2026} further proposed the following interesting open problem.

\begin{problem}[Qu and Zhang {\cite{Qu2026}}]\label{prob:Qu0}
  Find other pairs of symmetric statistics on Dyck paths with a given composition type $\mathfrak{c}$.
\end{problem}

To our knowledge, $(\run, \return)$ is the only symmetric statistic pair on Dyck paths with fixed composition type.
In \Cref{sec4}, we introduce a bounded version of $\vec{k}$-Dyck paths, and consider the following generalization of \Cref{prob:Qu0}, which, along with \Cref{prob:Dai}, makes the second motivation of the current article.

\begin{problem}\label{prob:Qu}
  Find symmetric statistics on $\vec{k}$-bounded paths with a given composition type $\mathfrak{c}$.
\end{problem}

Here is a summary of our main results:
\begin{itemize}
\item We introduce a family of generalized run and return statistics on bounded lattice paths, which allows us to answer \Cref{prob:Dai,prob:Qu}.
\item We introduce a new notion of prime decomposition for vector parking functions, obtaining symmetric statistic pairs that generalize the symmetry in \Cref{eq:Li-Lin}; see \Cref{coro:PF,coro:CPF}.
\item Using our prime decomposition and a decomposition of ballot paths, we compute the exponential generating function of $(a,b)$-parking functions refined by $\run$ and $\pri$; see \Cref{thm:parking-gf}.
\end{itemize}

The rest of the article is organized as follows.
In \Cref{sec2} we provide the relevant background on bounded lattice paths and define several generalizations of the run and return statistics.
We then construct in \Cref{sec3} four involutions that establish several symmetry properties of generalized run and return statistics on bounded lattice paths, which are applied in \Cref{sec4} to answer \Cref{prob:Dai,prob:Qu}.
In \Cref{sec5}, we introduce a prime decomposition of vector parking functions, to which we transfer some symmetries on bounded lattice paths.
We then compute generating functions of $(a,b)$-parking functions refined by related statistics in \Cref{sec6}.
We conclude in \Cref{sec7} with a summary and some related open problems.

\section{Preliminaries} \label{sec2}

A \tdef{(directed) lattice path} in $\mathbb{N}^2$ is a sequence of \tdef{north steps} (or \tdef{$N$-steps}) $N = (0, 1)$ and \tdef{east steps} (or \tdef{$E$-steps}) $E = (1, 0)$, from the origin to a given point $(m, n) \in \mathbb{N}^2$.
Fix $B$ a lattice path from the origin to $(m, n)$.
A lattice path $P$ is said to be \tdef{bounded} by $B$ if it also ends at $(m,n)$ while lying weakly below $B$.
Let $\P_B$ denote the set of lattice paths bounded by $B$.
We say that $B$ is the \tdef{boundary} of $\P_B$.
\tdef{Dyck paths} of semilength $n$ form a special class of bounded lattice paths with boundary $(EN)^n$.
We denote by $\D_n$ the set of Dyck paths of semilength $n$.
They are counted by the famous Catalan number $C_n = \frac{1}{2n + 1} \binom{2n + 1}{n}$.

For a lattice path $P$, the \tdef{height} of an $E$-step in $P$ is $1$ plus the number of $N$-steps preceding it.
The extra $1$ is for accommodating the convention of parking functions introduced later.
For convenience, we consider all lattice paths as preceded by a virtual $E$-step of height $1$.
This ensures a uniform treatment in later definitions for all lattice paths.
The \tdef{height sequence} of $P$, denoted $\h(P)$, is the weakly increasing list of the heights of all $E$-steps.
For example, the lattice path $P$ in \Cref{fig-bounded-path} has $\h(P) = (2, 3, 3, 4, 5, 5, 5, 10, 12, 13, 13, 13, 13, 15, 15)$.
A lattice path $P$ is uniquely determined by its height sequence $\h(P)$.

A \tdef{floor} in $P$ is a non-empty maximal sequence of consecutive $E$-steps.
We always order the floors from left to right.
All $E$-steps in the same floor have the same height.
The \tdef{height of the $i$-th floor} is defined as that common height.
The \tdef{floor height sequence}, denoted $\fh(P)$, is the increasing sequence of floor heights.
For example, the lattice path $P$ in \Cref{fig-bounded-path} has $\fh(P) = (1, 2, 3, 4, 5, 10, 12, 13, 15)$.
Note that by the convention of the starting virtual $E$-step, $\fh(P)$ always starts with $1$.

\subsection{Natural extensions of the classical run and return statistics}

For any Dyck path $D\in\D_n$, the \tdef{run} statistic is defined in \cite{Duarte2019} as
\[
  \run(D) = \max \{k : [k] \subseteq \fh(D) \}.
\]
There are two interpretations of $\run(D)$:
\begin{itemize}
\item It is the number of initial floors of $D$ whose heights are exactly $1, 2, \dots$ in order.
\item It is the maximum number of initial floors of the boundary $(EN)^n$ whose heights all appear as floor heights in $D$.
\end{itemize}
They coincide on Dyck paths but differ on general bounded lattice paths.
We therefore introduce two $\run$-like statistics: the \emph{run}, denoted $\run$, and the \emph{boundary run}, denoted $\brun_B$.

\begin{definition}
  Let $B$ be a boundary ending at $(m,n)$ with floor heights $\fh(B) = (b_1, b_2, \dots, b_\ell)$. 
  For $P \in \P_B$, define
  \begin{itemize}
  \item \tdef{Run} of $P$: $\run(P) = \max \{k : [k] \subseteq \fh(P) \}$;
  \item \tdef{Boundary run} of $P$: $\brun_B(P) = \max \{k : \{b_1, b_2, \dots, b_k\} \subseteq \fh(P) \}$.
  \end{itemize}
  Note that $\run(P)$ depends only on $P$ and not on the boundary $B$.
\end{definition}

\begin{figure}
  \centering
  \begin{tikzpicture}[
    scale=0.45,
    >=stealth,
    grid/.style={gray, thin},
    boundary/.style={},
    pathA/.style={blue!60, very thick},
    pathB/.style={very thick},
    highlight/.style={red, dotted, very thick}
    ]

    \begin{scope}
      
      \draw[grid] (0,0) grid (15,15);

      \draw[boundary] (0,0) rectangle (15,15);

      \draw[pathA]
      (0,0) -- (0,2) -- (1,2) -- (1,4) -- (5,4) 
      -- (5,5) -- (5,5) -- (5,9) -- (6,9) -- (6,10)
      -- (7,10) -- (7,12) -- (11,12)--(11,14)--(14,14)--(14,15)--(15,15);

      \draw[pathB]
      (-1,0)--(0,0) -- (0,1) -- (1,1) -- (1,2) -- (3,2) -- (3,3) -- (4,3)--(4,4)--(7,4)--(7,9)--(8,9)--(8,11)--(9,11)--(9,12)--(13,12) --(13,14)--(15,14)--(15,15);

      \draw[-latex, color=red] (-2, 0) -- (-1.2, 0);
      \draw[-latex, color=red] (-2, 4) -- (-1.2, 4);
      \draw[-latex, color=red] (-2, 12) -- (-1.2, 12);
      \draw[-latex, color=red] (-2, 14) -- (-1.2, 14);
    \end{scope}
  \end{tikzpicture}
  \caption{An example of bounded paths, with marked floors sharing steps.}
  \label{fig-bounded-path}
\end{figure}
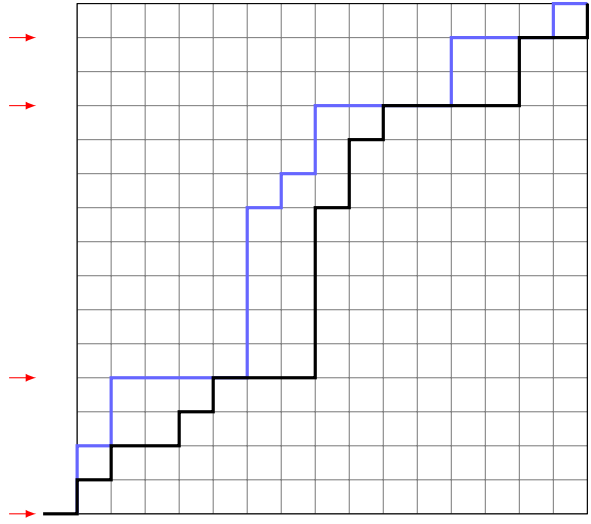

For example, for $P$ the path bounded by $B$ in \Cref{fig-bounded-path}, we have $\run(P) = 5$ and $\brun_B(P) = 4$.

For a Dyck path $D \in \D_n$, the classical $\return$ statistic is the number of $E$-steps starting from the diagonal $y = x$ in $D$. 
It can be interpreted as the number of floors sharing steps with the boundary $(EN)^n$.
This motivates the following extension to bounded paths. 

\begin{definition}
  Given a path $P \in \P_B$ bounded by $B$, we define $\return_B(P)$ to be the number of floors of $P$ that have at least one overlapping step with $B$.
  Note that the virtual $E$-step before the origin is taken into account.
\end{definition}

As an example, the bounded path $P$ in \Cref{fig-bounded-path} has $\return_B(P) = 4$, with contributing floors marked by red arrows.

\subsection{Composition run and block return} \label{sec:2.2}

The two statistics run and boundary run are different but exhibit similar properties.
To unify them, we now introduce \emph{composition run}.

A \tdef{composition} $\c$ of size $n > 0$, is a finite sequence $\c = (c_1, c_2, \dots, c_k)$ of positive integers such that $\sum_{i=1}^k c_i = n$.
We also write $\c \vDash n$.
The \tdef{composition-type} of a lattice path $P$ with $n$ east steps, denoted by $\comp(P)$, is the composition $\c \vDash n$ such that the $i$-th floor of $P$ is of length $c_i$. 
For instance, the bounded lattice path $P$ in \Cref{fig-bounded-path} has composition-type $(1, 1, 2, 1, 3, 1, 1, 4, 2)$.

For a lattice path $P$, suppose $\fh(P) = (h_1, h_2, \dots, h_m)$. 
The composition $(h_1, h_2 - h_1, \dots, h_m - h_{m-1}) \vDash h_m$ is called the \tdef{vertical composition-type} of $P$, denoted by $\vc(P)$.
For instance, for $B$ the boundary path in \Cref{fig-bounded-path}, we have $\vc(B) = (1, 2, 2, 5, 1, 2, 2, 1)$.
Note that by the convention of the starting virtual $E$-step, $\vc(P)$ always starts with $1$.

We need a partial order on compositions. Let $\mathbf{a} = (a_1, \dots, a_{\ell}) \vDash n$ and $\mathbf{b} = (b_1, \dots, b_{m}) \vDash n$ with $m \leq \ell$.
We say that $\mathbf{a}$ is smaller than $\mathbf{b}$ in the \tdef{componentwise order}, written as $\mathbf{a} \le_c \mathbf{b}$, if for all $i \in [m]$, $a_i \leq b_i$.
When we write $\mathbf{a} \le_c \mathbf{b}$, we assume implicitly that $\mathbf{a}, \mathbf{b}$ have the same size.
For example, let $\mathbf{a} = (1, 2, 3, 1, 5, 4)$, $\mathbf{b} = (1, 1, 2, 1, 1, 4, 3, 3)$, and $\mathbf{c} = (1, 1, 1, 2, 1, 4, 3, 3)$.
Then $\mathbf{b} \le_c \mathbf{a}$, $\mathbf{c} \nleq_c \mathbf{a}$.

\begin{definition}
  Let $P \in \P_B$ be a lattice path bounded by $B$, and $\vc(B) = (v_1, \ldots, v_m)$.
  For $\c = (c_1, \dots, c_{m'}) \le_c \vc(B)$, the \tdef{composition run} of $P$ corresponding to $\c$ and boundary $B$ is
  \[
    \run_{\c,B}(P) = \max\{k \leq m' : \{c_1, c_1 + c_2, \dots, c_1 + \dots + c_k\} \subseteq \fh(P) \}. 
  \]
\end{definition}

For example, take the bounded path $P$ in \Cref{fig-bounded-path} with $\c = (1, 1, 2, 1, 3, 1, 1, 4, 2)$ and $\c' = (1, 1, 1, 2, 1, 1, 2, 1, 4, 2)$, then $\run_{\c,B}(P) = \run_{\c',B}(P) = 4$.
Composition runs generalize both $\run$ and $\brun$: we have $\run(P) = \run_{\c, B}(P)$ for $\c = (1, \dots, 1)$, and $\brun_B(P) = \run_{\c', B}(P)$ for $\c' = \vc(B)$.

We now introduce a $\return$-like statistic.

\begin{definition}
  For a path $P \in \P_B$ bounded by $B$, a floor $F$ of $P$ is called a \tdef{block return floor} if
  \begin{itemize}
  \item $F$ has at least one step overlapping $B$; or
  \item there is a block return floor whose height is one plus that of $F$.
  \end{itemize}
  Equivalently, one can begin with all floors that have an overlapping step with the boundary, and then repeatedly select any floor of height precisely one less than a floor already selected. All block return floors are thus obtained in the end. The \tdef{block return} of $P$, denoted by $\breturn_B(P)$, is the number of its block return floors.
\end{definition}

The name ``block return'' reflects the fact that block return floors forms blocks of contiguous floors with height difference $1$, and the highest floors of blocks are exactly those counted by $\return_B(P)$.
For the path $P$ bounded by $B$ in \Cref{fig-bounded-path}, every floor except the sixth is a block return floor, hence $\breturn_B(P) = 8$.

When no confusion may arise, we omit the subscript $B$ from $\run_{\c, B}$, $\return_B$, and $\breturn_B$ hereinafter. 

\begin{figure}
  \centering
  \begin{tikzpicture}[
    scale=0.45,
    >=stealth,
    grid/.style={gray, thin},
    boundary/.style={},
    pathA/.style={blue!60, very thick},
    pathB/.style={very thick},
    highlight/.style={red, dotted, very thick}
    ]

    \begin{scope}
      \draw[grid] (0,0) grid (10,10);

      \draw[boundary] (0,0) rectangle (10,10);

      \draw[pathA]
      (0,0) -- (0,2) -- (1,2) -- (1,5) -- (3,5) -- (3,7) -- (6,7)--(6,8)-- (9,8)--(9,10)--(10,10);

      \draw[pathB]
      (-1,0)--(0,0)--(2,0)--(2,2)--(5,2)--(5,5)--(10,5)--(10,10);
    \end{scope}

    \begin{scope}[xshift=15cm]
      \draw[grid] (0,0) grid (10,10);

      \draw[boundary] (0,0) rectangle (10,10);

      \draw[pathA]
      (0,0) -- (0,2) -- (1,2) -- (1,5) -- (3,5) -- (3,7) -- (6,7)--(6,8)-- (9,8)--(9,10)--(10,10);

      \draw[pathB]
      (-1,0)--(0,0)--(2,0)--(2,2)--(3,2)--(3,3)--(5,3)--(5,5)--(10,5)--(10,10);

      \draw[highlight] (3,3) -- (3,7);
      \draw[highlight] (5,3) -- (5,7);
    \end{scope}
  \end{tikzpicture}
  \caption{Examples for the boundary constraint.}
  \label{fig-constraint}
\end{figure}
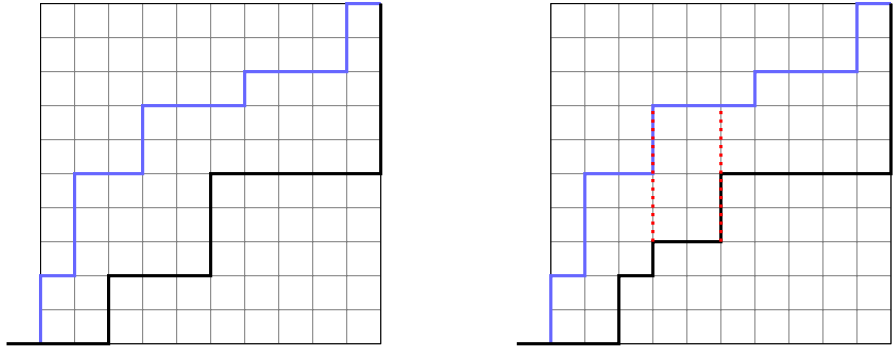

Let $\alpha = (\alpha_1, \dots, \alpha_m)$ and $\beta = (\beta_1, \dots, \beta_{m'})$ be compositions of $n$.
We say that $\alpha$ is \tdef{constrained} by $\beta$ if the following implication holds for every pair $(i, j) \in [m] \times [m']$:
\begin{equation} \label{eq:boundary-constrain}
  \sum_{k = 1}^{i - 1} \alpha_k \ge \sum_{k = 1}^{j - 1} \beta_k
  \quad \Longrightarrow \quad
  \sum_{k = 1}^{i} \alpha_k \geq \sum_{k = 1}^{j} \beta_k.
\end{equation}
Here, empty sums equal $0$.
Intuitively, it means that the partial sums of $\alpha$ never ``cross'' those of $\beta$ in a way that violates a natural monotonicity condition.
We denote by $\cons(\beta)$ the set of all compositions of $n$ constrained by $\beta$.

\begin{definition} \label{def:boundary-constrain}
  A lattice path $P \in \P_B$ bounded by $B$ is \tdef{boundary constrained} if $\comp(P) \in \cons(\comp(B))$.
  We denote by $\CP_B$ the set of all boundary constrained paths in $\P_B$.
\end{definition}

In geometric terms, \Cref{def:boundary-constrain} forbids the existence of a floor $F$ of $P$ with abscissa from $x_F$ to $x'_F$ and a floor $G$ of $B$ with abscissa from $x_G$ to $x'_G$ such that $x_G \leq x_F < x'_F < x'_G$.
For instance, the first path in \Cref{fig-constraint} is boundary constrained, 
while the second is not, due to the third floor of the bounded path and the fourth floor of the boundary.

For convenience, for a floor $F$ in $P \in \CP_B$ constrained by the boundary $B$, we define the \tdef{first-cover floor} of $F$ to be the leftmost floor of $B$ whose projection on the $x$-axis overlaps that of $F$ by a segment. It is clear that if $F$ shares an $E$-step with $B$, the shared step is always on the first-cover floor of $F$. We have the following simple consequence for boundary constraint.

\begin{proposition} \label{prop:bounded-index}
  For $P \in \CP_B$ constrained by the boundary $B$, let $F$ be the $i$-th floor of $P$ and $F_B$ its first-cover floor. Suppose that $j$ is the index of $F_B$ in $B$. Then
  \begin{enumerate}
  \item $j \geq i$;
  \item For $G$ the $i'$-th floor of $P$ with $i' > i$, and $G_B$ its first cover floor. Let $j'$ be the index of $G_B$ in $B$, then $j' \geq j + (i' - i)$.
  \end{enumerate}
\end{proposition}
\begin{proof}
  For the first point, by \Cref{eq:boundary-constrain}, iterating from the empty sum, for all $\ell$ we have
  \[
    \sum_{k=1}^{\ell} \comp(P)_k \geq \sum_{k=1}^{\ell} \comp(B)_k.
  \]
  We observe that, for $\ell = i$ (resp. $\ell = i - 1$), the left-hand side of the equation above is the $x$-coordinate of the right (resp. left) endpoint of $F$, while the right-hand side is that of the $i$-th floor $F_B'$ of $B$. Hence, the projection of $F_B'$ on the $x$-axis is either totally to the left of that of $F$, thus also to the left of $F_B$, or overlaps with that of $F$ to the left, meaning that $F_B' = F_B$ by the minimality of $F_B$. In both cases, the index $i$ of $F_B'$ is at most the index $j$ of $F_B$.

  For the second point, we only need to show it for $i' = i + 1$. Clearly $j' \geq j$. However, if $j' = j$, then the $x$-axis projection of $F_B$ has its rightmost endpoint strictly to the right of the leftmost endpoint of the $(i + 1)$-st floor of $P$, which is the rightmost endpoint of the $x$-axis projection of $F$. This violates the boundary constraint. Thus, $j' > j$.
\end{proof}

In some cases $\CP_B$ equals $\P_B$.
\begin{proposition}\label{prop:bounded}
  If a boundary path $B$ contains no two consecutive $E$-steps, then $\CP_B = \P_B$.
\end{proposition}
\begin{proof}
  Such a boundary $B$ has every floor of length $1$, thus $\comp(B) = (1, \dots, 1)$.
  For $P \in \P_B$, let $\comp(P) = (c_1, \dots, c_m)$, with $c_i \geq 1$ for all $i$.
  For valid $i, j > 0$, if $\sum_{k = 1}^{i - 1} c_k \ge j - 1 = \sum_{k = 1}^{j - 1} 1$ with $j > 0$, then $\sum_{k=1}^{i} c_k \ge j = \sum_{k=1}^{j} 1$. Hence, $\comp(P) \in \cons(\comp(B))$.
\end{proof}

\section{Symmetries of statistics on bounded lattice paths}\label{sec3}

In this section, we mainly show that, for a composition $\c$ satisfying a certain condition of admissibility, $\run_{\c}$ and $\return$ are symmetric on $\CP_B$.
For a clear exhibition, we start in \Cref{subsec3.1} with an involution $\Phi$ realizing the symmetry of $(\run, \return)$ over $\CP_B$ without $\c$.
As a detour, we give in \Cref{subsec3.2} an involution $\Theta$ based on $\Phi$ for the symmetry of $(\run, \breturn)$ on $\P_B$.
Subsection~\ref{subsec3.3} presents a third involution $\Psi$ showing the symmetry of $(\brun, \return)$ on $\CP_B$.
Finally, in Subsection~\ref{subsec3.4}, we modify $\Psi$ to show the symmetry of $(\run_{\c}, \return)$ for admissible $\c$.

The algorithms at the core of all these involutions share the same scheme: floors are lifted after some initial floors, transferring $1$ from a $\run$-like statistic to a $\return$-like statistic.
\Cref{tab:algorithms} summarizes the algorithms and the statistic pair associated with each involution.

\begin{table}[htbp]
  \centering
  \caption{Algorithms and associated statistic pairs for distinct involutions}
  \label{tab:algorithms}
  \begin{tabular}{cccc}
    \toprule 
    Statistic Pair & Section & Involution & Algorithms \\
    \midrule
    $(\run,\return)$ & \Cref{subsec3.1} & $\Phi$ & \Cref{constr:lifting,constr:descending} \\
    $(\run,\breturn)$ & \Cref{subsec3.2} & $\Theta$ & \Cref{constr:lifting-theta,constr:descending-theta}  \\
    $(\brun,\return)$ & \Cref{subsec3.3} & $\Psi$ & \Cref{constr:lifting-psi,constr:descending-psi}  \\
    $(\run_\c,\return)$ & \Cref{subsec3.4} & $\Psi_\c$ & \Cref{constr:lifting-psi-c,constr:descending-psi-c} \\
    \bottomrule
  \end{tabular}
\end{table}

\subsection{The symmetry of $(\run,\return)$ on $\CP_B$} \label{subsec3.1}

Given a path $P \in \P_B$ bounded by $B$, its \tdef{run floor} is the $\run(P)$-th floor.
The \tdef{run-return} statistic of $P$, denoted $\rr(P)$, is the number of floors of $P$ counted in both $\run(P)$ and $\return(P)$.
For instance, both bounded lattice paths in \Cref{fig:phi} have run-return value $3$.

\begin{lemma}\label{lem:phi}
  For a path $P \in \CP_B$, the floors counted by $\rr(P)$ are initial and consecutive.
\end{lemma}
\begin{proof}
  Only the first $\run(P)$ floors may contribute to $\rr(P)$.
  Let $0 < i < \run(P)$ such that the $i$-th floor $F$ of $P$ does not contribute to $\rr(P)$.
  Let the $j$-th floor $F_B$ of $B$ be the first-cover floor of $F$.
  Since $F$ is not counted in $\return(P)$, we have $\fh_i(P) < \fh_j(B)$.
  Since $i + 1 \le \run(P)$, the definition of $\run(P)$ ensures $\fh_{i + 1}(P) = \fh_i(P) + 1$, meaning that $\fh_{i + 1}(P) < \fh_j(B) + 1 < \fh_{j + 1}(B)$.
  By \Cref{prop:bounded-index}(2), the first-cover floor of the $(i + 1)$-st floor of $P$ has index in $B$ at least $j + 1$, thus height at least $\fh_{j + 1}(B)$.
  This implies the $(i + 1)$-st floor of $P$ does not contribute to $\return(P)$, hence neither to $\rr(P)$.
  The claim follows as the absence of the $i$-th floor in $\rr(P)$ implies that of the $(i + 1)$-st floor.
\end{proof}

For any boundary path $B$, we define
\begin{align*}
  \CP_B(k,\ell)  &:= \{ P \in \CP_B : \run(P)=k,\ \return(P)=\ell \},\\
  \CP_B^r(k,\ell) &:= \{ P \in \CP_B : \run(P)=k,\ \return(P)=\ell,\ \rr(P)=r \}.
\end{align*}

\begin{construction} \label{constr:lifting}
  For a boundary $B$ and positive integers $r, \ell, k$ with $r \leq \ell$ and $r < k$, we define a map $\phi$ from $\CP_B^r(k, \ell)$ to $\CP_B^r(k - 1, \ell + 1)$ as follows.
  For $P \in \CP_B^r(k, \ell)$, we do the following:
  \begin{enumerate}
  \item Set the run floor of $P$ as the current floor;
  \item Lift the current floor up until hitting $B$ at an $E$-step or up to one unit below the next floor;
  \item If the lifted floor hits $B$ at an $E$-step, terminate; otherwise, set the next floor as the current floor and go to step 2.
  \end{enumerate}
  We set $\phi(P)$ to be the path thus obtained after the procedure.
\end{construction}

\begin{construction} \label{constr:descending}
  For a boundary $B$ and positive integers $r, \ell, k$ with $r < \ell$ and $r \leq k$, we define a map $\overline{\phi}$ from $\CP_B^r(k, \ell)$ to $\CP_B^r(k + 1, \ell - 1)$ as follows.
  For $P' \in \CP_B^r(k, \ell)$, let $\rf(P')$ be the index of the lowest one above the run floor of $P'$ that hits $B$ at an $E$-step.
  We then do the following:
  \begin{enumerate}
  \item Set the current floor to be $\rf(P')$-th floor of $P'$;
  \item Lower the current floor down to one unit above the preceding floor;
  \item If the preceding floor is the run floor of $P'$, terminate; otherwise, set the preceding floor as the current floor and go to step 2.
  \end{enumerate}
  We set $\overline{\phi}(P')$ to be the path thus obtained after the procedure.
\end{construction}

See \Cref{fig:phi} for an example of the application of $\phi$ and $\overline{\phi}$. Note that in \Cref{constr:lifting,constr:descending}, the lifting and lowering of a floor may not change its height; see the fifth floor in \Cref{fig:phi}.

\begin{lemma}\label{lem:phi2}
  For a boundary $B$ and positive integers $r, \ell, k$ such that $r \leq \ell$ and $r < k$, the map $\phi$ is a composition-type-preserving bijection from $\CP_B^r(k, \ell)$ to $\CP_B^r(k - 1, \ell + 1)$, with $\overline{\phi}$ its inverse.
\end{lemma}
\begin{proof}
  We first show that $\phi$ in \Cref{constr:lifting} is well-defined and preserves composition-type.
  Let $P \in \CP_B^r(k, \ell)$.
  Since \Cref{constr:lifting} always terminates, $\phi(P)$ is well-defined.
  As $r < k$, \Cref{lem:phi} implies that the run floor of $P$ does not share any $E$-step with $B$, and by the maximality of the run floor, the height difference between it and the next floor is at least $2$.
  Hence, its height changes in \Cref{constr:lifting}, meaning that $\run(\phi(P)) = \run(P) - 1 = k - 1$.
  Noting that $\phi(P)$ remains boundary constrained by $B$ as $\comp(\phi(P)) = \comp(P)$, we now show that the last floor $F$ lifted in \Cref{constr:lifting} changed its height.
  Otherwise, in $\phi(P)$, $F$ shares an $E$-step with some floor $G$ of $B$, while the floor $F'$ preceding $F$ does not.
  Moreover, the height of $F$ is $h_{F} = h_{G} = h_{F'} + 1$, with $h_{F'}$ (resp. $h_G$) the height of $F'$ (resp. $G$).
  Let $G'$ be the floor of $B$ preceding $G$ of height $h_{G'} < h_G$.
  If $F'$ is directly under $G'$, then $h_{F'} \leq h_{G'} < h_G = h_{F'} + 1$, meaning that $h_{F'} = h_{G'}$, which is impossible since $F'$ shares no $E$-step with $B$.
  But otherwise, $F'$ and $G$ violate the fact that $P$ and thus $\phi(P)$ is boundary constrained by $B$.
  Therefore, $F$ must have changed its height in \Cref{constr:lifting}, hence $\return(\phi(P)) = \return(P) + 1 = \ell + 1$.
  We conclude that $\phi(P) \in \CP_B^r(k - 1, \ell + 1)$, meaning that $\phi(P)$ has the correct codomain.
  As only floor heights are changed, $\phi$ clearly preserves composition-type.

  We now show the same for $\overline{\phi}$.
  Let $P' \in \CP_B^r(k - 1, \ell + 1)$.
  As $r < \ell + 1$, \Cref{lem:phi} guarantees that at least one floor after the run floor of $P'$ shares an $E$-step with $B$, meaning that the first step in \Cref{constr:descending} is feasible.
  The rest of \Cref{constr:descending} always terminates, thus $\overline{\phi}(P')$ is well-defined.
  Now consider the first floor $F$ of $P'$ that is lowered and the floor $F'$ before it.
  By \Cref{constr:descending}, $F$ overlaps a floor $G$ of $B$ on an $E$-step.
  As $P'$ is boundary constrained by $B$, the projection of $F'$ on the $x$-axis does not lie entirely inside that of $G$, and $F'$ also does not share any $E$-step of $B$ due to the minimality of $F$.
  Hence, the height difference between $F$ and $F'$ is at least $2$. Thus $F$ is indeed lowered, and hence $\return(\overline{\phi}(P')) = \return(P') - 1$.
  Moreover, by the termination condition of \Cref{constr:descending}, we have $\run(\overline{\phi}(P')) = \run(P') + 1$, hence $\overline{\phi}(P')$ has the correct codomain.
  It is also clear that $\overline{\phi}$ preserves composition-type as only floor heights are changed.

  We now show the bijectivity of $\phi$.
  Clearly, $\overline{\phi}(\phi(P)) = P$, as the procedure to compute $\phi(P)$ stops at a floor blocked by $B$, which is the first such floor after the run floor of $\phi(P)$, and $\overline{\phi}$ reverses the changes of heights made by $\phi$.
  To show that $\phi(\overline{\phi}(P')) = P'$, suppose that $P = \overline{\phi}(P')$.
  We only need to show that \Cref{constr:lifting} on $P$ stops at exactly the $\rf(P')$-th floor.
  If it stops prematurely, then there is a floor in $P'$ before the $\rf(P')$-th one that hits the boundary $B$ on an $E$-step, which is impossible by the minimality of $\rf(P')$.
  Hence, $\phi(\overline{\phi}(P')) = P'$, with which we conclude.
\end{proof}

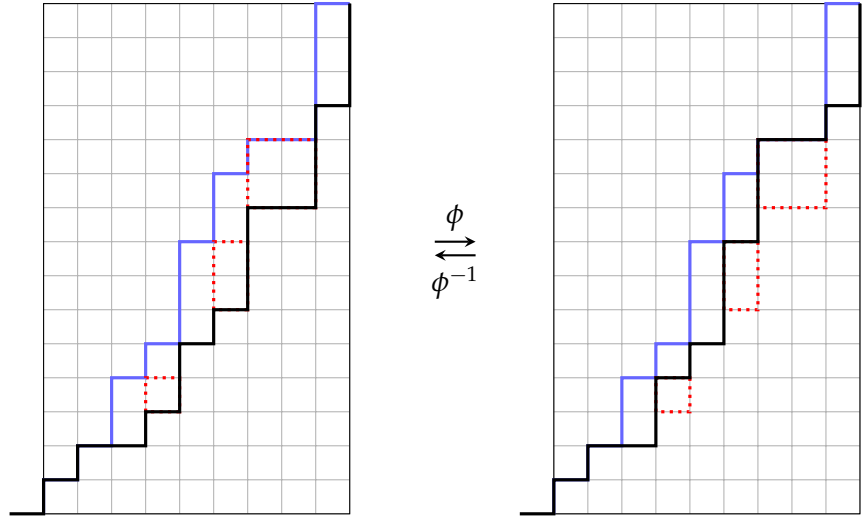
\begin{figure}
  \centering
  \begin{tikzpicture}[
    scale=0.45,
    >=stealth,
    grid/.style={gray!60, thin},
    boundary/.style={},
    pathA/.style={blue!60, very thick},
    pathB/.style={very thick},
    highlight/.style={red, dotted, very thick}
    ]

    \begin{scope}
      \draw[grid] (0,0) grid (9,15);
      \draw[boundary] (0,0) rectangle (9,15);
      \draw[pathA]
      (0,0) -- (0,1) -- (1,1) -- (1,2) -- (2,2)
      -- (2,4) -- (3,4) -- (3,5) -- (4,5) -- (4,8)
      -- (5,8) -- (5,10) -- (6,10) -- (6,11) -- (8,11) -- (8,14) -- (8,15) -- (9,15);

      \draw[highlight] (3,3) rectangle (4,4);
      \draw[highlight] (5,6) rectangle (6,8);
      \draw[highlight] (6,9) rectangle (8,11);

      \draw[pathB]
      (-1,0)--(0,0) -- (0,1) -- (1,1) -- (1,2) -- (3,2) -- (3,3) -- (4,3) -- (4,5) -- (5,5) -- (5,6) -- (6,6)
      -- (6,9) -- (8,9) -- (8,12) -- (9,12) -- (9,15);
    \end{scope}

    \begin{scope}[xshift=11.5cm, yshift=8cm]
      \draw[->, thick] (0,0) -- (1.2,0) node[midway,above] {$\phi$};
      \draw[<-, thick] (0,-0.4) -- (1.2,-0.4) node[midway,below] {$\phi^{-1}$};
    \end{scope}

    \begin{scope}[xshift=15cm]
      \draw[grid] (0,0) grid (9,15);
      \draw[boundary] (0,0) rectangle (9,15);
      \draw[pathA]
      (0,0) -- (0,1) -- (1,1) -- (1,2) -- (2,2)
      -- (2,4) -- (3,4) -- (3,5) -- (4,5) -- (4,8)
      -- (5,8) -- (5,10) -- (6,10) -- (6,11) -- (8,11) -- (8,14) -- (8,15) -- (9,15);

      \draw[highlight] (3,3) rectangle (4,4);
      \draw[highlight] (5,6) rectangle (6,8);
      \draw[highlight] (6,9) rectangle (8,11);

      \draw[pathB]
      (-1,0)--(0,0) -- (0,1) -- (1,1) -- (1,2) -- (3,2) -- (3,4) -- (4,4) -- (4,5) -- (5,5) -- (5,8) -- (6,8)
      -- (6,11) -- (8,11) -- (8,12) -- (9,12) -- (9,15);
    \end{scope}
  \end{tikzpicture}
  \caption{An example of the bijection $\phi$.}
  \label{fig:phi}
\end{figure}

\begin{theorem}\label{thm:Phi}
  For any boundary path $B$, we have the symmetry
  \begin{equation}
    \sum_{P \in \CP_B} q^{\run(P)} t^{\return(P)} = \sum_{P \in \CP_B} q^{\return(P)} t^{\run(P)}.
  \end{equation}
\end{theorem}
\begin{proof}
  The set $\CP_B$ is partitioned into $\CP_B^r(k, \ell)$ for positive integers $k, \ell \ge r$.
  For any $\alpha \in \CP_B^r(k, \ell)$, we set $\Phi(\alpha) := \phi^{k-\ell}(\alpha)$; see \Cref{fig:Phi} for an example.
  It follows from \Cref{lem:phi2} that $\Phi$ is an involution exchanging $\run$ and $\return$, which implies our claim.
\end{proof}

\begin{remark}
  \Cref{thm:Phi} does not hold without the boundary constraint, as it is crucial for the decrement of $\return$ in $\overline{\phi}$ as shown in the proof of \Cref{lem:phi2}. 
    For example, when the boundary path is $NNEE$, the bounded path $NENE$ has run $3$, but there is no bounded path in $\P_B$ whose return is $3$.
\end{remark}

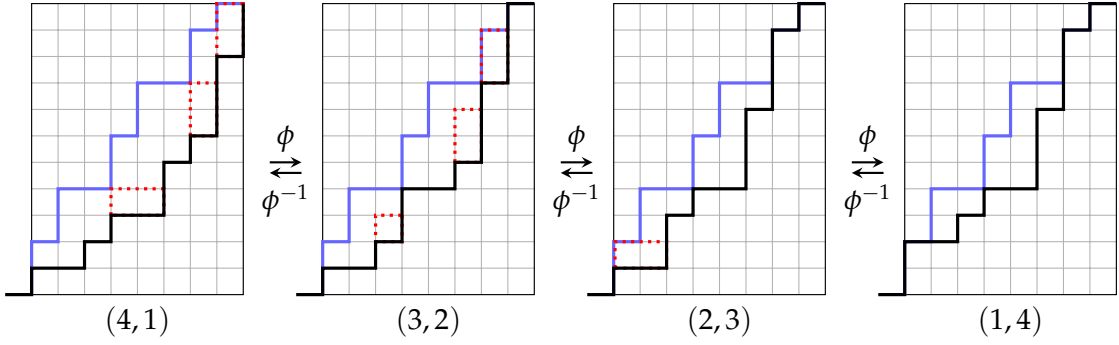
\begin{figure}
  \centering
  \begin{tikzpicture}[
    scale=0.35,
    >=stealth,
    grid/.style={gray!60, thin},
    boundary/.style={},
    pathA/.style={blue!60, very thick},
    pathB/.style={very thick},
    highlight/.style={red, dotted, very thick}
    ]

    \begin{scope}
      \draw[grid] (0,0) grid (8,11);
      \draw[boundary] (0,0) rectangle (8,11);
      \draw[pathA]
      (0,0) -- (0,2) -- (1,2) -- (1,4) -- (3,4)
      -- (3,6) -- (4,6) -- (4,8) -- (6,8) -- (6,10)
      -- (7,10) -- (7,11) -- (8,11);

      \draw[highlight] (3,3) rectangle (5,4);
      \draw[highlight] (6,6) rectangle (7,8);
      \draw[highlight] (7,9) rectangle (8,11);

      \draw[pathB]
      (-1,0)--(0,0) -- (0,1) -- (2,1) -- (2,2) -- (3,2) -- (3,3) -- (5,3) -- (5,5) -- (6,5) -- (6,6) -- (7,6)
      -- (7,9) -- (8,9) -- (8,11);

      \node at (4,-1){$(4,1)$};
    \end{scope}

    \begin{scope}[xshift=9cm, yshift=5cm]
      \draw[->, thick] (0,0) -- (1.2,0) node[midway,above] {$\phi$};
      \draw[<-, thick] (0,-0.4) -- (1.2,-0.4) node[midway,below] {$\phi^{-1}$};
    \end{scope}

    \begin{scope}[xshift=11cm]
      \draw[grid] (0,0) grid (8,11);
      \draw[boundary] (0,0) rectangle (8,11);
      \draw[pathA]
      (0,0) -- (0,2) -- (1,2) -- (1,4) -- (3,4)
      -- (3,6) -- (4,6) -- (4,8) -- (6,8) -- (6,10)
      -- (7,10) -- (7,11) -- (8,11);

      \draw[highlight] (3,3) rectangle (2,2);
      \draw[highlight] (5,5) rectangle (6,7);
      \draw[highlight] (6,8) rectangle (7,10);

      \draw[pathB]
      (-1,0)--(0,0) -- (0,1) -- (2,1) -- (2,2) -- (3,2) -- (3,4) -- (5,4) -- (5,5) -- (6,5) -- (6,8) -- (7,8)
      -- (7,11) -- (8,11);

      \node at (4,-1){$(3,2)$};
    \end{scope}

    \begin{scope}[xshift=20cm, yshift=5cm]
      \draw[->, thick] (0,0) -- (1.2,0) node[midway,above] {$\phi$};
      \draw[<-, thick] (0,-0.4) -- (1.2,-0.4) node[midway,below] {$\phi^{-1}$};
    \end{scope}

    \begin{scope}[xshift=22cm]
      \draw[grid] (0,0) grid (8,11);
      \draw[boundary] (0,0) rectangle (8,11);
      \draw[pathA]
      (0,0) -- (0,2) -- (1,2) -- (1,4) -- (3,4)
      -- (3,6) -- (4,6) -- (4,8) -- (6,8) -- (6,10)
      -- (7,10) -- (7,11) -- (8,11);

      \draw[highlight] (0.05,1) rectangle (2,2);

      \draw[pathB]
      (-1,0)--(0,0) -- (0,1) -- (2,1) -- (2,3) -- (3,3) -- (3,4) -- (5,4) -- (5,7) -- (6,7) -- (6,10) -- (7,10)
      -- (7,11) -- (8,11);

      \node at (4,-1){$(2,3)$};
    \end{scope}

    \begin{scope}[xshift=31cm, yshift=5cm]
      \draw[->, thick] (0,0) -- (1.2,0) node[midway,above] {$\phi$};
      \draw[<-, thick] (0,-0.4) -- (1.2,-0.4) node[midway,below] {$\phi^{-1}$};
    \end{scope}

    \begin{scope}[xshift=33cm]
      \draw[grid] (0,0) grid (8,11);
      \draw[boundary] (0,0) rectangle (8,11);
      \draw[pathA]
      (0,0) -- (0,2) -- (1,2) -- (1,4) -- (3,4)
      -- (3,6) -- (4,6) -- (4,8) -- (6,8) -- (6,10)
      -- (7,10) -- (7,11) -- (8,11);


      \draw[pathB]
      (-1,0)--(0,0) -- (0,2) -- (2,2) -- (2,3) -- (3,3) -- (3,4) -- (5,4) -- (5,7) -- (6,7) -- (6,10) -- (7,10)
      -- (7,11) -- (8,11);
      \node at (4,-1){$(1,4)$};
    \end{scope}
  \end{tikzpicture}
  \caption{An example of $\Phi$, where the pair $(\run, \return)$ is indicated under each picture.}
  \label{fig:Phi}
\end{figure}

\begin{corollary}\label{coro:hankel}
  Let $B$ be a boundary path, and suppose that the $k$-th $N$-step is the first $N$-step immediately followed by an $N$-step. Then, for $P \in \CP_B$ with $\run(P) > k$, we have
  \begin{equation}\label{eq:phi1}
    (\run, \return)(P) = (\run, \return)(\phi(P)) + (1, -1).
  \end{equation}
  Consequently, for all $a, b > 0$ with $b - a \ge k$, we have
  \begin{equation}\label{eq:phi2}
    |\CP_B(b,k)| = |\CP_B(b-a,k+a)|.
  \end{equation}
\end{corollary}
\begin{proof}
  From Lemma~\ref{lem:phi}, for $P \in \CP_B$, only its first $k$ floors may contribute to both $\return(P)$ and $\run(P)$, thus $\rr(P) \le k$.
  Then Lemma~\ref{lem:phi2} yields \Cref{eq:phi1}, which implies \Cref{eq:phi2}.
\end{proof}

In a path, an $N$-step immediately followed by an $E$-step is called a \tdef{peak}. 

\begin{corollary}\label{coro:hankel2}
  Let $B$ be a boundary path with $n - 1$ peaks.
  Let $M = (m_{k, \ell})_{1 \leq k, \ell \leq n}$ with $m_{k, \ell} = |\CP_B(k, \ell)|$.
  If the first $N$-step of $B$ is immediately followed by an $N$-step, then $M$ is a Hankel matrix, \textit{i.e.}, all entries on the same skew diagonal are equal.
\end{corollary}
\begin{proof}
  Since $B$ has $n - 1$ peaks, it has $n$ floors. Hence $\return(P) \leq n$, as each contributing floor of $P$ overlaps a distinct floor of $B$.
  Our claim then follows from Corollary~\ref{coro:hankel}.
\end{proof}

\subsection{The symmetry of $(\run,\breturn)$ on $\P_B$}\label{subsec3.2}

Given a path $P \in \P_B$, we define $\rbr(P)$ as the number of floors of $P$ counted in both $\run(P)$ and $\breturn(P)$. We have the following analogue of Lemma~\ref{lem:phi} for $\rbr$.

\begin{lemma}\label{lem:theta}
  For $P \in \P_B$ bounded by $B$, the floors counted by $\rbr(P)$ are initial and consecutive.
\end{lemma}
\begin{proof}
  Such floors are initial due to the definition of $\run$, and consecutive due to that of $\breturn$.
\end{proof}

For a boundary path $B$, we define
\[
  \widehat{\P}_B^r(k, \ell) := \{ P \in \P_B : \run(P) = k,\ \breturn(P) = \ell,\ \rbr(P) = r \}.
\]
We now define the analogue of \Cref{constr:lifting,constr:descending} for $(\run, \breturn)$, with the differences highlighted.

\begin{construction} \label{constr:lifting-theta}
  For a boundary $B$ and positive integers $r, \ell, k$ with $r \leq \ell$ and $r < k$, we define a map $\theta$ from \alert{$\widehat{\P}_B^r(k, \ell)$} to \alert{$\widehat{\P}_B^r(k - 1, \ell + 1)$} as follows.
  For $P \in \widehat{\P}_B^r(k, \ell)$, we do the following:
  \begin{enumerate}
  \item Set the run floor of $P$ as the current floor;
  \item Lift the current floor up until hitting $B$ at an $E$-step or up to one unit below the next floor;
  \item If the lifted floor \alert{becomes a block return floor}, terminate; otherwise, set the next floor as the current floor and go to step 2.
  \end{enumerate}
  We set $\theta(P)$ to be the path thus obtained after the procedure.
\end{construction}

\begin{construction} \label{constr:descending-theta}
  For a boundary $B$ and positive integers $r, \ell, k$ with $r < \ell$ and $r \leq k$, we define a map $\overline{\theta}$ from \alert{$\widehat{\P}_B^r(k, \ell)$} to \alert{$\widehat{\P}_B^r(k + 1, \ell - 1)$} as follows.
  For $P' \in \widehat{\P}_B^r(k, \ell)$, we do the following:
  \begin{enumerate}
  \item Set the current floor to be the lowest \alert{block return floor} above the run floor of $P'$;
  \item Lower the current floor down to one unit above the preceding floor;
  \item If the preceding floor is the run floor of $P'$, terminate; otherwise, set the preceding floor as the current floor and go to step 2.
  \end{enumerate}
  We set $\overline{\theta}(P')$ to be the path thus obtained after the procedure.
\end{construction}

\begin{lemma}\label{lem:theta2}
  For a boundary $B$ and positive integers $r, \ell, k$ such that $r \leq \ell$ and $r < k$, the map $\theta$ is a composition-type-preserving bijection from $\widehat{\P}_B^r(k, \ell)$ to $\widehat{\P}_B^r(k - 1, \ell + 1)$, with $\overline{\theta}$ its inverse.
\end{lemma}
\begin{proof}
  We use similar arguments as in the proof of \Cref{lem:phi2}, except that we do not need to handle the boundary constraint.
  To show that $\theta$ in \Cref{constr:lifting-theta} is well-defined, we use \Cref{lem:theta} instead of \Cref{lem:phi} to show that $\run(\theta(P)) = \run(P) - 1$, and the termination condition of \Cref{constr:lifting-theta} ensures that $\breturn(\theta(P)) = \breturn(P) + 1$.
  To show that $\overline{\theta}$ in \Cref{constr:descending-theta} is well-defined, \Cref{lem:theta} is used instead of \Cref{lem:phi} to ensure that the first step is feasible, and \Cref{constr:descending-theta} along with the definitions of $\run$ and $\breturn$ ensures that the statistics are modified accordingly.
  The bijectivity of $\theta$ holds exactly for the same reason as for $\phi$ in the proof of \Cref{lem:phi2}.
\end{proof}

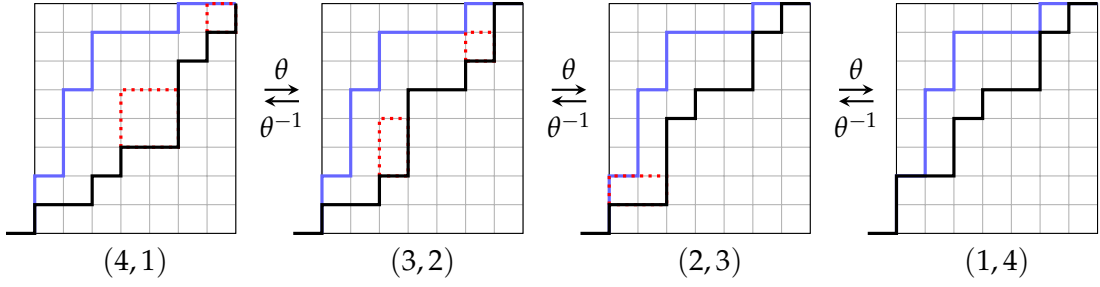
\begin{figure}
  \centering
  \begin{tikzpicture}[
    scale=0.38,
    >=stealth,
    grid/.style={gray!60, thin},
    boundary/.style={},
    pathA/.style={blue!60, very thick},
    pathB/.style={very thick},
    highlight/.style={red, dotted, very thick}
    ]

    \begin{scope}
      \draw[grid] (0,0) grid (7,8);
      \draw[boundary] (0,0) rectangle (7,8);
      \draw[pathA]
      (0,0) -- (0,2) -- (1,2) -- (1,5) -- (2,5)
      -- (2,7) -- (5,7) -- (5,8) -- (7,8) ;

      \draw[highlight] (3,3) rectangle (5,5);
      \draw[highlight] (6,7) rectangle (7,8);

      \draw[pathB]
      (-1,0)--(0,0) -- (0,1) -- (2,1) -- (2,2) -- (3,2) -- (3,3) -- (5,3) -- (5,6) -- (6,6) -- (6,7) -- (7,7)-- (7,8);

      \node at (3.5,-1){$(4,1)$};
    \end{scope}

    \begin{scope}[xshift=8cm, yshift=5cm]
      \draw[->, thick] (0,0) -- (1.2,0) node[midway,above] {$\theta$};
      \draw[<-, thick] (0,-0.4) -- (1.2,-0.4) node[midway,below] {$\theta^{-1}$};
    \end{scope}

    \begin{scope}[xshift=10cm]
      \draw[grid] (0,0) grid (7,8);
      \draw[boundary] (0,0) rectangle (7,8);
      \draw[pathA]
      (0,0) -- (0,2) -- (1,2) -- (1,5) -- (2,5)
      -- (2,7) -- (5,7) -- (5,8) -- (7,8) ;

      \draw[highlight] (2,2) rectangle (3,4);
      \draw[highlight] (5,6) rectangle (6,7);

      \draw[pathB]
      (-1,0)--(0,0) -- (0,1) -- (2,1) -- (2,2) -- (3,2) -- (3,5) -- (5,5) -- (5,6) -- (6,6) -- (6,8) -- (7,8);

      \node at (3.5,-1){$(3,2)$};
    \end{scope}

    \begin{scope}[xshift=18cm, yshift=5cm]
      \draw[->, thick] (0,0) -- (1.2,0) node[midway,above] {$\theta$};
      \draw[<-, thick] (0,-0.4) -- (1.2,-0.4) node[midway,below] {$\theta^{-1}$};
    \end{scope}

    \begin{scope}[xshift=20cm]
      \draw[grid] (0,0) grid (7,8);
      \draw[boundary] (0,0) rectangle (7,8);
      \draw[pathA]
      (0,0) -- (0,2) -- (1,2) -- (1,5) -- (2,5)
      -- (2,7) -- (5,7) -- (5,8) -- (7,8) ;

      \draw[highlight] (0,1) rectangle (2,2);

      \draw[pathB]
      (-1,0)--(0,0) -- (0,1) -- (2,1) -- (2,4) -- (3,4) -- (3,5) -- (5,5) -- (5,7) -- (6,7) -- (6,8) -- (7,8);

      \node at (3.5,-1){$(2,3)$};
    \end{scope}

    \begin{scope}[xshift=28cm, yshift=5cm]
      \draw[->, thick] (0,0) -- (1.2,0) node[midway,above] {$\theta$};
      \draw[<-, thick] (0,-0.4) -- (1.2,-0.4) node[midway,below] {$\theta^{-1}$};
    \end{scope}

    \begin{scope}[xshift=30cm]
      \draw[grid] (0,0) grid (7,8);
      \draw[boundary] (0,0) rectangle (7,8);
      \draw[pathA]
      (0,0) -- (0,2) -- (1,2) -- (1,5) -- (2,5)
      -- (2,7) -- (5,7) -- (5,8) -- (7,8) ;


      \draw[pathB]
      (-1,0)--(0,0) -- (0,2) -- (2,2) -- (2,4) -- (3,4) -- (3,5) -- (5,5) -- (5,7) -- (6,7) -- (6,8) -- (7,8);

      \node at (3.5,-1){$(1,4)$};
    \end{scope}
  \end{tikzpicture}
  \caption{An example of $\Theta$ with $(\run, \breturn)$ indicated under each path.}
  \label{fig:Theta}
\end{figure}

\begin{theorem}\label{thm:Theta}
  For every boundary path $B$, we have the symmetry
  \begin{equation}
    \sum_{P \in \P_B} q^{\run(P)} t^{\breturn(P)} = \sum_{P \in \P_B} q^{\breturn(P)} t^{\run(P)}.
  \end{equation}
\end{theorem}
\begin{proof}
  For any $\alpha \in \widehat{\P}_B^r(k, \ell)$, we set $\Theta(\alpha) := \theta^{k-\ell}(\alpha)$; see \Cref{fig:Theta} for an example.
  It follows from \Cref{lem:theta2} that $\Theta$ is an involution exchanging $\run$ and $\breturn$, which implies our claim.
\end{proof}

\subsection{The symmetry of $(\brun,\return)$ on $\CP_B$}\label{subsec3.3}

The results of this section can be seen as a special case of those in \Cref{subsec3.4}, but we still lay out full proofs as they prepare the reader for the far more technical \Cref{subsec3.4}.

Unlike the classical $\run$, the boundary run $\brun$ takes into account floors of the boundary. 
For a path $P \in \P_B$ bounded by $B$, the \tdef{boundary run floor} of $B$ relative to $P$ is the $\brun(P)$-th floor of $B$.
By the definition of $\brun$, $P$ has a floor at the same height as the boundary run floor, which is called the \tdef{bottom run floor} of $P$. 
All floors of $B$ up to and including the boundary run floor are counted by $\brun(P)$.

We may consider $\return(P)$ as counting floors of $B$ that share steps with $P$.
Similar to $\rr$ and $\rbr$, the \tdef{boundary run-return} of $P$, denoted by $\brr(P)$, counts the number of floors of $B$ counted in both $\brun(P)$ and $\return(P)$.
As an example, for the path $P$ bounded by $B$ on the left of \Cref{fig:psi1}, the third floor of $B$ is the boundary run floor relative to $P$, the third floor of $P$ is its bottom run floor, and $\brr(P) = 2$.

\begin{lemma}\label{lem:psi}
  For $P \in \CP_B$, the floors of $B$ counted by $\brr(P)$ are initial and consecutive.
  Furthermore, if $B$ has $k$ floors, then for every $r \le s \in [k - 1]$ there exists $P' \in \CP_B$ with $(\brun, \return, \brr)(P') = (s, r, r)$. 
\end{lemma}
\begin{proof}
  For the first claim, let $0 < i < \brun(P)$ and $G_i$ the $i$-th floor of $B$.
  By the definition of $\brun$, $G_i$ cannot be counted only in $\return(P)$ but not in $\brun(P)$.
  Suppose that $G_i$ is not counted in $\return(P)$, which means the projection of the $i$-th floor of $B$ on the $x$-axis does not intersect that of the floor of $P$ with height $\fh_i(B)$ for more than a point.
  As $P \in \CP_B$, by \Cref{eq:boundary-constrain}, the projection of the $(i + 1)$-st floor of $B$ on the $x$-axis does not overlap that of the floor in $P$ at height $\fh_{i+1}(B)$ for more than a point.
  Hence, $G_i$ is not counted in $\return(P)$ implies the same for $G_{i + 1}$, and by induction we have the first claim.

  For the second claim, we construct $P'$ as follows.
  We take the initial segment of $B$ up to just after its $r$-th floor.
  For each $0 < i \le s - r$, we lower the $(k - (s - r) + i)$-th floor of $B$ to the height of the $(r + i)$-th floor of $B$ to get a new floor.
  Since these floors are consecutive, we link them into a segment using $N$-steps.
  When $s = r$, we set this new segment to be at the endpoint of $B$.
  Connect the initial segment and the new one with a subpath of the form $E^p N^q$, add a subpath of the form $N^\ell$ to reach the endpoint of $B$, and we get $P'$.
  We can check that $P'$ is bounded by $B$, and it satisfies $(\brun, \return, \brr)(P') = (s, r, r)$.
  See \Cref{fig:lem-psi} for an example of constructing a bounded path $P'$ with $(\brun, \return, \brr)(P') = (6, 3, 3)$.
\end{proof}

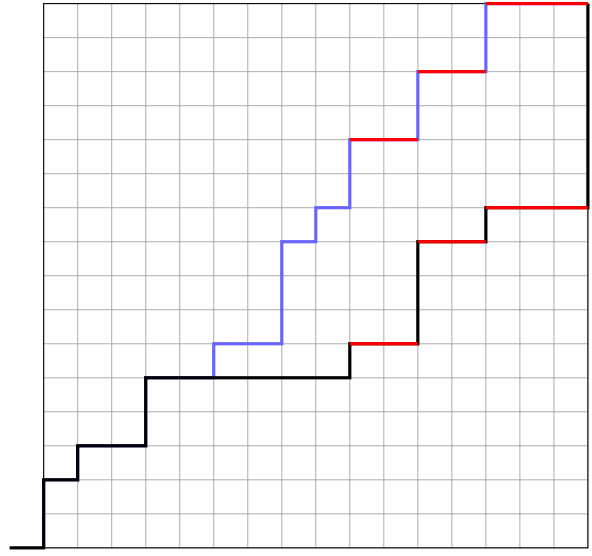
\begin{figure}
  \centering
  \begin{tikzpicture}[
    scale=0.45,
    >=stealth,
    grid/.style={gray!60, thin},
    boundary/.style={},
    pathA/.style={blue!60, very thick},
    pathB/.style={very thick},
    highlight/.style={red, dotted, very thick}
    ]

    \begin{scope}
      \draw[grid] (0,0) grid (16,16);
      \draw[boundary] (0,0) rectangle (16,16);
      \draw[pathA]
      (0,0) -- (0,2) -- (1,2) -- (1,3) -- (3,3)
      -- (3,5) -- (5,5) -- (5 ,6) -- (7, 6) -- (7, 9)
      -- (8, 9) -- (8,10) -- (9,10) -- (9,12) -- (11,12) -- (11,14) -- (13,14) -- (13, 16) -- (16, 16);

      \draw[pathB]
      (-1,0)--(0,0) -- (0,2) -- (1,2) -- (1,3) -- (3,3)
      -- (3,5) -- (9,5) -- (9,6) -- (11,6) -- (11,9) -- (13,9)
      -- (13,10) -- (16,10) -- (16,16);
      \draw[red,very thick] (13, 16) -- (16, 16);
      \draw[red,very thick] (11,14) -- (13,14);
      \draw[red,very thick] (9,12) -- (11,12);
      \draw[red,very thick] (13, 10) -- (16, 10);
      \draw[red,very thick] (11,9) -- (13,9);
      \draw[red,very thick] (9,6) -- (11,6);
    \end{scope}

  \end{tikzpicture}
  \caption{An example of constructing a bounded path whose $(\brun, \return, \brr)$ equals $(6, 3, 3)$.}
  \label{fig:lem-psi}
\end{figure}

For any boundary path $B$, define
\begin{align*}
  &\overline{\CP}_B(k, \ell) := \{ P \in \CP_B : \brun(P) = k,\ \return(P) = \ell \},\\
  &\overline{\CP}_B^r(k, \ell) := \{ P \in \CP_B : \brun(P) = k,\ \return(P) = \ell,\ \brr(P) = r \}.
\end{align*}

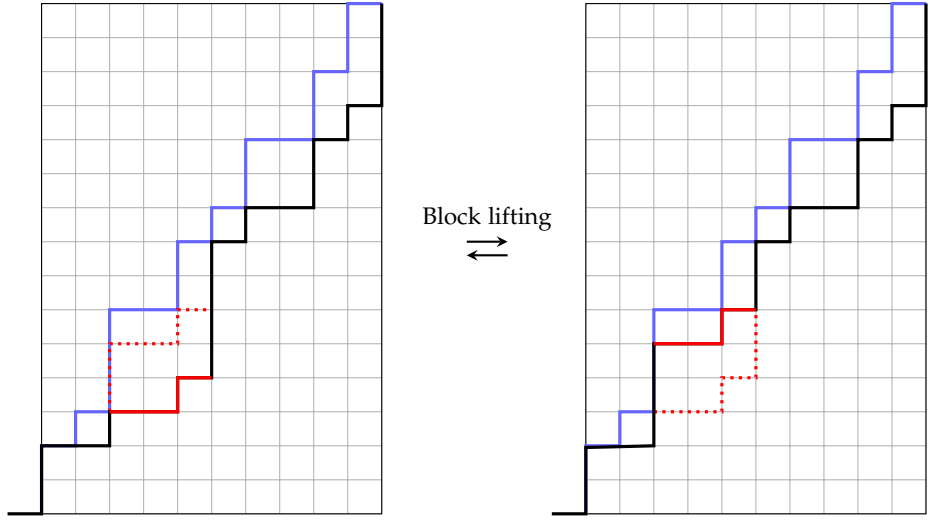
\begin{figure}
  \centering
  \begin{tikzpicture}[
    scale=0.45,
    >=stealth,
    grid/.style={gray!60, thin},
    boundary/.style={},
    pathA/.style={blue!60, very thick},
    pathB/.style={very thick},
    highlight/.style={red, dotted, very thick}
    ]

    \begin{scope}
      \draw[grid] (0,0) grid (10,15);
      \draw[boundary] (0,0) rectangle (10,15);
      \draw[pathA]
      (0,0) -- (0,2) -- (1,2) -- (1,3) -- (2,3)
      -- (2,6) -- (4,6) -- (4,8) -- (5,8) -- (5,9)
      -- (6,9) -- (6,11) -- (8,11) -- (8,13) -- (9,13) -- (9,15) -- (10,15);

      \draw[highlight] (2,3) -- (2,5) -- (4,5) -- (4,6)--(5,6);

      \draw[pathB]
      (-1,0)--(0,0) -- (0,2) -- (2,2) -- (2,3) -- (4,3) -- (4,4) -- (5,4) -- (5,8) -- (6,8) -- (6,9) -- (8,9)
      -- (8,11) -- (9,11) -- (9,12) -- (10,12) -- (10,15);
      \draw[red,very thick] (2,3)--(4,3)--(4,4)--(5,4);
    \end{scope}

    \begin{scope}[xshift=12.5cm, yshift=8cm]
      \draw[->, thick] (0,0) -- (1.2,0) node[midway,above] {\footnotesize Block lifting};
      \draw[<-, thick] (0,-0.4) -- (1.2,-0.4) node[midway,below] {};
    \end{scope}

    \begin{scope}[xshift=16cm]
      \draw[grid] (0,0) grid (10,15);
      \draw[boundary] (0,0) rectangle (10,15);
      \draw[pathA]
      (0,0) -- (0,2) -- (1,2) -- (1,3) -- (2,3)
      -- (2,6) -- (4,6) -- (4,8) -- (5,8) -- (5,9)
      -- (6,9) -- (6,11) -- (8,11) -- (8,13) -- (9,13) -- (9,15) -- (10,15);

      \draw[highlight] (2,3) -- (4,3) -- (4,4) -- (5,4)--(5,6);

      \draw[pathB]
      (-1,0)--(0,0) -- (0,1.95) -- (2,2) -- (2,3) -- (2,5) -- (4,5) -- (4,6)--(5,6) -- (5,8) -- (6,8) -- (6,9) -- (8,9)
      -- (8,11) -- (9,11) -- (9,12) -- (10,12) -- (10,15);
      \draw[red,very thick] (2,5)--(4,5)--(4,6)--(5,6);
    \end{scope}

  \end{tikzpicture}
  \caption{An example of the block lifting of the bijection $\psi$.}
  \label{fig:psi1}
\end{figure}

\begin{construction} \label{constr:lifting-psi}
  For a boundary $B$ and positive integers $r, \ell, k$ with $r \leq \ell$ and $r < k$, we define a map $\psi$ from $\overline{\CP}_B^r(k, \ell)$ to $\overline{\CP}_B^r(k - 1, \ell + 1)$ as follows.
  For $P \in \overline{\CP}_B^r(k, \ell)$, its \tdef{top run floor} is the highest floor below height $\fh_{\brun(P)+1}(B)$, whose index in $P$ is denoted by $\truf(P)$.
  We then do the following:
  \begin{enumerate}
  \item \textbf{Block lifting}:
    Lift the floors of height between $\fh_{\brun(P)}(B)$ and $\fh_{\brun(P) + 1}(B)$ by $\fh_{\brun(P) + 1}(B) - \fh_{\truf(P)}(P)$, after which the $\truf(P)$-th floor is lifted to the height $\fh_{\brun(P) + 1}(B)$.
    If the lifted top run floor hits the boundary on an $E$-step, terminate; otherwise, go to Ordinary lifting.
  \item \textbf{Ordinary lifting}:   
    \begin{enumerate}
    \item Set the lifted top run floor as the current floor;
    \item Lift the current floor up until hitting $B$ at an $E$-step or up to one unit below the next floor;
    \item If the lifted floor hits $B$ at an $E$-step, terminate; otherwise, set the next floor as the current floor and go to the previous step.
    \end{enumerate}
  \end{enumerate}
  We set $\psi(P)$ to be the path thus obtained after the procedure. See \Cref{fig:psi1,fig:psi2} for an example of the application of $\psi$.
\end{construction}

\begin{figure}
  \centering
  \begin{tikzpicture}[
    scale=0.45,
    >=stealth,
    grid/.style={gray!60, thin},
    boundary/.style={},
    pathA/.style={blue!60, very thick},
    pathB/.style={very thick},
    highlight/.style={red, dotted, very thick}
    ]

    \begin{scope}
      \draw[grid] (0,0) grid (10,15);
      \draw[boundary] (0,0) rectangle (10,15);
      \draw[pathA]
      (0,0) -- (0,2) -- (1,2) -- (1,3) -- (2,3)
      -- (2,6) -- (4,6) -- (4,8) -- (5,8) -- (5,9)
      -- (6,9) -- (6,11) -- (8,11) -- (8,13) -- (9,13) -- (9,15) -- (10,15);

      \draw[highlight] (4,6) -- (4,7) -- (5,7);
      \draw[highlight] (6,9) -- (6,10) -- (8,10);
      \draw[highlight] (9,12) -- (9,15) -- (10,15);

      \draw[pathB]
      (-1,0)--(0,0) -- (0,2) -- (2,2) -- (2,3) -- (2,5) -- (4,5) -- (4,6)--(5,6) -- (5,8) -- (6,8) -- (6,9) -- (8,9)
      -- (8,11) -- (9,11) -- (9,12) -- (10,12) -- (10,15);
    \end{scope}

    \begin{scope}[xshift=12.5cm, yshift=8cm]
      \draw[->, thick] (0,0) -- (1.2,0) node[midway,above] {\footnotesize Ordinary lifting};
      \draw[<-, thick] (0,-0.4) -- (1.2,-0.4) node[midway,below] {};
    \end{scope}

    \begin{scope}[xshift=16cm]
      \draw[grid] (0,0) grid (10,15);
      \draw[boundary] (0,0) rectangle (10,15);
      \draw[pathA]
      (0,0) -- (0,2) -- (1,2) -- (1,3) -- (2,3)
      -- (2,6) -- (4,6) -- (4,8) -- (5,8) -- (5,9)
      -- (6,9) -- (6,11) -- (8,11) -- (8,13) -- (9,13) -- (9,15) -- (10,15);

      \draw[highlight] (5,7) -- (5,6) -- (4,6);
      \draw[highlight] (6,9) -- (8,9) -- (8,10);
      \draw[highlight] (9,12) -- (10,12) -- (10,15);
      \draw[pathB]
      (-1,0)--(0,0) -- (0,2) -- (2,2) -- (2,3) -- (2,5) -- (4,5) -- (4,7)--(5,7) -- (5,8) -- (6,8) -- (6,10) -- (8,10)
      -- (8,11) -- (9,11) -- (9,15) -- (10,15);
    \end{scope}

  \end{tikzpicture}
  \caption{An example of the ordinary lifting of the bijection $\psi$.}
  \label{fig:psi2}
\end{figure}
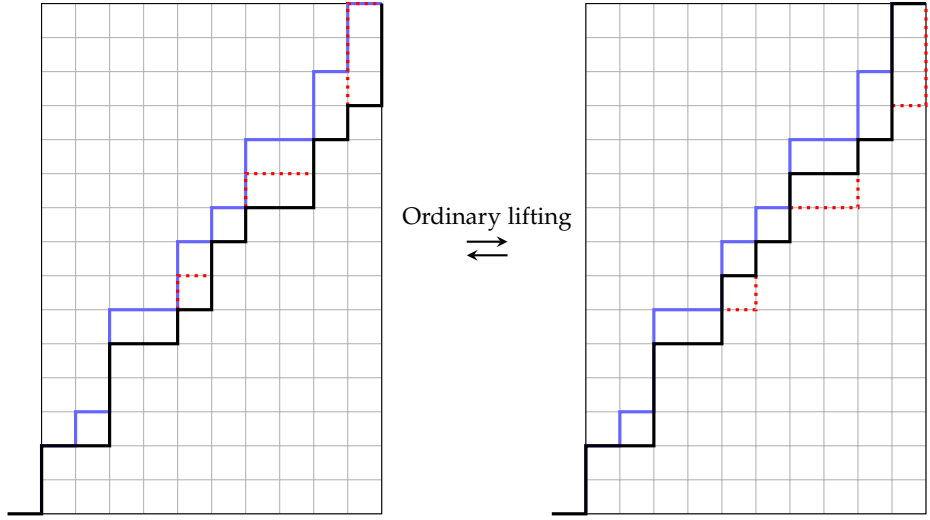

\begin{lemma} \label{lem:psi-defined}
  For a boundary $B$ and positive integers $r, k, \ell$ such that $r \leq \ell$ and $r < k$, the map $\psi$ is always well-defined on $\overline{\CP}_B^r(k, \ell)$, and sends it to $\overline{\CP}_B^r(k - 1, \ell + 1)$, while preserving composition-type.
\end{lemma}
\begin{proof}
  We first show that $\psi$ is well-defined.
  Let $P \in \overline{\CP}_B^r(k,\ell)$.
  The boundary run floor of $B$ relative to $P$ cannot be the last floor of $B$, since otherwise by \Cref{lem:psi}, all floors of $B$ are counted by $\return(P)$, thus $k = \ell = r$, which is impossible as $r < k$.
  Therefore, the $(\brun(P) + 1)$-st floor of $B$ and the top run floor of $P$ exist, and \Cref{constr:lifting-psi} is always well-defined.
  It is also clear that the procedure always terminates.
  Just before the block lifting step, by the definitions of $\truf$ and $\brun$, the top run floor of $P$ has height $\fh_{\truf(P)}(P)$ such that $\fh_{\brun(P)}(B) \leq \fh_{\truf(P)}(P) < \fh_{\brun(P) + 1}(B)$.
  Hence, the floors with height between $\fh_{\brun(P)}(B)$ and $\fh_{\brun(P) + 1}(B)$ are lifted at most to the height $\fh_{\brun(P) + 1}$(B), and after the block lifting step, the path remains bounded by $B$.
  Since the boundary constraint is stable under floor lifting, $\psi(P) \in \overline{\CP}_B$.

  Next we show that $\psi(P) \in \overline{\CP}_B^r(k - 1,\ell + 1)$.
  As $r < k$, \Cref{lem:psi} implies that the bottom run floor of $P$ does not share steps with the boundary $B$.
  Moreover, $P$ has no floor at height $\fh_{\brun(P)+1}(B)$, which means the bottom run floor, which is of height $\fh_{\brun(P)}(B)$, is lifted by $\fh_{\brun(P) + 1}(B) - \fh_{\truf(P)}(P) > 0$, thus $\brun(\psi(P)) = \brun(P) - 1$.
  If $\psi$ terminates at the block lifting procedure, as $P \in \overline{\CP}_B$, the top run floor of $P$ coincides with its bottom run floor, otherwise it violates the boundary constraint with the $\brun(P)$-th floor of $B$.
  The lifting of such a floor of $P$ makes a new return.
  Otherwise, we observe that the ordinary lifting is the same as \Cref{constr:lifting} except for initialization.
  With the same argument as in the proof of \Cref{lem:phi2}, a new return is also created afterwards.
  We thus have $\return(\psi(P)) = \return(P) + 1$ in both cases.
  By \Cref{lem:psi}, \Cref{constr:lifting-psi} does not change the number of floors of $P$ that are counted by both $\brun$ and $\return$, which means $\brr(\psi(P)) = \brr(P)$ and $\psi(P) \in \overline{\CP}_B^r(k - 1, \ell + 1)$.
  As only floor heights are changed, $\psi$ preserves the composition-type of $P$.
\end{proof}

We now lay out the inverse of $\psi$, with an example in \Cref{fig:psi1,fig:psi2}.

\begin{construction} \label{constr:descending-psi}
  For a boundary $B$ and positive integers $r, \ell, k$ with $r < \ell$ and $r \leq k$, we define a map $\overline{\psi}$ from $\overline{\CP}_B^r(k, \ell)$ to $\overline{\CP}_B^r(k + 1, \ell - 1)$ as follows.
  For $P' \in \overline{\CP}_B^r(k, \ell)$, let $\trf(P')$ be the index of the first floor of $P'$ above the bottom run floor that shares steps with $B$, and $\bref(P')$ the index of the first floor of $P'$ above the height $\fh_{\brun(P')+1}(B)$.
  We then do the following:
  \begin{enumerate}
  \item \textbf{Ordinary lowering}:
    If $\fh_{\trf(P')}(P') \neq \fh_{\brun(P')+2}(B)$, do the following steps:
    \begin{enumerate}
    \item Set the current floor to be the $\trf(P')$-th floor of $P'$;
    \item Lower the current floor down until reaching height $\fh_{\brun(P')+2}(B)$ or one unit above the preceding floor;
    \item If the current floor is not yet at height $\fh_{\brun(P')+2}(B)$, go to the previous step.
    \end{enumerate}
  \item \textbf{Block lowering}: Lower the floors of heights between $\fh_{\brun(P') + 1}(P')$ and $\fh_{\brun(P') + 2}(P')$ by
    \[
      \min\left\{\fh_{\bref(P')}(P'), \fh_{\brun(P') + 2}(B)\right\} - \fh_{\brun(P') + 1}(B).
    \]
    Then terminate.
  \end{enumerate}
  We set $\overline{\psi}(P')$ to be the path thus obtained after the procedure.
\end{construction}

\begin{lemma} \label{lem:psi-inv-defined}
  For a boundary $B$ and positive integers $r, k, \ell$ such that $r \leq \ell$ and $r < k$, the map $\overline{\psi}$ is always well-defined on $\overline{\CP}_B^r(k - 1, \ell + 1)$, and sends it to $\overline{\CP}_B^r(k, \ell)$, while preserving composition-type.  
\end{lemma}
\begin{proof}
  We first show that $\overline{\psi}$ is well-defined.
  Let $P' \in \overline{\CP}_B^r(k - 1, \ell + 1)$.
  Since $\ell + 1 > r$, \Cref{lem:psi} guarantees at least one floor after the bottom run floor of $P'$ that shares steps with the boundary $B$, which implies that $\trf(P')$ is well-defined.
  Moreover, the $(\brun(P') + 1)$-st floor of $B$ cannot be the last floor, which implies that the $(\brun(P') + 2)$-nd floor of $B$ and the $\bref(P')$-th floor of $P'$ exist.
  Therefore, the steps in \Cref{constr:descending-psi} are well-defined, and its termination is clear, meaning that $\overline{\psi}$ is well-defined.

  Next we show that $\overline{\psi}(P') \in \overline{\CP}_B^r(k, \ell)$.
  Let $F$ be the $\trf(P')$-th floor of $P'$, and $F'$ the floor preceding $F$.
  If $F$ is of height $\fh_{\brun(P') + 2}(B)$, we skip the ordinary lowering and go directly to the block lowering.
  By the minimality of $\trf(P')$, $F'$ cannot share steps with $B$ while being of height $\fh_{\brun(P') + 1}(B)$, which implies that if $F'$ is of height higher than $\fh_{\brun(P') + 1}(B)$, it will violate the boundary constraint with the $(\brun(P') + 2)$-nd floor of $B$.
  Hence, $F$ is the only floor in $P'$ with height between $\fh_{\brun(P') + 1}(B)$ and $\fh_{\brun(P') + 2}(B)$, thus $\bref(P') = \trf(P')$.
  Therefore, $F$ is lowered by $\fh_{\brun(P') + 2}(B) - \fh_{\brun(P') + 1}(B)$ in the block lowering step, to the height $\fh_{\brun(P') + 1}(B)$.
  If the height of $F$ is greater than $\fh_{\brun(P') + 2}(B)$, then $F$ is lowered in the ordinary lowering step in \Cref{constr:descending-psi} to the height $\fh_{\brun(P') + 2}(B)$ or one plus the height of $F'$.
  In the latter case, using the same argument as for $F$ of height $\fh_{\brun(P') + 2}(B)$, by the boundary constraint, the height difference between $F$ and $F'$ is greater than $1$.
  Therefore, in all cases $F$ is lowered by a positive amount, and hence $\return(\overline{\psi}(P')) = \return(P') - 1$.
  Moreover, the block lowering procedure always puts the $\bref(P')$-th floor of $P'$ at height $\fh_{\brun(P')+1}(B)$: either it has already been lowered to height $\fh_{\brun(P') + 2}(B)$ in the ordinary lowering step when $\fh_{\bref{P'}}(P') > \fh_{\brun(P') + 2}(B)$, or it remains untouched otherwise.
  Hence, $\brun(\overline{\psi}(P')) = \brun(P') + 1$.
  Finally, by \Cref{lem:psi}, \Cref{constr:descending-psi} does not change the number of floors of $P'$ counted by both $\brun$ and $\return$, thus $\brr(\overline{\psi}(P')) = \brr(P')$ and $\overline{\psi}(P') \in \overline{\CP}_B^r(k, \ell)$.
  It is also clear that $\overline{\psi}$ preserves the composition-type of $P'$ as only floor heights are changed.
\end{proof}

\begin{lemma}\label{lem:psi2}
  For a boundary $B$ and positive integers $r, k, \ell$ such that $r \leq \ell$ and $r < k$, the map $\psi$ is a composition-type-preserving bijection from  $\overline{\CP}_B^r(k, \ell)$ to $\overline{\CP}_B^r(k - 1, \ell + 1)$ with $\overline{\psi}$ its inverse.
\end{lemma}
\begin{proof}
  By \Cref{lem:psi-defined,lem:psi-inv-defined}, we only need to show that $\psi$ and $\overline{\psi}$ are inverse of each other.
  Let $P$ (resp. $P'$) be the bounded paths involved in \Cref{constr:lifting-psi} (resp. \Cref{constr:descending-psi}).
  We assume implicitly that they are linked either by $\psi$ or $\overline{\psi}$ in the following, depending on the direction of bijections we are dealing with.
  It is clear that the ordinary lifting is the inverse of the ordinary lowering, because they concern the starting point of the former is the floor of the same height $\fh_{\brun(P) + 1}(B) = \fh_{\brun(P') + 2}(B)$ as the ending point of the latter, and the arguments in the proof of \Cref{lem:phi2} apply.
  Hence, we only need to show that the block lifting in \Cref{constr:lifting-psi} is the inverse of the block lowering in \Cref{constr:descending-psi}.

  We first check the early termination conditions.
  In the proof of \Cref{lem:psi-defined}, the ordinary lifting is skipped if and only if the top run floor of $P$ coincides with its bottom run floor, which shares steps with $B$, and is lifted to height $\fh_{\brun(P) + 1}(B) = \fh_{\brun(P') + 2}(B)$.
  A floor of such height in $P'$ is clearly the $\trf(P')$-th floor of $P'$.
  Furthermore, it is the unique floor in $P'$ with height between $\fh_{\brun(P') + 1}(B)$ and $\fh_{\brun(P') + 2}(B)$ according to the proof of \Cref{constr:descending-psi}, which implies $\bref(P') = \trf(P')$.
  Hence, the ordinary lowering in \Cref{constr:descending-psi} for $P'$ is skipped exactly when the ordinary lifting in \Cref{constr:lifting-psi} for $P$ is skipped.

  Since $\brun(P) = \brun(P') + 1$, both the block lifting and the block lowering do not move floors in the related height range outside the same range.
  We only need to show that the distances of moving in both directions are the same.
  This distance in \Cref{constr:lifting-psi} is $\fh_{\brun(P) + 1}(B) - \fh_{\truf(P)}(P)$.
  Hence, the floor $F$ of height $\fh_{\brun(P)}(B)$ of $P$ corresponds exactly to the first floor $F'$ of $P'$ above the height $\fh_{\brun(P)}(B) = \fh_{\brun(P') + 1}(B)$, meaning that it is the $\bref(P')$-th floor of $P'$.
  However, the height of $F$ in $P'$ may change due to ordinary lifting, which happens if and only if it is also the top run floor of $P$, thus when $\fh_{\bref(P')}(P') > \fh_{\brun(P') + 2}(B)$.
  In this case, $F'$ is lowered to the height $\fh_{\brun(P') + 2}(B)$ in the ordinary lowering, then further lowered to the height $\fh_{\brun(P') + 1}(B) = \fh_{\brun(P)}(B)$ in the block lowering.
  Otherwise, we have $\fh_{\bref(P')}(P') \leq \fh_{\brun(P') + 2}(B)$, and $F'$ is not moved in ordinary lowering.
  In this case, $F'$ is also lowered back to the same height.
  Hence, the distances of moving for both block lifting and block lowering are the same for corresponding paths $P$ and $P'$ in both directions.
  We thus have the bijectivity.
\end{proof}

\begin{corollary}\label{coro:hankel3}
  Let $B$ be a boundary path with $n-1$ peaks. For any positive integers $k,\ell$ with $\ell+2\le k\le n$, we have
  \begin{equation}\label{eq:psi1}
    \big|\overline{\CP}_B(k,\ell)\big|<\big|\overline{\CP}_B(k-1,\ell+1)\big|.
  \end{equation}
\end{corollary}
\begin{proof}
  By Lemma~\ref{lem:psi}, we have $\big|\overline{\CP}^{\ell+1}_B(k-1,\ell+1)\big|\neq 0$.
  The result then follows directly from Lemma~\ref{lem:psi2}.
\end{proof}

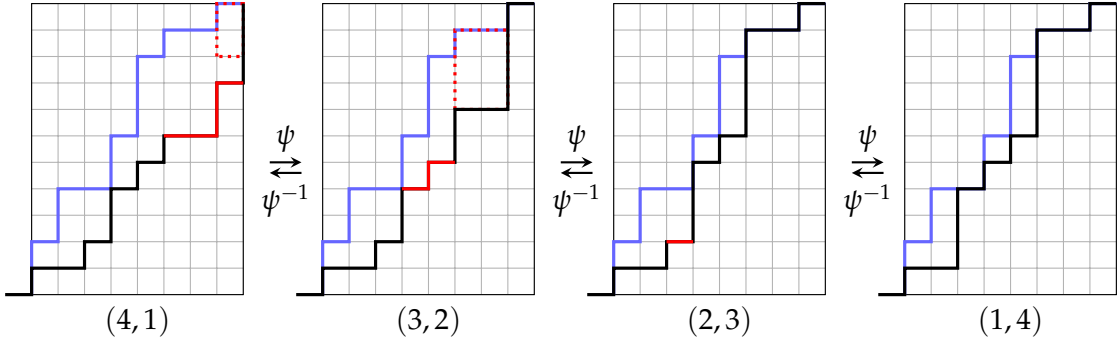
\begin{figure}
  \centering
  \begin{tikzpicture}[
    scale=0.35,
    >=stealth,
    grid/.style={gray!60, thin},
    boundary/.style={},
    pathA/.style={blue!60, very thick},
    pathB/.style={very thick},
    highlight/.style={red, dotted, very thick}
    ]

    \begin{scope}
      \draw[grid] (0,0) grid (8,11);
      \draw[boundary] (0,0) rectangle (8,11);
      \draw[pathA]
      (0,0) -- (0,2) -- (1,2) -- (1,4) -- (3,4) 
      -- (3,6) -- (4,6) -- (4,9) -- (5,9) -- (5,10)
      -- (7,10) -- (7,11) -- (8,11);


      \draw[highlight] (7,9) rectangle (8,11);

      \draw[pathB]
      (-1,0)--(0,0) -- (0,1) -- (2,1) -- (2,2) -- (3,2) -- (3,4) -- (4,4) --(4,5)-- (5,5) -- (5,6) -- (6,6)
      -- (7,6) --(7,8)-- (8,8) -- (8,11);

      \draw[red, very thick] (5,6) -- (6,6)
      -- (7,6) -- (7,8) -- (8,8);
      \node at (4,-1){$(4,1)$};
    \end{scope}

    \begin{scope}[xshift=9cm, yshift=5cm]
      \draw[->, thick] (0,0) -- (1.2,0) node[midway,above] {$\psi$};
      \draw[<-, thick] (0,-0.4) -- (1.2,-0.4) node[midway,below] {$\psi^{-1}$};
    \end{scope}

    \begin{scope}[xshift=11cm]
      \draw[grid] (0,0) grid (8,11);
      \draw[boundary] (0,0) rectangle (8,11);
      \draw[pathA]
      (0,0) -- (0,2) -- (1,2) -- (1,4) -- (3,4) 
      -- (3,6) -- (4,6) -- (4,9) -- (5,9) -- (5,10)
      -- (7,10) -- (7,11) -- (8,11);

      \draw[highlight] (5,7) rectangle (7,10);

      \draw[pathB]
      (-1,0)--(0,0) -- (0,1) -- (2,1) -- (2,2) -- (3,2) -- (3,4) --(4,4)--(4,5)-- (5,5) -- (5,7)--(7,7)--(7,11)-- (8,11) ;
      \draw[red, very thick] (3,4) --(4,4)--(4,5)-- (5,5) ;
      \node at (4,-1){$(3,2)$};
    \end{scope}

    \begin{scope}[xshift=20cm, yshift=5cm]
      \draw[->, thick] (0,0) -- (1.2,0) node[midway,above] {$\psi$};
      \draw[<-, thick] (0,-0.4) -- (1.2,-0.4) node[midway,below] {$\psi^{-1}$};
    \end{scope}

    \begin{scope}[xshift=22cm]
      \draw[grid] (0,0) grid (8,11);
      \draw[boundary] (0,0) rectangle (8,11);
      \draw[pathA]
      (0,0) -- (0,2) -- (1,2) -- (1,4) -- (3,4) 
      -- (3,6) -- (4,6) -- (4,9) -- (5,9) -- (5,10)
      -- (7,10) -- (7,11) -- (8,11);

      \draw[pathB]
      (-1,0)--(0,0) -- (0,1) -- (2,1) -- (2,2) -- (3,2) -- (3,5) --(4,5)--(4,6)-- (5,6) -- (5,10)--(7,10)--(7,11)-- (8,11) ;
      \draw[red, very thick] (2,2) --(3,2);
      \node at (4,-1){$(2,3)$};
    \end{scope}

    \begin{scope}[xshift=31cm, yshift=5cm]
      \draw[->, thick] (0,0) -- (1.2,0) node[midway,above] {$\psi$};
      \draw[<-, thick] (0,-0.4) -- (1.2,-0.4) node[midway,below] {$\psi^{-1}$};
    \end{scope}

    \begin{scope}[xshift=33cm]
      \draw[grid] (0,0) grid (8,11);
      \draw[boundary] (0,0) rectangle (8,11);
      \draw[pathA]
      (0,0) -- (0,2) -- (1,2) -- (1,4) -- (3,4) 
      -- (3,6) -- (4,6) -- (4,9) -- (5,9) -- (5,10)
      -- (7,10) -- (7,11) -- (8,11);

      \draw[pathB]
      (-1,0)--(0,0) -- (0,1) -- (2,1) -- (2,4) -- (3,4) -- (3,5) --(4,5)--(4,6)-- (5,6) -- (5,10)--(7,10)--(7,11)-- (8,11) ;

      \node at (4,-1){$(1,4)$};
    \end{scope}
  \end{tikzpicture}
  \caption{An example of $\Psi$, where the pair $(\brun, \return)$ is indicated under each figure.}
  \label{fig:Psi}
\end{figure}

\begin{theorem}\label{thm:Psi}
  Given a boundary $B$, we have the symmetry
  \begin{equation}
    \sum_{P \in \CP_B} q^{\brun(P)} t^{\return(P)} =
    \sum_{P \in \CP_B} q^{\return(P)} t^{\brun(P)}.
  \end{equation}
\end{theorem}
\begin{proof}
  For $\alpha \in \overline{\CP}_B^r(k, \ell)$, we set $\Psi(\alpha) := \psi^{k-\ell}(\alpha)$; see \Cref{fig:Psi} for an example.
  It follows from Lemma~\ref{lem:psi2} that $\Psi$ is an involution exchanging $\brun$ and $\return$, which implies our claim.
\end{proof}

\subsection{The symmetry of $(\run_\c,\return)$ on $\CP_B$}\label{subsec3.4}

Let $B$ be a boundary path and $\c = (c_1, \dots, c_m)$ a composition such that $\c \le_c \vc(B)$.
For a path $P \in \P_B$ bounded by $B$, we say that a part $c_i$ of $\c$ is counted by $\run_\c(P)$ (resp.~$\return(P)$) if $P$ has a floor at height $c_1 + \dots + c_i$ that is counted by $\run_\c(P)$ (resp.~$\return(P)$).
Let $\rr_\c(P)$ denote the number of parts of $\c$ counted by both $\run_\c$ and $\return$.
For convenience, we denote by $\psum(\c)$ the \tdef{partial sum sequence} of $\c$ with the same length as $\c$, i.e., $\psum(\mathbf{c})_j = \sum_{i = 1}^j c_i$ for $1 \leq j \leq m$.
We note that $\c \leq_c \vc(B)$ implies  $\psum(\c)\leq_c \fh(B)$.
We have the following result.

\begin{lemma}\label{lem:comp}
  For a path $P \in \CP_B$ and a composition $\c \le_c \vc(B)$, floors of $P$ with height less than $\psum(\c)_{\run_\c(P)}$ cannot be counted only by $\return(P)$, and the parts of $\c$ that are counted by both $\run_\c(P)$ and $\return(P)$ are initial and consecutive. 
\end{lemma}
\begin{proof}
  For the first claim, let $F$ be the $i$-th floor of $P$ and $F_B$ the first-cover floor of $F$.
  By \Cref{prop:bounded-index}(1), the index $j$ of $F_B$ in $B$ satisfies $j \ge i$.
  Suppose that $F$ has height less than $\psum(\c)_{\run_\c(P)}$ and is counted only by $\return(P)$.
  As $\c \le_{c} \vc(B)$, we have $\fh_j(B) \ge \psum(\c)_{j}$.
  As $F$ is counted only by $\return(P)$, $\fh_j(B) > \psum(\c)_{j}$, hence at least $j$ floors of $P$ precede $F$, contradicting $F$ being the $i$-th floor of $P$ with $i \le j$.

  For the second claim, suppose $F$ is a floor of $P$ with height strictly less than $\psum(\c)_{\run_\c(P)}$ and counted only by $\run_\c(P)$.
  The height of $F$ is thus $\psum(\c)_{k}$ for some $k < \run_{\c}(P)$.
  By iteration, we only need to show that the floor $F'$ of $P$ of height $\psum(\c)_{k + 1}$, also counted by $\run_{\c}(P)$, is not counted by $\return(P)$.
  Let $i$ be the index of $F$ in $P$.
  Since $F$ is counted by $\run_\c(P)$, there are at least $k - 1$ floors before $F$ also counted by $\run_\c(P)$, hence $i \geq k$.
  Let $F_B$ (resp. $F'_B$) be the first-cover floor of $F$ (resp. $F'$), and $j$ (resp. $j'$) its index in $B$.
  By \Cref{prop:bounded-index}(1) on $F$, we have $j \geq i \geq k$.
  \Cref{prop:bounded-index}(1) on $F'$ leads to $j' \geq k + 1$ as $F'_B \neq F_B$ implies $j' > j$.
  If $j' > k + 1$, then as $\c \le_c \vc(B)$, the height of $F'_B$ is strictly greater than $\psum(\c)_{k + 1}$, meaning that $F'$ is not a return floor.
  If $j' = k + 1$, then $j \leq j' - 1 = k$.
  But $j \geq k$, thus $j = k$.
  Since $\c \leq_c \vc(B)$, $\fh_\ell(B) > \psum(\c)_{\ell}$ implies $\fh_q(B) > \psum(\c)_{q}$ for all $q > \ell$. 
  As the height $\fh_{k}(B)$ of $F$ is strictly greater than $\psum(\c)_{k}$, we have $\fh_{k + 1}(B) > \psum(\c)_{k + 1}$, meaning that $F'_B$ is strictly higher than $F'$, hence $F'$ is not a return floor.
  We thus conclude for both cases.
\end{proof}

For any boundary path $B$ and composition $\c \le_c \vc(B)$, define
\[
  \cCP_B^r(k, \ell) := \{P \in \CP_B : \run_\c(P) = k, \return(P) = \ell, \rr_\c(P) = r\}.
\]

\begin{construction} \label{constr:lifting-psi-c}
  For a boundary $B$, a composition $\c \le_c \vc(B)$, and positive integers $r, \ell, k$ with $r \leq \ell$ and $r < k$, we define a map $\psi_\c$ from $\cCP_B^r(k, \ell)$ to $\cCP_B^r(k - 1, \ell + 1)$ as follows.
  For $P \in \cCP_B^r(k, \ell)$, its \tdef{bottom $\c$-run floor} is the floor at height $\psum(\c)_{\run_\c(P)}$, whose index in $P$ is denoted by $\bruf_\c(P)$.
  The \tdef{top block-lifted floor} of $P$ is the highest floor below height $\psum(\c)_{\run_\c(P) + 1}$, whose index in $P$ is denoted by $\truf_{\c}(P)$.
  Let $\bbruf_\c(P)$ denote the index of the first-cover floor of the bottom $\c$-run floor of $P$.
  We then do the following:
  \begin{enumerate}
  \item \textbf{Block lifting}: 
    Lift the floors of height between $\psum(\c)_{\run_\c(P)}$ and $\psum(\c)_{\run_\c(P) + 1}$  by 
    \[ 
    \min\left\{\fh_{\bbruf_\c(P)}(B) - \fh_{\bruf_\c(P)}(P), \psum(\c)_{\run_\c(P) + 1} - \fh_{\truf_\c(P)}(P)\right\}.
    \] 
    If a lifted floor hits $B$ on an $E$-step, terminate; otherwise, go to Ordinary lifting.
  \item \textbf{Ordinary lifting}:   
    \begin{enumerate}
    \item Set the lifted top block-lifted floor  as the current floor;
    \item Lift the current floor up until hitting $B$ at an $E$-step or up to one unit below the next;
    \item If the lifted floor hits $B$ at an $E$-step, terminate; otherwise, set the next floor as the current floor and go to the previous step.
    \end{enumerate}
  \end{enumerate}
  We set $\psi_\c(P)$ to be the path thus obtained after the procedure.
  We give two examples of the Block lifting step. See \Cref{fig:psic1} for the case $\fh_{\bbruf_\c(P)}(B) - \fh_{\bruf_\c(P)}(P) \le \psum(\c)_{\run_\c(P) + 1} - \fh_{\truf_\c(P)}(P)$, and \Cref{fig:psic2} for the other case.
\end{construction}

\begin{figure}
  \centering
  \begin{tikzpicture}[
    scale=0.45,
    >=stealth,
    grid/.style={gray!60, thin},
    boundary/.style={},
    pathA/.style={blue!60, very thick},
    pathB/.style={very thick},
    highlight/.style={red, dotted, very thick}
    ]

    \begin{scope}
      \draw[grid] (0,0) grid (10,15);
      \draw[boundary] (0,0) rectangle (10,15);
      \draw[pathA]
      (0,0) -- (0,2) -- (1,2) -- (1,3) -- (2,3)
      -- (2,5) -- (4,5) -- (4,7) -- (6,7) -- (6,11) 
      -- (7,11) -- (8,11) -- (8,13) -- (9,13) -- (9,15) -- (10,15);

      \draw[highlight] (5,5) -- (5,7) -- (7,7) -- (7,8) -- (9,8)--(9,6);

      \draw[pathB]
      (-1,0)--(0,0) -- (0,1) -- (1,1) -- (1,2) -- (3,2) -- (3,4) -- (5,4) -- (5,5) -- (7,5) -- (7,6)
      -- (9,6) -- (9,10) -- (10,10) -- (10,15);
      \draw[red,very thick] (5,5)--(7,5)--(7,6)--(9,6);
      \draw[green,very thick, dotted] (-0.5,0)--(0,0); 
      \draw[green,very thick, dotted] (-0.5,1)--(0,1);
      \draw[green,very thick, dotted] (-0.5,2)--(0,2);
      \draw[green,very thick, dotted] (-0.5,4)--(0,4);
      \draw[green,very thick, dotted] (-0.5,5)--(0,5);
      \draw[green,very thick, dotted] (-0.5,9)--(0,9);
      \draw[green,very thick, dotted] (-0.5,11)--(0,11);
      \draw[green,very thick, dotted] (-0.5,13)--(0,13);
      \draw[green,very thick, dotted] (-0.5,15)--(0,15);
    \end{scope}

    \begin{scope}[xshift=12.5cm, yshift=8cm]
      \draw[->, thick] (0,0) -- (1.2,0) node[midway,above] {\footnotesize Block lifting};
      \draw[<-, thick] (0,-0.4) -- (1.2,-0.4) node[midway,below] {};
    \end{scope}

    \begin{scope}[xshift=16cm]
      \draw[grid] (0,0) grid (10,15);
      \draw[boundary] (0,0) rectangle (10,15);
      \draw[pathA]
      (0,0) -- (0,2) -- (1,2) -- (1,3) -- (2,3)
      -- (2,5) -- (4,5) -- (4,7) -- (6,7) -- (6,11) 
      -- (7,11) -- (8,11) -- (8,13) -- (9,13) -- (9,15) -- (10,15);

      \draw[highlight] (5,5)--(7,5)--(7,6)--(9,6) -- (9,8);

      \draw[pathB]
      (-1,0)--(0,0) -- (0,1) -- (1,1) -- (1,2) -- (3,2) -- (3,4) -- (5,4) -- (5,7) -- (7,7) -- (7,8)
      -- (9,8) -- (9,10) -- (10,10) -- (10,15);
      \draw[red,very thick] (5,7) -- (7,7) -- (7,8)
      -- (9,8);
      \draw[green,very thick, dotted] (-0.5,0)--(0,0); 
      \draw[green,very thick, dotted] (-0.5,1)--(0,1);
      \draw[green,very thick, dotted] (-0.5,2)--(0,2);
      \draw[green,very thick, dotted] (-0.5,4)--(0,4);
      \draw[green,very thick, dotted] (-0.5,5)--(0,5);
      \draw[green,very thick, dotted] (-0.5,9)--(0,9);
      \draw[green,very thick, dotted] (-0.5,11)--(0,11);
      \draw[green,very thick, dotted] (-0.5,13)--(0,13);
      \draw[green,very thick, dotted] (-0.5,15)--(0,15);
    \end{scope}

  \end{tikzpicture}
  \caption{The first example of the block lifting of the bijection $\psi_{\c}$ with $\c = (1, 1, 1, 2, 1, 4, 2, 2, 2)$, in which the bottom $\c$-run floor touches the boundary $B$ after lifting.}
  \label{fig:psic1}
\end{figure}
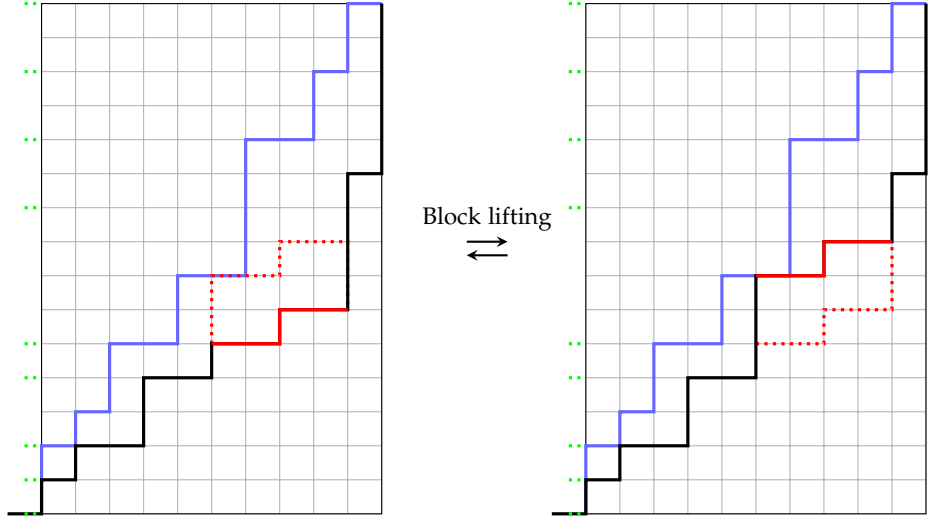

\begin{figure}
  \centering
  \begin{tikzpicture}[
    scale=0.45,
    >=stealth,
    grid/.style={gray!60, thin},
    boundary/.style={},
    pathA/.style={blue!60, very thick},
    pathB/.style={very thick},
    highlight/.style={red, dotted, very thick}
    ]

    \begin{scope}
      \draw[grid] (0,0) grid (10,15);
      \draw[boundary] (0,0) rectangle (10,15);
      \draw[pathA]
      (0,0) -- (0,2) -- (1,2) -- (1,3) -- (2,3)
      -- (2,5) -- (4,5) -- (4,7) -- (6,7) -- (6,11) 
      -- (7,11) -- (8,11) -- (8,13) -- (9,13) -- (9,15) -- (10,15);

      \draw[highlight] (5,5) -- (5,6) -- (7,6) -- (7,8) -- (8,8)--(8,9) -- (9, 9) -- (9, 8);

      \draw[pathB]
      (-1,0)--(0,0) -- (0,1) -- (1,1) -- (1,2) -- (3,2) -- (3,4) -- (5,4) -- (5,5) -- (7,5) -- (7,7)
      -- (8,7) -- (8,8) -- (9,8) -- (9,12) -- (10, 12) -- (10,15);
      \draw[red,very thick] (5,5)--(7,5)--(7,7)--(8,7) -- (8,8) -- (9,8);
      \draw[green,very thick, dotted] (-0.5,0)--(0,0); 
      \draw[green,very thick, dotted] (-0.5,1)--(0,1);
      \draw[green,very thick, dotted] (-0.5,2)--(0,2);
      \draw[green,very thick, dotted] (-0.5,4)--(0,4);
      \draw[green,very thick, dotted] (-0.5,5)--(0,5);
      \draw[green,very thick, dotted] (-0.5,9)--(0,9);
      \draw[green,very thick, dotted] (-0.5,11)--(0,11);
      \draw[green,very thick, dotted] (-0.5,13)--(0,13);
      \draw[green,very thick, dotted] (-0.5,15)--(0,15);
    \end{scope}

    \begin{scope}[xshift=12.5cm, yshift=8cm]
      \draw[->, thick] (0,0) -- (1.2,0) node[midway,above] {\footnotesize Block lifting};
      \draw[<-, thick] (0,-0.4) -- (1.2,-0.4) node[midway,below] {};
    \end{scope}

    \begin{scope}[xshift=16cm]
      \draw[grid] (0,0) grid (10,15);
      \draw[boundary] (0,0) rectangle (10,15);
      \draw[pathA]
      (0,0) -- (0,2) -- (1,2) -- (1,3) -- (2,3)
      -- (2,5) -- (4,5) -- (4,7) -- (6,7) -- (6,11) 
      -- (7,11) -- (8,11) -- (8,13) -- (9,13) -- (9,15) -- (10,15);

      \draw[highlight] (5,5)--(7,5)--(7,7)--(8,7) -- (8,8) -- (9,8) -- (9,9);

      \draw[pathB]
      (-1,0)--(0,0) -- (0,1) -- (1,1) -- (1,2) -- (3,2) -- (3,4) -- (5,4) -- (5,6) -- (7,6) -- (7,8)
      -- (8,8) -- (8,9) -- (9,9) -- (9,12) -- (10, 12) -- (10,15);
      \draw[red,very thick] (5,6) -- (7,6) -- (7,8)
      -- (8,8) -- (8,9) -- (9,9);
      \draw[green,very thick, dotted] (-0.5,0)--(0,0); 
      \draw[green,very thick, dotted] (-0.5,1)--(0,1);
      \draw[green,very thick, dotted] (-0.5,2)--(0,2);
      \draw[green,very thick, dotted] (-0.5,4)--(0,4);
      \draw[green,very thick, dotted] (-0.5,5)--(0,5);
      \draw[green,very thick, dotted] (-0.5,9)--(0,9);
      \draw[green,very thick, dotted] (-0.5,11)--(0,11);
      \draw[green,very thick, dotted] (-0.5,13)--(0,13);
      \draw[green,very thick, dotted] (-0.5,15)--(0,15);
    \end{scope}

  \end{tikzpicture}
  \caption{The second example of the block lifting of the bijection $\psi_{\c}$ with $\c = (1, 1, 1, 2, 1, 4, 2, 2, 2)$, in which the top block-lifted floor gets to height $\psum(\c)_{\run_{\c}(P) + 1}$.}
  \label{fig:psic2}
\end{figure}
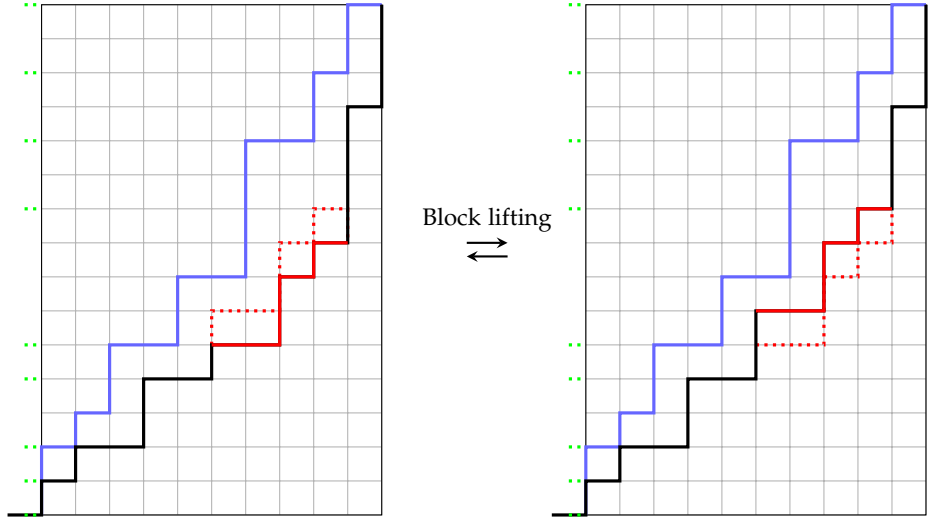

\begin{lemma}\label{lem:psi-c-defined}
  For a boundary $B$, a composition $\c\le_c\vc(B)$, and positive integers $r, \ell, k$ such that $r \leq \ell$ and $r < k$, the map $\psi_\c$ is always well-defined on $\cCP_B^r(k, \ell)$, and sends it to $\cCP_B^r(k - 1, \ell + 1)$, while preserving composition-type.
\end{lemma}
\begin{proof}
  We first show that $\psi_\c$ is well-defined.
  Let $P \in \cCP_B^r(k, \ell)$ and $F$ be the bottom $\c$-run floor of $P$.
  As $r < k$, by Lemma~\ref{lem:comp}, $F$ is counted only by $\run_\c(P)$, and the $\bbruf_\c(P)$-th floor of $B$, denoted by $F_B$, exists.
  By \Cref{prop:bounded-index}(1), we have $\bbruf_\c(P) \ge \bruf_\c(P)$, and the definition of composition run leads to $\bruf_\c(P) \ge \run_\c(P)$.
  There are two cases for the comparison between $\bbruf_{\c}(P)$ and $\run_{\c}(P)$.
  \begin{itemize}
  \item $\bbruf_\c(P) > \run_\c(P)$.
    As $\c \le_c \vc(B)$, the height of $F_B$ is at least $\psum(\c)_{\run_\c(P) + 1}$, hence $\fh_{\bbruf_\c(P)}(B) - \fh_{\bruf_\c(P)}(P) \ge \psum(\c)_{\run_\c(P) + 1} - \fh_{\truf_\c(P)}(P)$.
    The floors of heights between $\psum(\c)_{\run_\c(P)}$ and $\psum(\c)_{\run_\c(P) + 1}$ in $P$ are thus lifted by $\psum(\c)_{\run_\c(P) + 1} - \fh_{\truf_\c(P)}(P)$ in the block lifting.
    All these floors are lifted to height at most $\psum(\c)_{\run_\c(P) + 1}$, and of index at least $\run_{\c}(P)$.
    Then, for each of these floors, as $P \in \CP_B$, by \Cref{prop:bounded-index}(1), the possible floor of $B$ that it may share an $E$-step with is of index at least $\bbruf_{\c}(P)$, hence at least of height $\psum(\c)_{\run_\c(P) + 1}$.
    The path thus remains bounded by $B$.
    Thus, we either stop if the top block-lifted floor becomes a return floor, or proceed to the ordinary lifting, which always terminates, and the result remains in $\CP_B$.
    The former case happens if and only if $F_B$ is of height $\psum(\c)_{\run_{\c}(P) + 1}$ and $\truf_{\c}(P) = \bruf_{\c}(P)$, as otherwise \Cref{prop:bounded-index}(1) implies that the lifted top block-lifted floor hits a floor of $B$ strictly higher than $F_B$.
    Hence, the ordinary lifting is skipped if and only if $\fh_{\bbruf_\c(P)}(B) - \fh_{\bruf_\c(P)}(P) = \psum(\c)_{\run_\c(P) + 1} - \fh_{\truf_\c(P)}(P)$.
    We note that, by the definition of $\truf_{\c}(P)$, we have $\psum(\c)_{\run_\c(P) + 1} - \fh_{\truf_\c(P)}(P) > 0$, thus the resulting path has one less $\c$-run floor, and we have a new return floor either from the early termination or from the ordinary lifting as in \Cref{lem:phi2}.
  \item $\bbruf_\c(P) = \bruf_\c(P) = \run_\c(P)$. Again, there are two cases for the height of block lifting.
    \begin{itemize}
    \item $\fh_{\bbruf_\c(P)}(B) - \fh_{\bruf_\c(P)}(P) > \psum(\c)_{\run_\c(P) + 1} - \fh_{\truf_\c(P)}(P)$.
      In this case, the floors of heights between $\psum(\c)_{\run_c(P)}$ and $\psum(\c)_{\run_c(P) + 1}$ are lifted by $\psum(\c)_{\run_\c(P) + 1} - \fh_{\truf_\c(P)}(P)$, thus to height at most $\psum(\c)_{\run_\c(P) + 1}$.
      For the bottom $\c$-run floor, as $\fh_{\bbruf_\c(P)}(B) - \fh_{\bruf_\c(P)}(P) > \psum(\c)_{\run_\c(P) + 1} - \fh_{\truf_\c(P)}(P)$, it does not touch $F_B$ after lifting.
      For any other lifted floor $F'$, its index in $P$ is at least $\run_{\c}(P) + 1$.
      As $P \in \CP_B$, by \Cref{prop:bounded-index}, the index of the floor of $B$ that may come into contact with $F'$ is at least $\bbruf_{\c}(P) + 1$, and hence is of height strictly greater than $\psum(\c)_{\run_\c(P) + 1}$.
      This follows from $\bbruf_{\c}(P) = \run_{\c}(P)$, $\fh_{\bbruf_{\c}(P)}(B) > \psum(\c)_{\run_{\c}(P)}$, and $\c \leq_c \vc(B)$.
      Thus, $F'$ remains strictly lower than $B$ after lifting, meaning that $P$ after block lifting still remains bounded by $B$; so does ordinary lifting.
      Similar to the case $\bbruf_\c(P) > \run_\c(P)$, the procedure always terminates with one less $\c$-run floor and one more return floor, except that we always pass through the ordinary lifting.
    \item $\fh_{\bbruf_\c(P)}(B) - \fh_{\bruf_\c(P)}(P) \le \psum(\c)_{\run_\c(P) + 1} - \fh_{\truf_\c(P)}(P)$.
      In this case, the floors between $\psum(\c)_{\run_c(P)}$ and $\psum(\c)_{\run_c(P) + 1}$ are lifted by $\fh_{\bbruf_\c(P)}(B) - \fh_{\bruf_\c(P)}(P)$, making the bottom $\c$-run floor $F$ sharing an $E$-step with $F_B$ after block lifting.
      For any other lifted floor $F'$, by the same argument as the case above, the index of the first-cover floor $F'_B$ of $F'$ is of height strictly greater than $\psum(\c)_{\run_\c(P) + 1}$.
      However, the height of $F'$ after block lifting is at most
      \[
        \fh_{\truf_\c(P)}(P) + \fh_{\bbruf_\c(P)}(B) - \fh_{\bruf_\c(P)}(P) \leq \psum(\c)_{\run_\c(P) + 1},
      \]
      meaning that $F'$ remains strictly below $B$ after block lifting.
      Hence, $F$ remains bounded by $B$ after block lifting.
      The procedure clearly always terminates, and as $r < k$, the bottom $\c$-run floor is not a return floor, meaning that $\fh_{\bbruf_\c(P)}(B) - \fh_{\bruf_\c(P)}(P) > 0$.
      Thus, the block lifting makes one less $\c$-run floor, and the whole procedure makes exactly one more return floor, either the lifted bottom $\c$-run floor after the block lifting as others remains strictly below $B$, or by the termination of the ordinary lifting as in \Cref{lem:phi2}.
    \end{itemize}
  \end{itemize}
  We thus conclude that in all cases $\psi_\c(P)$ is well-defined and is in $\cCP_B^r(k - 1, \ell + 1)$. It is clear that $\psi_\c$ preserves the composition-type of $P$ as only floor heights are changed. 
\end{proof}

\begin{construction} \label{constr:descending-psi-c}
  For a boundary $B$, a composition $\c \le_c \vc(B)$, and positive integers $r, \ell, k$ with $r < \ell$ and $r \leq k$, we define a map $\overline{\psi_\c}$ from $\cCP_B^r(k, \ell)$ to $\cCP_B^r(k + 1, \ell - 1)$ as follows.
  For $P' \in \cCP_B^r(k, \ell)$, let $\trf_{\c}(P')$ be the index of the first floor of $P'$ higher than $\psum(\c)_{\run_\c(P')}$ that hits the boundary $B$ and let $\bref_{\c}(P')$ be the index of the first floor above the height $\psum(\c)_{\run_\c(P') + 1}$.
  We then do the following:
  \begin{enumerate}
  \item \textbf{Ordinary lowering}: 
    If $\fh_{\trf_\c(P')}(P') > \psum(\c)_{\run_\c(P') + 2}$, do the following steps: 
    \begin{enumerate}
    \item Set the current floor to be the $\trf_\c(P')$-th floor of $P'$;
    \item Lower the current floor down until reaching height $\psum(\c)_{\run_\c(P') + 2}$ or one unit above the preceding floor;
    \item If the current floor is not yet at height $\psum(\c)_{\run_\c(P') + 2}$, go to the previous step.
    \end{enumerate}
  \item \textbf{Block lowering}:
    Lower the floors of height between $\psum(\c)_{\run_\c(P') + 1}$ and $\psum(\c)_{\run_\c(P') + 2}$ by
    \[
      \min\left\{\fh_{\bref_\c(P')}(P'), \psum(\c)_{\run_\c(P') + 2} \right\} - \psum(\c)_{\run_\c(P') + 1}.
    \]
    Then terminate.  
  \end{enumerate}
  We set $\overline{\psi_\c}(P')$ to be the path thus obtained after the procedure.
\end{construction}

\begin{lemma}\label{lem:psi-c-inv-defined}
  For a boundary $B$, a composition $\c\le_c\vc(B)$, and positive integers $r, \ell, k$ such that $r \leq \ell$ and $r < k$, the map $\overline{\psi_\c}$ is always well-defined on $\cCP_B^r(k - 1, \ell + 1)$, and sends it to $\cCP_B^r(k, \ell)$, while preserving composition-type.
\end{lemma}
\begin{proof}
  We first show that $\overline{\psi_\c}$ is well-defined.
  Let $P' \in \cCP_B^r(k - 1, \ell + 1)$.
  As $k - 1 \leq r < \ell + 1$, by Lemma~\ref{lem:comp}, there is at least a return floor of $P'$ above the last $\c$-run floor of $P'$, thus the $\trf_\c(P')$-th floor of $P'$, denoted by $F$, exists, and is of index at least $\run_{\c}(P') + 1$.
  Let $F_B$ be the first-cover floor of $F$.
  By \Cref{prop:bounded-index}(1), the index of $F_B$ in $B$, denoted by $i_B$, is at least $\run_\c(P') + 1$.
  We have two cases.
  \begin{itemize}
  \item If $i_B > \run_\c(P') + 1$, then $|\c| \ge \run_\c(P') + 2$, and as $\c \leq_c \vc(B)$, we have $\fh_{i_B}(B) \geq \psum(\c)_{\run_\c(P') + 2}$.
  \item If $i_B = \run_\c(P') + 1$, then we have $\fh_{i_B}(B) \geq \psum(\c)_{\run_\c(P') + 1}$ as $\c \leq_c \vc(B)$.
    However, $F$ is only counted by $\return(P')$, we cannot take the equality, which implies $|\c| \ge \run_\c(P') + 2$.
  \end{itemize}
  In both cases, the height of $F$ is strictly greater than $\psum(\c)_{\run_\c(P') + 1}$, and $|\c| \ge \run_\c(P') + 2$.
  Thus, both the $\bref_\c(P')$-th floor of $P'$ and $\c_{\run_\c(P') + 2}$ exist, meaning that all steps in \Cref{constr:descending-psi-c} are well-defined.
  Since the algorithm in \Cref{constr:descending-psi-c} clearly terminates, $\overline{\psi_\c}$ is well-defined.

  Next we show that $\overline{\psi_\c}(P') \in \cCP_B^r(k, \ell)$.
  We distinguish two cases.
  \begin{itemize}
  \item $\fh_{\trf_{\c}(P')}(P') \leq \psum(\c)_{\run_\c(P') + 2}$.
    In this case, we go directly to the block lowering.
    Since $\c \leq_c \vc(B)$, we cannot have $i_B > \run_{\c}(P') + 2$, otherwise $\fh_{\trf_{\c}(P')}(P') = \fh_{i_B}(B) \geq \psum(\c)_{i_B} > \psum(\c)_{\run_\c(P') + 2}$, which is impossible.
    We thus have $i_B \leq \run_{\c}(P') + 2$.
    As the index of $F$ in $P'$ is at least $\run_{\c}(P') + 1$, by $P' \in \CP_B$ and \Cref{prop:bounded-index}(1), we have $i_B \geq \run_{\c}(P') + 1$.
    We now show that $\trf_{\c}(P') = \bref_{\c}(P')$ for the two cases for $i_B$.
    \begin{itemize}
    \item $i_B = \run_{\c}(P') + 1$.
      In this case, the floor of $P'$ preceding $F$ is the last $\c$-run floor and is of height $\psum(\c)_{\run_{\c}(P')}$, while by the analysis of the same case before, the height of $F$ is strictly greater than $\psum(\c)_{\run_{\c}(P') + 1}$, making $F$ the first floor to exceed this height.
      Hence, $\trf_{\c}(P') = \bref_{\c}(P')$.
    \item $i_B = \run_{\c}(P') + 2$.
      In this case, as $\c \le_c \vc(B)$, the height of $F$ and $F_B$ is at least $\psum(\c)_{\run_{\c}(P') + 2}$.
      Along with $\fh_{\trf_{\c}(P')}(P') \leq \psum(\c)_{\run_\c(P') + 2}$, we have the height of $F$ and $F_B$ is exactly $\psum(\c)_{\run_{\c}(P') + 2}$.
      As $\c \leq_c \vc(B)$ and $F_B$ is of index $i_B = \run_{\c}(P') + 2$, we have $\fh_{i}(B) = \psum(\c)_{i}$ for all $1 \leq i \leq i_B$.
      In particular, $\fh_{\run_{\c}(P') + 1}(B) = \psum(\c)_{\run_{\c}(P') + 1}$.
      Hence, the first floor of $P'$ of height at least $\psum(\c)_{\run_{\c}(P') + 1}$ cannot have any floor of $B$ of index strictly smaller than $i_B$ being above it.
      By \Cref{prop:bounded-index}(2), we have $\trf_{\c}(P') \le \bref_{\c}(P')$.
      We check that $F$ is of height $\psum(\c)_{\run_{\c}(P') + 2}$ strictly above $\psum(\c)_{\run_{\c}(P') + 1}$, meaning that $\trf_{\c}(P') \ge \bref_{\c}(P')$.
      Hence, $\trf_{\c}(P') = \bref_{\c}(P')$.
    \end{itemize}
    As $\trf_{\c}(P') = \bref_{\c}(P')$, the height of $F$ satisfies $\fh_{\bref_{\c}(P')}(P') = \fh_{\trf_{\c}(P')}(P') \leq \psum(\c)_{\run_\c(P') + 2}$, thus $F$ is the lowest floor lowered in block lowering.
    Moreover, as $\fh_{\trf_{\c}(P')}(P') \leq \psum(\c)_{\run_\c(P') + 2}$, $F$ is lowered by $\fh_{\trf_{\c}(P')}(P') - \psum(\c)_{\run_{\c}(P') + 1}$, reaching height $\psum(\c)_{\run_{\c}(P') + 1}$, making it a new $\c$-run floor.
    There is no other new $\c$-run floor, as all other lowered floors are lower than $\psum(\c)_{\run_\c(P') + 2}$.
    By the definition of $\bref_{\c}(P')$, the height change is strictly positive, hence $F$ ceases to be a return floor after block lowering.
    There is no other affected return floor, as by $\c \leq_c \vc(B)$, every floor of $B$ after $F_B$ is of height at least $\psum(\c)_{\run_{\c}(P') + 2}$.
  \item $\fh_{\trf_{\c}(P')}(P') > \psum(\c)_{\run_\c(P') + 2}$.
    In this case, we start with ordinary lowering, in which $F$ is either lowered to $\psum(\c)_{\run_\c(P') + 2}$, or one unit higher than the floor $F'$ of $P$ preceding it.
    In the former case, clearly the height change is strictly positive.
    In the latter case, by the minimality of $\trf_{\c}(P')$, $F'$ cannot be a return floor, meaning that it is at least $2$ units lower than $F$, and the height change is also strictly positive.
    Thus in both cases, $F$ ceases to be a return floor after lowering, and no other return floor is affected in ordinary lowering by the minimality of $\trf_{\c}(P')$.
    Now for the block lowering, let $G$ be the $\bref_{\c}(P')$-th floor of $P'$, and we consider two cases.
    \begin{itemize}
    \item $\fh_{\bref_{\c}(P')}(P') > \psum(\c)_{\run_{\c}(P') + 2}$.
      In this case, by the minimality of $\bref_{\c}(P')$, $G$ is the last floor moved in ordinary lowering, lowered to the height $\psum(\c)_{\run_{\c}(P') + 2}$ with a strictly positive height change.
      $G$ is then further lowered to height $\psum(\c)_{\run_{\c}(P') + 1}$ in block lowering, becoming a new $\c$-run floor.
      There is no other new $\c$-run floor, as $G$ occupies already the height $\psum(\c)_{\run_{\c}(P') + 2}$, while no floor above this height is lowered in the block lowering.
    \item $\fh_{\bref_{\c}(P')}(P') \leq \psum(\c)_{\run_{\c}(P') + 2}$.
      In this case, $G$ is not moved in ordinary lowering, and is then moved down by $\fh_{\bref_{\c}(P')}(P') - \psum(\c)_{\run_{\c}(P') + 1}$, which is strictly positive, and gets to height $\psum(\c)_{\run_{\c}(P') + 1}$, making it a new $\c$-run floor.
      Using the same reasoning as the previous case, there is no other new $\c$-run floor.
    \end{itemize}
  \end{itemize}
  As lowering does not change the fact of $P'$ being bounded by $B$, we see that in all cases, $\overline{\psi_{\c}}(P')$ is well-defined, and has exactly one more $\c$-run floor and one less return floor than $P'$, hence $\overline{\psi_{\c}}(P') \in \cCP^r_B(k, \ell)$ and with composition-type preserved as only floor heights are changed.
\end{proof}

\begin{lemma}\label{lem:comp2}
  For a boundary $B$, a composition $\c \le _c\vc(B)$, and positive integers $r, \ell, k$ such that $r \leq \ell$ and $r < k$, the map $\psi_\c$ is a composition-type-preserving bijection from $\cCP_B^r(k, \ell)$ to $\cCP_B^r(k - 1, \ell + 1)$ with $\overline{\psi_\c}$ its inverse.
\end{lemma}
\begin{proof}
  By \Cref{lem:psi-c-defined,lem:psi-c-inv-defined}, we only need to show that $\psi_\c$ and $\overline{\psi_\c}$ are inverse of each other.
  Similar to the proof of \Cref{lem:psi2}, we let $P$ (resp. $P'$) be the bounded paths appearing in \Cref{constr:lifting-psi-c} (resp. \Cref{constr:descending-psi-c}).
  We assume implicitly that they are linked either by $\psi_\c$ or $\overline{\psi_\c}$ in the following, depending on the direction of bijections we are dealing with.
  It is clear that the ordinary lifting is the inverse of the ordinary lowering, because when they are executed, they concern the starting point of the former is the floor of the same height $\psum(\c)_{\run_\c(P) + 1} = \psum(\c)_{\run_\c(P') + 2}$ as the ending point of the latter, and the arguments in the proof of \Cref{lem:phi2} apply.
  Hence, we only need to show that the block lifting in \Cref{constr:lifting-psi-c} is the inverse of the block lowering in \Cref{constr:descending-psi-c}.

  We first check the early termination conditions. 
  In the proof of \Cref{lem:psi-c-defined}, we see that the ordinary lifting is skipped if and only if $\fh_{\bbruf_\c(P)}(B) - \fh_{\bruf_\c(P)}(P) \le \psum(\c)_{\run_\c(P) + 1} - \fh_{\truf_\c(P)}(P)$, which implies $\fh_{\bbruf_\c(P)}(B) \le \psum(\c)_{\run_\c(P) + 1} = \psum(\c)_{\run_\c(P') + 2}$ by the definitions of $\bruf_{\c}(P)$ and $\truf_{\c}(P)$.
  Moreover, the $\bruf_\c(P)$-th floor of $P$ is lifted to height $\fh_{\bbruf_\c(P)}(B)$, and the floor of this height in $P'$ is clearly the $\trf_\c(P')$-th floor of $P'$.
  Therefore, $\fh_{\trf_\c(P')}(P') \le \psum(\c)_{\run_\c(P') + 2}$. This is the only case in which the ordinary lowering step for $P'$ in \Cref{constr:descending-psi-c} is skipped.

  Conversely, if for $P'$ we have $\fh_{\trf_\c(P')}(P') \le \psum(\c)_{\run_\c(P') + 2}$, the proof of \Cref{lem:psi-c-inv-defined} gives $\trf_\c(P') = \bref_\c(P')$.
  Floors with height between $\psum(\c)_{\run_\c(P') + 1}$ and $\psum(\c)_{\run_\c(P') + 2}$ are thus lowered by $\fh_{\bref_\c(P')}(P') - \psum(\c)_{\run_\c(P') + 1}$.
  The resulting top lifted floor is thus of height at most $\psum(\c)_{\run_\c(P') + 2} - \fh_{\bref_\c(P')}(P') + \psum(\c)_{\run_\c(P') + 1}$.
  By the definition of $\truf_{\c}(P)$, we have $\fh_{\bref_\c(P')}(P') - \psum(\c)_{\run_\c(P') + 1} \le \psum(\c)_{\run_\c(P') + 2} - \fh_{\truf_\c(P)}(P)$.
  As $\fh_{\bref_\c(P')}(P') = \fh_{\bbruf_\c(P)}(B)$, $\psum(\c)_{\run_\c(P') + 1} = \fh_{\bruf_\c(P)}(P)$ and $\run_{\c}(P) = \run_{\c}(P') + 1$, we have $\fh_{\bbruf_\c(P)}(B) - \fh_{\bruf_\c(P)}(P) \le \psum(\c)_{\run_\c(P) + 1} - \fh_{\truf_\c(P)}(P)$, which is the only case in which the ordinary lifting step for $P$ in \Cref{constr:lifting-psi-c} is skipped.
  Hence, the ordinary lifting in \Cref{constr:lifting-psi-c} for $P$ is skipped exactly when the ordinary lowering in \Cref{constr:descending-psi-c} for $P'$ is skipped.

  Since $\run_\c(P) = \run_\c(P') + 1$, both the block lifting and the block lowering do not move floors in the related height range outside the same range.
  We only need to show that the distances of moving in both directions are the same. 
  For the case that the ordinary lifting (resp. lowering) is skipped, this distance in \Cref{constr:lifting-psi-c} (resp. \Cref{constr:descending-psi-c}) is $\fh_{\bbruf_\c(P)}(B) - \fh_{\bruf_\c(P)}(P)$ (resp.~$\fh_{\bref_\c(P')}(P') - \sum_{p = 1}^{\run_\c(P') + 1} \c_p$).
  Since in this case $\fh_{\bref_\c(P')}(P') = \fh_{\bbruf_\c(P)}(B)$ and $\sum_{p = 1}^{\run_\c(P') + 1} \c_p = \fh_{\bruf_\c(P)}(P)$, the moving distance for block lifting and block lowering is the same for the corresponding paths $P$ and $P'$ in both directions.
  When the ordinary lifting is not skipped, the moving distance in \Cref{constr:lifting-psi-c} is $\psum(\c)_{\run_{\c}(P) + 1} - \fh_{\truf_{\c}(P)}(P)$.
  When the ordinary lowering is not skipped, we have $\fh_{\trf_{\c}(P')}(P') > \psum(\c)_{\run_{\c}(P') + 2}$, and there are two cases.
  \begin{itemize}
  \item $\fh_{\bref_{\c}(P')} \geq \psum(\c)_{\run_{\c}(P') + 2}$.
    In this case, the moving distance is $\c_{\run_{\c}(P') + 2}$.
    Moreover, in the ordinary lowering, the $\bref_{\c}(P')$-th floor of $P'$ is lowered to height $\psum(\c)_{\run_{\c}(P') + 2}$, thus further lowered to $\psum(\c)_{\run_{\c}(P') + 1}$ after block lowering.
    Hence, there is no other floor moved in block lowering, and the $\bref_{\c}(P')$-th floor of $P'$ becomes the top block-lifted floor of $P$, and by $\run_{\c}(P) = \run_{\c}(P') + 1$, we have the same moving distances.
    Conversely, if the moving distance of the ordinary lifting is also $\c_{\run_{\c}(P') + 2}$, then the top block-lifted floor of $P$ is at the height $\psum(\c)_{\run_{\c}(P)}$, meaning that there is no other floor of height between $\psum(\c)_{\run_{\c}(P)}$ and $\psum(\c)_{\run_{\c}(P) + 1}$, so this floor becomes the $\bref_{\c}(P')$-th floor of $P'$, with $\fh_{\bref_{\c}(P')} \geq \psum(\c)_{\run_{\c}(P') + 2}$.
    We thus also have the same moving distances in this case.
  \item $\fh_{\bref_{\c}(P')} < \psum(\c)_{\run_{\c}(P') + 2}$.
    In this case, the moving distance is $\fh_{\bref_{\c}(P')} - \psum(\c)_{\run_{\c}(P') + 1}$, and the $\bref_{\c}(P')$-th floor of $P'$ is not moved in the ordinary lowering.
    The floors of $P'$ of heights between $\psum(\c)_{\run_{\c}(P') + 1}$ and $\psum(\c)_{\run_{\c}(P') + 2}$ are thus moved together so that the bottom one touches the height $\psum(\c)_{\run_{\c}(P') + 1}$.
    As the floor of height $\psum(\c)_{\run_{\c}(P') + 2}$ becomes the top block-lifted floor in $P$, we have the same moving distances.
    The converse holds similarly.
  \end{itemize}
  As the moving distances are the same in each case, we have the bijectivity.
\end{proof}

\begin{theorem}\label{thm:Comp}
  For a boundary $B$ and composition $\c \le_c \vc(B)$, we have the symmetry
  \begin{equation}
    \sum_{P \in \CP_B} q^{\run_\c(P)} t^{\return(P)} =
    \sum_{P \in \CP_B} q^{\return(P)} t^{\run_\c(P)}.
  \end{equation}
\end{theorem}
\begin{proof}
  For any $\alpha \in \cCP_B^r(k, \ell)$, we set $\Psi_\c(\alpha) := \psi_\c^{k-\ell}(\alpha)$.
  It follows from Lemma~\ref{lem:comp2} that $\Psi_\c$ is an involution exchanging $\run_\c$ and $\return$, which implies our claim.
\end{proof}

\begin{corollary}\label{coro:equidist}
  For boundary $B$ and composition $\c \le_c \vc(B)$, we have the equidistribution
  \begin{equation}
    \sum_{P \in \CP_B} t^{\return(P)} = \sum_{P \in \CP_B} t^{\run_\c(P)}.
  \end{equation}
\end{corollary}
\begin{proof}
  Take $q = 1$ in \Cref{thm:Comp}.
\end{proof}

\begin{remark}
  We have three remarks regarding the involution $\Psi_\c$.
  \begin{enumerate}
  \item The involution $\Psi_\c$ actually unifies both $\Phi$ and $\Psi$. We choose not to introduce $\Psi_\c$ directly, because we think that the boundary run can be visualized intuitively, and that presenting the symmetry of boundary run and return first may make $\Psi_\c$ easier to understand.
  \item The constraint on the composition $\c$ for the $\run_\c$ statistic 
    can be relaxed. The original range $\c\le_c\vc(B)$ may be extended to $\c \leq_d \vc(B)$, as long as the following condition is satisfied:
    \begin{center}
  if $\psum(c)_i < \fh_i(B)$ for some $i$, then $\psum(\c)_j < \fh_j(B)$ for every $j > i$.
    \end{center}
    Here, $\leq_d$ is the \tdef{dominance order}, which holds when $\sum_{j=1}^i c_j \leq \sum_{j=1}^i \vc(B)_j$ for all $i \in [m]$, where $m$ is the shortest length of $\c$ and $\vc(B)$.
    This additional condition ensures that the properties established in the proof of \Cref{lem:comp,lem:psi-c-defined,lem:psi-c-inv-defined} continue to hold. 
  \item $\Psi_{\c}$ does not preserve a third statistic $\run_{\c'}$ for another $\c' \leq_c \vc(B)$.
    Moreover, although by \Cref{coro:equidist}, $\run_{\c}$ and $\run_{\c'}$ are always equidistributed, they are not symmetric in general.
    For instance, take $\run$ and $\brun$, which correspond to $\c = (1, \ldots, 1)$ and $\c = \vc(B)$ respectively, and the generating function $G_B(s, t) = \sum_{P \in \CP_B} s^{\run(P)} t^{\brun(P)}$ for $B = NNENNE$ is $G_B(s, t) = s^3t^2 + st^3 + 3st^2 + 4s^2t + 3st$, for which clearly $G_B(s, t) \neq G_B(t, s)$.
  \end{enumerate}
\end{remark}

\section{Applications to some families of lattice paths}\label{sec4}

In this section, we apply results in \Cref{sec3} to two families of lattice paths recast as bounded lattice paths: rational Dyck paths to answer \Cref{prob:Dai}, and $\vec{k}$-bounded paths to answer \Cref{prob:Qu}. 

\subsection{Rational Dyck paths} \label{sec:rational-dyck}

For $m, n > 0$, the \tdef{diagonal cells} of the $m \times n$ grid are squares intersecting the main diagonal.
See \Cref{fig:ration} for an example, in which they are colored blue.
The \tdef{diagonal boundary} of the grid, denoted by $B_{m,n}$, is the lower-right boundary of all diagonal cells.
We denote by $\D_{m,n} = \P_{B_{m, n}}$ the family of lattice paths 
bounded by $B_{m,n}$, also called \tdef{$(m, n)$-rational Dyck paths}.

\begin{proposition}\label{prop:ration}
  For $m \leq n$, we have $\CP_{B_{m,n}} = \P_{B_{m,n}}$.
\end{proposition}
\begin{proof}
  The path $B_{m, n}$ is a Christoffel word in $\{N, E\}$ with slope $m / n \leq 1$, hence it does not have consecutive $E$-steps; see~\cite[Theorem 6.12]{Berstel2008}.
  We then conclude by \Cref{prop:bounded}.
\end{proof}

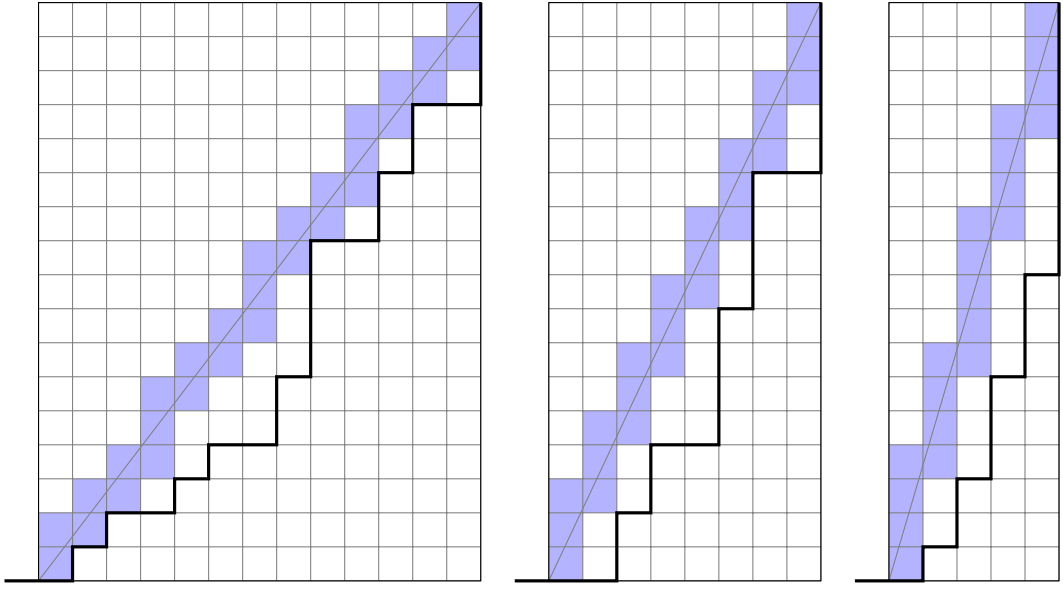
\begin{figure}
  \centering
  \begin{tikzpicture}[
    scale=0.45,
    >=stealth,
    grid/.style={gray, thin},
    boundary/.style={},
    pathA/.style={blue, very thick},
    pathB/.style={very thick},
    highlight/.style={red, dotted, very thick}
    ]

    \begin{scope}
      \fill[blue!30] 
      (0,0) rectangle (1,2)
      (1,1) rectangle (2,3)
      (2,2) rectangle (3,4)
      (3,3) rectangle (4,6)
      (4,5) rectangle (5,7)
      (5,6) rectangle (6,8)
      (6,7) rectangle (7,10)
      (7,9) rectangle (8,11)
      (8,10) rectangle (9,12)
      (9,11) rectangle (10,14)
      (10,13) rectangle (11,15)
      (11,14) rectangle (12,16)
      (12,15) rectangle (13,17)
      ;
      \draw[grid] (0,0) grid (13,17);
      \draw[grid] (0,0) -- (13,17);
      \draw[boundary] (0,0) rectangle (13,17);
      \draw[pathB]
      (-1,0)--(0,0) -- (1,0) -- (1,1) -- (2,1) -- (2,2) -- (4,2) -- (4,3)--(5,3)-- (5,4)--(7,4)--(7,6)--(8,6)--(8,10)--(10,10)--(10,12)--(11,12)--(11,14)--(13,14)--(13,17) ;
    \end{scope}

    \begin{scope}[xshift=15cm]

      \fill[blue!30] 
      (0,0) rectangle (1,3)
      (1,2) rectangle (2,5)
      (2,4) rectangle (3,7)
      (3,6) rectangle (4,9)
      (4,8) rectangle (5,11)
      (5,10) rectangle (6,13)
      (6,12) rectangle (7,15)
      (7,14) rectangle (8,17);


      \draw[grid] (0,0) grid (8,17);
      \draw[grid] (0,0) -- (8,17);
      \draw[boundary] (0,0) rectangle (8,17);
      \draw[pathB]
      (-1,0)--(0,0) -- (2,0) -- (2,2) -- (3,2) -- (3,4) -- (5,4) -- (5,8)--(6,8) -- (6,12) -- (8,12) -- (8,17) ;
    \end{scope}

    \begin{scope}[xshift=25cm]
      \fill[blue!30] 
      (0,0) rectangle (1,4)
      (1,3) rectangle (2,7)
      (2,6) rectangle (3,11)
      (3,10) rectangle (4,14)
      (4,13) rectangle (5,17)
      ;

      \draw[grid] (0,0) grid (5,17);
      \draw[grid] (0,0) -- (5,17);
      \draw[boundary] (0,0) rectangle (5,17);
      \draw[pathB]
      (-1,0)--(0,0) -- (1,0) -- (1,1) -- (2,1) -- (2,3) -- (3,3) -- (3,6)--(4,6)-- (4,9)--(5,9)--(5,17) ;
    \end{scope}
  \end{tikzpicture}
  \caption{Examples of rational Dyck paths (with size $(13, 17)$, $(8, 17)$, and $(5, 17)$).}
  \label{fig:ration}
\end{figure}

Dai, Fu, and Qiu~\cite{Dai2026} introduced the following run-like statistic on $\D_{m, n}$ for $m \le n$. 

\begin{definition}\label{def:rrun}
  For $1 \leq m \leq n$, let $k = \left\lfloor \frac{n}{m} \right\rfloor$. The \tdef{ratio-run} statistic of $P \in \D_{m, n}$, denoted by $\rrun(P)$, is defined as 
  \[
    \rrun(P) = \min \left\{i \in[n]: k i + 1 \notin \fh(B_{m,n})\right\}.
  \]
\end{definition}

For example, the pair $(\rrun, \brun)$ for each path in \Cref{fig:ration} is $(5,4)$, $(3,3)$, and $(4,3)$ respectively.

One of the main results of \cite{Dai2026} is that $(\run, \return)$ and $(\rrun, \return)$ are both symmetric over $\D_{m,n}$ for $m \le n$.
We recover and generalize this result by applying \Cref{thm:Comp} to $\D_{m, n}$.

\begin{theorem}\label{thm:ration2}
  Let $1 \leq m \le n$ and $\c \vDash n + 1$ with $\c \le_c \vc(B_{m,n})$. 
  Then there is a composition-type-preserving involution on $\D_{m,n}$ realizing the symmetry
  \begin{equation} \label{eq:rational-general}
    \sum_{P \in \D_{m,n}} q^{\run_\c(P)} t^{\return(P)} =
    \sum_{P \in \D_{m,n}} q^{\return(P)} t^{\run_\c(P)}.
  \end{equation}
  In particular, taking $\c = (1, \dots, 1)$ and $\c = \vc(B_{m,n})$, we have the symmetries
  \begin{equation}\label{eq:ration-run}
    \sum_{P \in \D_{m,n}} q^{\run(P)} t^{\return(P)}=
    \sum_{P \in \D_{m,n}} q^{\return(P)} t^{\run(P)}
  \end{equation}
  and
  \begin{equation}
    \sum_{P \in \D_{m,n}} q^{\brun(P)} t^{\return(P)}=
    \sum_{P \in \D_{m,n}} q^{\return(P)} t^{\brun(P)}.
  \end{equation}
  For $\c = \c_{m, n} = (c_1, c_2, \dots, c_\ell)$ such that
  \[
    c_1 = 1, \quad c_i = \left\lfloor \frac{n}{m} \right\rfloor \text{ for } 1 < i < \ell, \quad c_\ell = n - (\ell - 2) \left\lfloor \frac{n}{m} \right\rfloor \le \left\lfloor \frac{n}{m} \right\rfloor,
  \]
  we have the symmetry
  \begin{equation} \label{eq:rational-rrun}
    \sum_{P \in \D_{m,n}} q^{\rrun(P)} t^{\return(P)} =
    \sum_{P \in \D_{m,n}} q^{\return(P)} t^{\rrun(P)}.
  \end{equation}
\end{theorem}
\begin{proof}
  The general case is simply the application of \Cref{thm:Comp} to $\D_{m, n}$, and the first two symmetries come directly from the definitions of $\run$ and $\brun$.
  For the last symmetry, we first show that $\c_{m,n} \le_c \vc(B_{m,n})$.
  let $d$ be the shortest length of maximal segments of consecutive $N$-steps in $B_{m,n}$.
  By properties of Christoffel words (see \cite[Theorem 6.12]{Berstel2008}), $d \ge \left\lfloor \frac{n}{m} \right\rfloor$.
  Moreover, no part of $\c_{m,n}$ exceeds $\left\lfloor \frac{n}{m} \right\rfloor$, and $\c_{m,n} \vDash n + 1$.
  Hence, $\c_{m,n} \le_c \vc(B_{m,n})$.
  \Cref{eq:rational-general} thus applies to $\c_{m, n}$.
  By the definition of $\run_{\c}$, we have
  \[
    \run_{\c_{m,n}}(P) = \max \{i \in [n] : \{c_1, c_1 + c_2, \dots, c_1 + \dots + c_i\} \subseteq \fh(P) \}.
  \]
  Since $B_{m,n}$ has no floor at height $n + 1$, neither $\sum_{1\le i\le \ell}c_i$ nor $\sum_{1\le i\le \ell-1}c_i+\left\lfloor \frac{n}{m} \right\rfloor$ is in $\fh(P)$.
  Thus,
  \[
    \run_{\c_{m,n}}(P)
    = \min \left\{i \in [n] : c_1 + \dots + c_{i+1} \notin \fh(P) \right\}
    = \min \left\{i \in [n] : 1 + i \left\lfloor \frac{n}{m} \right\rfloor \notin \fh(P) \right\}.
  \]
  Comparing with \Cref{def:rrun}, $\run_{\c_{m,n}}$ coincides with $\rrun$, hence \Cref{eq:rational-general} turns into \Cref{eq:rational-rrun}.
\end{proof}

For general $m, n$, applying \Cref{thm:Theta} to $\D_{m, n}$ leads to the following answer to Problem~\ref{prob:Dai}.

\begin{theorem}\label{thm:ration1}
  For $m, n \geq 1$, there is a composition-type-preserving involution on $\D_{m,n}$ realizing the symmetry
  \begin{equation}
    \sum_{P \in \D_{m,n}} q^{\run(P)} t^{\breturn(P)} =
    \sum_{P \in \D_{m,n}} q^{\breturn(P)} t^{\run(P)}.
  \end{equation}
\end{theorem}

\subsection{$\vec{k}$-lattice paths}

There are several equivalent models for $\vec{k}$-Dyck paths~\cite{Xin2023}, and we use the following.
In this section, we allow $E$-steps to have length larger than $1$.
Given a vector $\vec{k} = (k_1, k_2, \dots, k_\ell) \vDash n$ of positive integers, a \tdef{$\vec{k}$-Dyck path} is a lattice path from $(0,0)$ to $(n, n)$ weakly below the diagonal $y = x$.
Its $E$-steps have lengths $k_1, k_2, \dots, k_\ell$ from left to right, 
while all north steps are of length $1$.
A $\vec{k}$-Dyck path can be seen as a classical Dyck path by breaking each $E$-step of length $k$ into $k$ $E$-steps of length $1$.
Classical Dyck paths are special cases of $\vec{k}$-Dyck paths with $\vec{k} = (1, \dots, 1)$.
We now give a bounded version of $\vec{k}$-Dyck paths.

\begin{figure}
  \centering
  \begin{tikzpicture}[
    scale=0.45,
    >=stealth,
    grid/.style={gray, thin},
    boundary/.style={},
    pathA/.style={blue!60, very thick},
    pathB/.style={very thick},
    highlight/.style={red, dotted, very thick}
    ]

    \begin{scope}

      \draw[grid] (0,0) grid (15,15);

      \draw[boundary] (0,0) rectangle (15,15);

      \draw[pathA]
      (0,0) -- (0,2) -- (1,2) -- (1,4) -- (3,4) 
      -- (3,5) -- (5,5) -- (5,9) -- (6,9) -- (6,10)
      -- (7,10) -- (7,12) -- (11,12)--(11,14)--(14,14)--(14,15)--(15,15);

      \draw[pathB]
      (0,0) -- (1,0) -- (1,2) -- (4,2) -- (4,4) -- (6,4) -- (6,7)--(10,7)--(10,11)--(14,11)--(14,14)--(15,14)--(15,15) ;

      \draw[->](-1,0) -- (0,0);
      \draw[->](0,0) -- (1,0);
      \draw[->](1,2) -- (3,2);
      \draw[->](3,2) -- (4,2);
      \draw[->](4,4) -- (5,4);
      \draw[->](5,4) -- (6,4);
      \draw[->](6,7) -- (10,7);
      \draw[->](10,11) -- (11,11);
      \draw[->](11,11) -- (14,11);
      \draw[->](14,14) -- (15,14);

    \end{scope}
  \end{tikzpicture}
  \caption{A $\vec{k}$-bounded path with $\vec{k}=(1,2,1,1,1,4,1,3,1)$.}
  \label{fig:vec}
\end{figure}
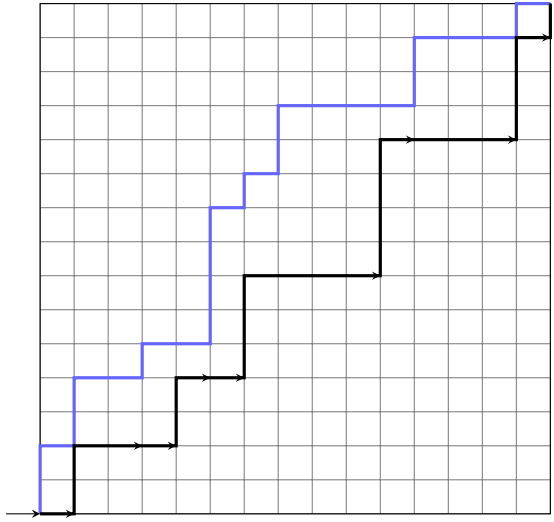

\begin{definition}
  Fix $m, n > 0$ and $\vec{k} = (k_1, \dots, k_\ell) \vDash m$ a composition of $m$ of length $\ell$. Let $B$ be a path from $(0, 0)$ to $(m, n)$. A \tdef{$\vec{k}$-bounded path} with boundary $B$ is a path from $(0, 0)$ to $(m, n)$ weakly below $B$ such that the $i$-th $E$-step is of length $k_i$ and all $N$-steps are of length $1$. See \Cref{fig:vec}. We denote by $\D_B(\vec{k})$ the set of all such paths. 
\end{definition}

Again, decomposing length $k$ $E$-steps to $k$ $E$-steps of length $1$, each $\vec{k}$-bounded path is associated with a path bounded by $B$, to which we may assign all previously defined statistics.

\begin{proposition}\label{prop:obQu}
  Suppose that $(\text{stat}_1, \text{stat}_2)$ is symmetric on paths bounded by a given boundary $B$ of every composition-type $\c$ refining $\vec{k}$. Then the same holds for $\vec{k}$-bounded paths with the same boundary $B$.
\end{proposition}
\begin{proof}
  It is because $\D_B(\vec{k})$ is the union of paths bounded by $B$ of all composition-type $\c$ such that $\vec{k}$ refines $\c$.
\end{proof}

Research on the statistics of $\vec{k}$-Dyck paths is in its beginning, with limited symmetry results.
Qu and Zhang \cite{Qu2026} hence posed Problem~\ref{prob:Qu0}.
Using Proposition~\ref{prop:obQu}, the results established in Section~\ref{sec3} answer Problem~\ref{prob:Qu}, summarized as follows.

\begin{theorem}
  For a boundary $B$ and a vector $\vec{k}$, there is an involution on $\D_B(\vec{k})$ preserving the composition-type and realizing the symmetry
  \begin{equation}
    \sum_{P \in \D_B(\vec{k})} q^{\run(P)} t^{\breturn(P)} =
    \sum_{P \in \D_B(\vec{k})} q^{\breturn(P)} t^{\run(P)}.
  \end{equation}
\end{theorem}
\begin{proof}
  This is a consequence of \Cref{prop:obQu,thm:Theta}.
\end{proof}

\begin{theorem}
  Let $B$ be a boundary and $\c$ a composition with $\c \le_c \vc(B)$. For any $\vec{k} \in \cons(\comp(B))$, there is a composition-type-preserving involution on $\D_B(\vec{k})$ that realizes the symmetry
  \begin{equation}
    \sum_{P \in \D_B(\vec{k})} q^{\run_\c(P)} t^{\return(P)} =
    \sum_{P \in \D_B(\vec{k})} q^{\return(P)} t^{\run_\c(P)}.
  \end{equation}
  In particular, for $\c = (1, \dots, 1)$ and $\c = \vc(B)$, we get the following symmetries respectively:
  \begin{align}
    \sum_{P \in \D_B(\vec{k})} q^{\run(P)} t^{\return(P)} &=
                                                            \sum_{P \in \D_B(\vec{k})} q^{\return(P)} t^{\run(P)}, \\
    \sum_{P\in\D_B(\vec{k})} q^{\brun(P)} t^{\return(P)} &=
                                                           \sum_{P\in\D_B(\vec{k})} q^{\return(P)} t^{\brun(P)}.
  \end{align}
\end{theorem}
\begin{proof}
  This is a consequence of~\Cref{prop:obQu,thm:Comp}.
\end{proof}

\section{Symmetries on vector parking functions} \label{sec5}

In this section, we present some statistics on vector parking functions transferred from bounded lattice paths, generalizing those in classical parking functions.
Some of them come from a notion of prime decomposition for 
vector parking functions, which is new to the authors' knowledge.
Then, using \Cref{thm:Theta,thm:Comp}, we obtain several symmetries of vector parking functions on these statistics.

Let $\u = (u_1, \dots, u_n)$ be a weakly increasing vector of positive integers.
A sequence $\alpha = (\alpha_1, \dots, \alpha_n)$ is called a \tdef{vector parking function constrained by $\u$}, or simply a \tdef{$\u$-parking function}, if
\[
  \forall i \in [n], \; | \{j \in [n] : \alpha_j \le u_i\} | \ge i.
\]
In other words, the non-decreasing rearrangement $\inc(\alpha) = (\beta_1, \dots, \beta_n)$ of $\alpha$ satisfies $\beta_i \le u_i$ for every $i$.
We denote by $\PF_\u$ the set of all $\u$-parking functions.
As an example, let $\u = (1, 3, 5, 5, 5, 7)$, $\u' = (1, 3, 3, 4, 7, 7)$, and $\alpha = (5, 1, 2, 7, 1, 5)$, $\alpha' = (2, 6, 1, 3, 1, 7)$. Then $\alpha \in \PF_{\u}$ but $\alpha \notin \PF_{\u'}$, while $\alpha' \notin \PF_{\u}$ but $\alpha' \in \PF_{\u'}$.

Throughout this section, unless stated otherwise, $\u$ denotes a weakly increasing vector of positive integers, and $u_i$ its $i$-th element.
Every $\u$ can be identified with a lattice path $P_{\u}$ whose height sequence is $\u$.
Each $\u$-parking function $\alpha$ is thus associated with a path $P_\alpha$ bounded by $P_{\u}$ such that the height sequence of $P_\alpha$ is $\h(P_\alpha) = \inc(\alpha)$.
This is a many-to-one map from the set of all $\u$-parking functions to $\P_B$ with $B = P_\u$; see Figure~\ref{fig:PF-path} for an example.
More precisely, we have the following lemma. 

  \begin{figure}
    \centering
    \begin{tikzpicture}[
      scale=0.45,
      >=stealth,
      grid/.style={gray!60, thin},
      boundary/.style={},
      pathA/.style={blue!60, very thick},
      pathB/.style={very thick},
      highlight/.style={red, dotted, very thick},
      description/.style={
      anchor=east,
      font=\small
    }
      ]
      \begin{scope}[xshift=0cm, yshift=6cm]
        \node[description] at (-8, 1) {$(1, 2, 4, 7, 8, 8, 9) \in \PF_u$};
        \node[description] at (-8, 0) {$(7, 2, 8, 9, 1, 8, 4) \in \PF_u$};
        \node[description] at (-8, -1) {$u = (2, 4, 7, 7, 8, 11, 13)$};

        \draw[->, thick] (-6,0) -- (-3,0) node[midway,above] {};
      \end{scope}

      \begin{scope}
        \draw[grid] (0,0) grid (7,14);
        \draw[boundary] (0,0) rectangle (7,14);
        \draw[pathA]
        (0,0) -- (0,2) -- (1,2) -- (1,4) -- (2,4)
        -- (2,7) --  (4,7) -- (4,8)
        -- (5,8) -- (5, 11) -- (6,11) -- (6,13) -- (7,13);

        \draw[pathB]
        (-1,0) -- (0,0) -- (0,1) -- (1,1) -- (1,2) -- (2,2) -- (2,4) -- (3,4) -- (3,7) -- (4,7) -- (4,8) -- (6,8) -- (6, 11) -- (7,11) -- (7, 14);
      \end{scope}

    \end{tikzpicture}
    \caption{An example of the many-to-one map from $u$-parking functions to bounded path}
    \label{fig:PF-path}
  \end{figure}
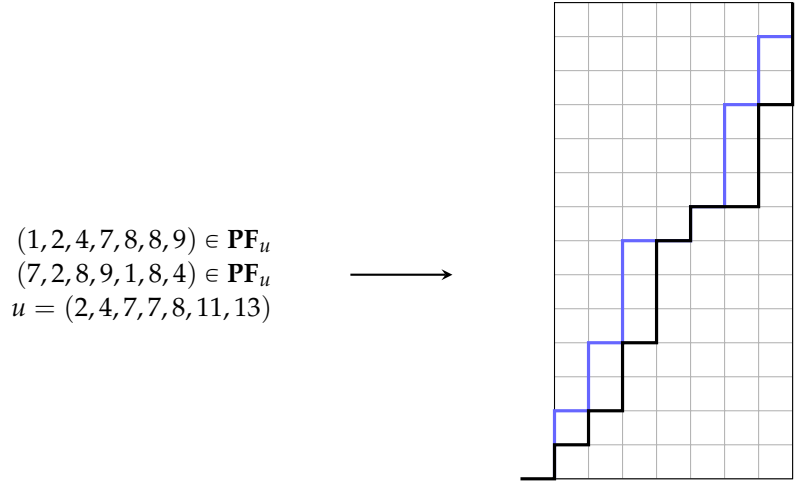

\begin{lemma} \label{lem:parking-as-weighted-path}
  Let $B$ be a boundary ending at $(m, n)$ and $P \in \P_B$ a path bounded by $B$.
  Let $\u = \h(B)$.
  If the composition-type of $P$ is $\comp(P) = (c_1 + 1, c_2, \dots, c_k)$, then the number of $\u$-parking functions $\alpha \in \PF_{\u}$ such that $P_\alpha=P$ is
  \[
    \frac{n!}{c_{1}! c_{2}! \cdots c_{k}!} = \binom{n}{c_1, c_2, \ldots, c_k}.
  \]
\end{lemma}
\begin{proof}
  By the premises, such $\alpha$ is a sequence with a $c_i$ number of $\h(P)_i$'s for all $i$, whose counting is given by the multinomial.
\end{proof}

Recall that a parking function $\alpha = (a_1, \dots, a_n)$ is prime if
\[
  | \{j \in [n] : a_j < i \}| \ge i \quad \text{for every } i > 1.
\]
We extend this notion to $\u$-parking functions as follows.

\begin{definition} \label{defn:prime-u-parking}
  A $\u$-parking function $\alpha = (a_1, \dots, a_n)$ is \tdef{prime} if 
  \[
    | \{ j \in [n] : a_j < u_i \} | \ge i \quad \text{for every } u_i > 1.
  \]
\end{definition}

\begin{remark} \label{rmk:prime-parking}
  We have four remarks regarding prime $\u$-parking functions:
  \begin{enumerate}
  \item For $\u = \h(B_{m,n})$ (the case of $(m, n)$-rational Dyck paths), our definition of prime rational parking functions coincides with that of Cai and Yan~\cite{Cai2020}.
    Thus, an inclusion-exclusion formula for counting prime rational parking functions can be found in~\cite[Theorem~3.4]{Cai2020}.
    It can also be seen as an easy exercise of the symbolic method by noticing that a rational parking function is a sequence of prime ones.
  \item For a given $\u = (u_1, \dots, u_n)$, let $i$ be the smallest index such that $u_i > 1$.
    According to~\cite{Kung2003}, the counting formula for prime $\u$-parking functions is given by the Gon\v{c}arov polynomials $g_n$ as
    \[
      (-1)^ng_n(0; u_1, \dots, u_{i-1}, u_i - 1, u_{i+1}, \dots, u_n).
    \]
  \item Armon, Beckford, Hanson, Krawzik,  Mandelshtam,  Martinez, and  Yan~\cite{Armon2025} introduced a different notion of prime $\u$-parking functions, for which a $\u$-parking function $\alpha = (a_1, \dots, a_n)$ is prime if $n = 1$, or $n \geq 2$ with
    \[
      | \{j \in [n] : a_j \le u_i \} | > i \quad \text{for every } i\in [n-1].
    \]
    For instance, $\alpha = (1,1,2,2)$ is a prime $(1,2,2,3)$-parking function in this sense, but not for our \Cref{defn:prime-u-parking}. 
  \item Comparing \Cref{defn:prime-u-parking} and the definition of $\u$-parking functions, for $\u = (u_1, \dots, u_n)$ such that $u_1 > 1$, it is clear that prime $\u$-parking functions are exactly $\u'$-parking functions, where $\u' = (u_1 - 1, \ldots, u_n - 1)$.
  \end{enumerate}
\end{remark}

Analogous to the classical case, we define a prime decomposition procedure for $\u$-parking functions.
Here, an empty parking function is considered as a prime parking function.

\begin{definition}\label{def:pri-dec}
  Let $\alpha = (a_1, \dots, a_n)$ be a $\u$-parking function.
  We distinguish two cases. 
  \begin{enumerate}
  \item $u_1 = 1$ or $\inc(\alpha)_1 < u_1$.
    Let $k$ be the largest value such that the subsequence $\alpha'$ of $\alpha$ with elements strictly smaller than $u_k$ is a prime $(u_1, \ldots, u_k)$-parking function.
    Let $L$ be the index set of $\alpha'$ in $\alpha$ and $m = u_{k + 1} - 1$.
    We check that $(a_i - m)_{i \in [n] \setminus L}$ is an $(1, u_{|L| + 2} - m, \dots, u_n - m)$-parking function.
    Repeating this procedure until getting a prime vector parking function leads to the decomposition of $\alpha$ into a sequence of prime vector parking functions.
  \item $\inc(\alpha)_1 = u_1$ and $u_1 > 1$.
    In this case, we first decompose $\alpha$ into an empty parking function and the modified sequence $\tilde{\alpha} = (a_1 - u_1 + 1, \dots, a_n - u_1 + 1)$. 
    It is clear that $\tilde{\alpha}$ is an $(1, u_2 - u_1 + 1, \dots, u_n - u_1 + 1)$-parking function, which is in the first case and can be thus decomposed into prime vector parking functions.
  \end{enumerate}
\end{definition}

\begin{remark}
  We note that, apart from the first prime vector parking function, every remaining factor and its corresponding vector has minimal entry equal to $1$.
  The reason we split into two cases for whether $u_1 = 1$ is to align the lattice path with the virtual $E$-step preceding its start.
\end{remark}

\begin{figure}
  \centering
  \begin{tikzpicture}[
    box/.style={
      draw, 
      rectangle, 
      thick, 
      align=left, 
      inner sep=6pt, 
      minimum height=0.8cm,
    },
    boxrd/.style={
      draw, 
      rectangle, 
      thick, 
      align=left, 
      inner sep=6pt, 
      fill=red!15,
      minimum height=0.8cm,
    },
    arr/.style={
      ->, 
      >=stealth, 
      thick,
    },
    btxt/.style={
      align=left, 
      inner sep=0pt, 
      font=\large
    }
    ]
    \node[box] (n1) at (0,0) {$\alpha=(\textcolor{red}{3},9,8,11,\textcolor{red}{2},6,11,6,9,\textcolor{red}{2}) \in PF_{(\textcolor{red}{3},\textcolor{red}{4},\textcolor{red}{4},6,8,8,10,11,11,13)}$};
    \node[box, below=1.5cm of n1.center, anchor=center] (n2) {\((4,3,6,\textcolor{red}{1},6,\textcolor{red}{1},4) \in PF_{(\textcolor{red}{1},\textcolor{red}{3},3,5,6,6,8)}\)};
    \node[boxrd, below=1.5cm of n1.center, anchor=center, xshift=5cm] (n3) {\((3,2,2) \in PPF_{(3,4,4)}\)};

    \node[box, below=1.5cm of n2.center, anchor=center] (n4) {\((\textcolor{red}{2},\textcolor{red}{1},4,4,\textcolor{red}{2}) \in PF_{(\textcolor{red}{1},\textcolor{red}{3},\textcolor{red}{4},4,6)}\)};
    \node[boxrd, below=1.5cm of n3.center, anchor=center] (n5) {\((1,1) \in PPF_{(1,3)}\)};

    \node[boxrd, below=1.5cm of n4.center, anchor=center] (n6) {\((1,1) \in PPF_{(1,3)}\)};
    \node[boxrd, below=1.5cm of n5.center, anchor=center] (n7) {\((2,1,2) \in PPF_{(1,3,4)}\)};

    \draw[arr] (n1.south) -- (n2.north);
    \draw[arr] (n1.south) -- (n3.north);

    \draw[arr] (n2.south) -- (n4.north);
    \draw[arr] (n2.south) -- (n5.north);

    \draw[arr] (n4.south) -- (n6.north);
    \draw[arr] (n4.south) -- (n7.north);
  \end{tikzpicture}
  \caption{An example of the prime decomposition of a vector parking function $\alpha$}
  \label{fig:prime-decomp}
\end{figure}
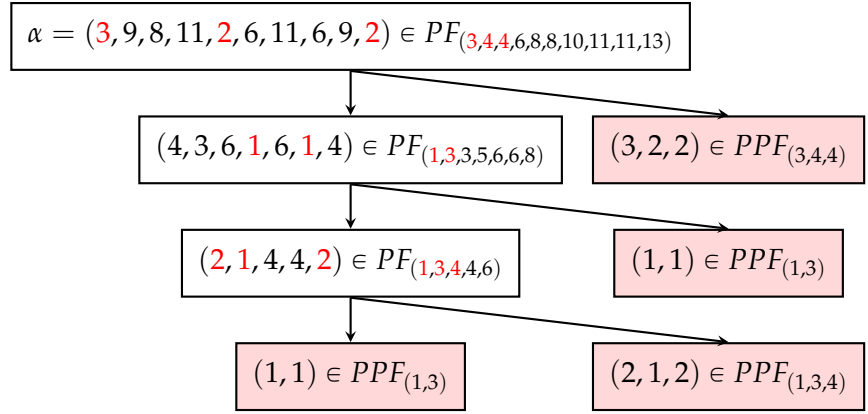

For a $\u$-parking function $\alpha$, its \tdef{prime} statistic, denoted $\pri(\alpha)$, is the number of prime vector parking functions in its prime decomposition, including the empty one when $u_1 > 1$. 
See \Cref{fig:prime-decomp} for an example.
The classical prime decomposition of parking functions is a special case of \Cref{def:pri-dec}. 

\begin{definition}
  Let $\alpha$ be a $\u$-parking function of length $n$. 
  The \tdef{block return} of $\alpha$, denoted $\breturn(\alpha)$, is equal to $\breturn(P_{\alpha})$.
  For a composition $\c$ such that $\c \le_c \vc(P_\u)$, the \tdef{composition run corresponding to $\c$} of $\alpha$ is defined by
  \begin{equation*}
    \run_\c(\alpha) := \max \{ k \in [n] : \{ \c_1, \c_1 + \c_2, \dots, \c_1 + \dots + \c_k\} \subseteq  \{1\} \cup \alpha \},
  \end{equation*}
  or simply $\run_\c(\alpha) = \run_\c(P_{\alpha})$. 
  For the two crucial choices $\c = (1, \dots, 1)$ and $\c = \vc(P_\u)$, we write
  \[
    \run(\alpha) := \run_{(1,\dots,1)}(\alpha)\quad\text{and}\quad\brun(\alpha) := \run_{\vc(P_\u)}(\alpha).
  \]
\end{definition}

\begin{lemma}\label{lem:cpf}
  Let $\alpha$ be a $\u$-parking function. For a composition $\c$ satisfying $\c \le_c \vc(P_\u)$, we have
  \[
    (\pri, \breturn, \run_\c)(\alpha) = (\return, \breturn, \run_\c)(P_\alpha).
  \]
\end{lemma}
\begin{proof}
  Let $\alpha = (a_1, \dots, a_n)$ be a $\u$-parking function. It is clear that $\breturn(\alpha) = \breturn(P_\alpha)$ and $\run_\c(\alpha) = \run_\c(P_\alpha)$, as the definitions of $\breturn$ and $\run_\c$ on $\u$-parking functions are defined via its corresponding bounded path.
  We show that $\pri(\alpha) = \return(P_\alpha)$ by induction on $\pri(\alpha)$.
  
  For the base case $\pri(\alpha) = 1$, by \Cref{defn:prime-u-parking}, $\inc(\alpha)_i < u_i$ for all $i$, meaning that the only $E$-step $P_\alpha$ and $P_{\u}$ share is the initial virtual one, thus $\return(P_\alpha) = 1$.
  Now suppose that $\pri(\alpha) = \return(P_\alpha)$ for all $\alpha, \u$ with $\pri(\alpha) < p$ for $p > 1$.
  We now show that the same holds for $\pri(\alpha) = p$.
  We distinguish the two cases in \Cref{def:pri-dec}.

  For the first case, $u_1 = 1$ or $\inc(\alpha)_1 < u_1$ if and only if the first non-virtual $E$-step of $P_\alpha$ does not start a return floor of height greater than $1$.
  Let $k$, $L$ and $m$ as defined in \Cref{def:pri-dec}, which means $k > 1$ by the prime condition of $\alpha'$ extracted from the index set $L$.
  If $k = n$, then $\pri(\alpha) = 1$, which is the base case.
  We have $k < n$, as otherwise $\pri(\alpha) = 1 < p$.
  By the maximality of $k$, we have $\inc(\alpha)_{k + 1} = u_{k + 1}$.
  Thus, $P_\alpha$ shares with $P_{\u}$ its $(k + 1)$-st non-virtual $E$-step, making a return floor.
  The segments of $P_\alpha$ and $P_\u$ starting from this step gives a path $P_{\alpha''}$ bounded by $P_{\u''}$ with $\u'' = (1, u_{k + 2} - m, \dots, u_{n} - m)$, and $\alpha'' = (a_i - m)_{i \in [n] \setminus L}$ a $\u''$-parking function, whose first non-virtual $E$-step does not start a return floor of height greater than $1$.
  Hence, $P_{\alpha''}$ has one less return floor than $P_\alpha$, and $\pri(\alpha) = \pri(\alpha'') + 1$.
  By the induction hypothesis, we have $\pri(\alpha) = \return(P_{\alpha''}) + 1 = \return(P_\alpha)$.

  For the second case, $\inc(\alpha)_1 = u_1 > 1$ if and only if the first non-virtual $E$-step of $P_\alpha$ starts a return floor at height greater than $1$.
  In this case, by \Cref{def:pri-dec}, we decompose $\alpha$ into an empty prime vector parking function and $\alpha'$, where $\alpha'$ is a $\u'$-parking function whose elements are those of $\alpha$ decremented by $u_1 - 1$, with $\u'$ being $\u$ decremented by the same amount.
  It is clear that $\alpha'$ falls in the first case, and $P_{\alpha'}$ (resp. $P_{\u'}$) is $P_\alpha$ (resp. $P_{\u}$) with the initial $N$-steps removed.
  The first non-virtual $E$-step of $P_{\alpha'}$ no longer starts a return floor, hence $\return(P_{\alpha'}) = \return(P_\alpha) - 1$.
  By the induction hypothesis, we have $\return(P_{\alpha'}) = \pri(P_\alpha)$, hence $\pri(\alpha) = \pri(\alpha') + 1 = \return(P_\alpha)$.
\end{proof}

Note that for $\alpha \in \PF_n$, we have $\breturn(\alpha) = \pri(\alpha)$. 
Combining \Cref{thm:Theta} with \Cref{lem:cpf} gives the following symmetry, generalizing \Cref{eq:Li-Lin} from classical parking functions to vector parking functions.

\begin{corollary}
\label{coro:PF}
  Given a weakly increasing vector of positive integers $\u$, we have 
  \begin{equation}
    \sum_{\alpha\in\PF_\u} t^{\run(\alpha)} q^{\breturn(\alpha)}
    = \sum_{\alpha\in\PF_\u} t^{\breturn(\alpha)} q^{\run(\alpha)}.
  \end{equation}
\end{corollary}

A $\u$-parking function $\alpha$ is \tdef{crowded} if its associated 
bounded path $P_\alpha$ is boundary constrained by $P_\u$.
We denote by $\CPF_\u$ the set of crowded $\u$-parking functions.
Combining \Cref{thm:Comp} with \Cref{lem:cpf} leads to another extension of \Cref{eq:Li-Lin}.

\begin{corollary}\label{coro:CPF}
  Let $\u$ be a weakly increasing vector of positive integers.
  For any composition $\c$ satisfying $\c\le_c\vc(P_\u)$, the following symmetry holds:
  \begin{equation}
    \sum_{\alpha\in\CPF_\u} t^{\run_\c(\alpha)} q^{\pri(\alpha)}
    = \sum_{\alpha\in\CPF_\u} t^{\pri(\alpha)} q^{\run_\c(\alpha)}.
  \end{equation}
  As special cases, we have
  \begin{equation}
    \sum_{\alpha\in\CPF_\u} t^{\run(\alpha)} q^{\pri(\alpha)}
    = \sum_{\alpha\in\CPF_\u} t^{\pri(\alpha)} q^{\run(\alpha)},
  \end{equation}
  \begin{equation}
    \sum_{\alpha\in\CPF_\u} t^{\brun(\alpha)} q^{\pri(\alpha)}
    = \sum_{\alpha\in\CPF_\u} t^{\pri(\alpha)} q^{\brun(\alpha)}.
  \end{equation}
\end{corollary}

\section{The distribution of $(\run,\pri)$ on $(a,b)$-parking functions} \label{sec6}

This section focuses on refined enumeration on a special class of vector parking functions called the \emph{$(a, b)$-parking functions}.
They generalize the classical parking functions and have been investigated extensively from the enumerative aspect~\cite{Eu2010,Kung2003,Stanley2018,WYY}.
Many enumerative results and combinatorial correspondences for classical  parking functions can be extended to $(a, b)$-parking functions, see~\cite[Sec.~13.4.4]{Yan2014}.
The main result of this section is an explicit expression of the exponential generating function of $(a, b)$-parking functions refined by $(\run, \pri)$, which is new even for classical parking functions.

Given two positive integers $a$ and $b$, an \tdef{$(a, b)$-parking function} of length $n$ is a $\u$-parking function with $\u \equiv \u(a, b, n) = (a, a + b, \dots, a + (n - 1) b)$.
Let $\PF_n(a, b)$ denote the set of $(a, b)$-parking functions of length $n$. 
Because $P_{\u(a, b, n)}$ has no double $E$-steps, \Cref{prop:bounded} implies $\CPF_{\u(a, b, n)} = \PF_n(a, b)$, which means \Cref{coro:CPF} applies to $\PF_n(a, b)$ as follows.

\begin{corollary}
  \label{coro:abPF}
  If $\c$ is a composition of $n$ with $|\c|\ge n$, $c_1\le a$, and $c_i\le b$ for all $i=2,\dots,n$, then
  \begin{equation}
    \sum_{\alpha\in\PF_n(a,b)} t^{\run_\c(\alpha)} q^{\pri(\alpha)}
    = \sum_{\alpha\in\PF_n(a,b)} t^{\pri(\alpha)} q^{\run_\c(\alpha)}.
  \end{equation}
\end{corollary}

The counting formula of $(a, b)$-parking functions was given in \cite[Equation~(7)]{Stanley2002}, which can be shown using Pollak's argument \cite{pollak} (see the comment after Theorem~11 in \cite{Stanley2002}):
\begin{equation}\label{eq:PF}
  |\PF_n(a,b)| = a (a + bn)^{n - 1}.
\end{equation}
On the other hand, the number of prime parking functions of length $n$ is $(n - 1)^{n - 1}$, a result first proven algebraically by Gessel and later combinatorially by Kalikow (see~\cite[Ex.~5.49(b)]{Stanley1999}).
We now extend Gessel's result to prime $(a, b)$-parking functions.
Let $\PPF_n(a,b)$ denote the set of prime $(a, b)$-parking functions of length $n$.
\begin{theorem}\label{thm:prime:ab}
  For integers $a, b > 0$, we have:
  \begin{equation}\label{eq:pri}
    |\PPF_n(a,b)| =
    \begin{cases}
      (a - 1) (a - 1 + bn)^{n - 1} &\text{ for } a > 1,\\
      (bn - 1)^{n - 1} &\text{ for } a = 1.
    \end{cases}
  \end{equation}
\end{theorem}

In the following, we are interested in the more general problem of enumerating $(a, b)$-parking functions refined by $\run$ and $\pri$. We define the exponential generating function of $(a, b)$-parking functions by
\[
  PF_{a,b}(x, q, t) := \varepsilon_a + \sum_{n \ge 1} \frac{x^n}{n!} \sum_{\alpha \in \PF_n(a,b)} q^{\run(\alpha)} t^{\pri(\alpha)}.
\]
Here, $\varepsilon_a$ stands for the weight of the empty $(a, b)$-parking function, with the convention $\varepsilon_a = 1$ when $a = 1$, and $\varepsilon_a = qt$ otherwise.
Using the first-passage decomposition of extended $b$-ballot paths, which corresponds to the prime decomposition of $(a, b)$-parking functions, we obtain the following expressions for $PF_{a,b}(x, q, t)$, which imply \Cref{eq:PF,thm:prime:ab}.

\begin{theorem} \label{thm:parking-gf}
  For integers $a, b > 0$, we have the following expression:
  \begin{equation}
    \label{eq:parking-gf}
    PF_{a, b}(x, q, t) = \frac{\varepsilon_{a}\, A_{b}(x)}{D_{b}(x, q, t)}\left(A_{b}(x)^{a-1}+ \mathbf{1}_{\{a\le 2\}}\frac{(1-q)(1-t)(A_{b}(x)-1)}{1-\mathbf{1}_{\{b=1\}}xqt}\right).
  \end{equation}
  where 
  \[
    A_{b}(x) := PF_{1, b}(x, 1, 1),
    \quad
    D_{b}(x, q, t) := (q + (1 - q) A_{b}(x))(t + (1 - t) A_{b}(x)).
  \]
  Here, $A_{b}(x)$ satisfies the functional equation
  \begin{equation}
    \label{eq:parking-gf-lambert}
    A_{b}(x) = \exp\left(xA_{b}(x)^b \right).
  \end{equation}
\end{theorem}
\begin{proof}
  By Lemma~\ref{lem:parking-as-weighted-path}, vector parking functions of length $n$ are in bijection with $E$-step-labeled bounded lattice paths, where labels of $E$-steps on the same floor are increasing.
  We adopt the convention that $b$ $N$-steps are appended to the end of these paths, which is compatible with the prime decomposition and does not affect the statistics.
  For $(a,b)$-parking functions, the corresponding bounded lattice paths are in bijection with certain \tdef{$b$-ballot paths}: $E$-steps (other than the virtual) are mapped to \tdef{up steps} $U = (1,b)$, $N$-steps to \tdef{down steps} $D = (1,-1)$, starting at $(0, a - 1)$ and ending on the $x$-axis while staying above it; see \Cref{fig:ballot-path} for an example.
  We denote by $\ballot_{a, b}$ the set of such $b$-ballot paths, including the empty path when $a = 1$.
  
  By \Cref{lem:cpf}, $\run$ and $\pri$ on $(a, b)$-parking functions are translated to $\run$ and $\return$ on bounded lattice paths.
  More precisely, on paths in $\ballot_{a, b}$, $\run$ now counts up step runs before the first double down steps, plus an extra $1$ when the path starts with a down step, and $\return$ counts up steps starting from the $x$-axis, plus an extra $1$ when $a > 1$.
  By abuse of notation, we still use $\run$ and $\return$ for the corresponding statistics on $b$-ballot paths. 

  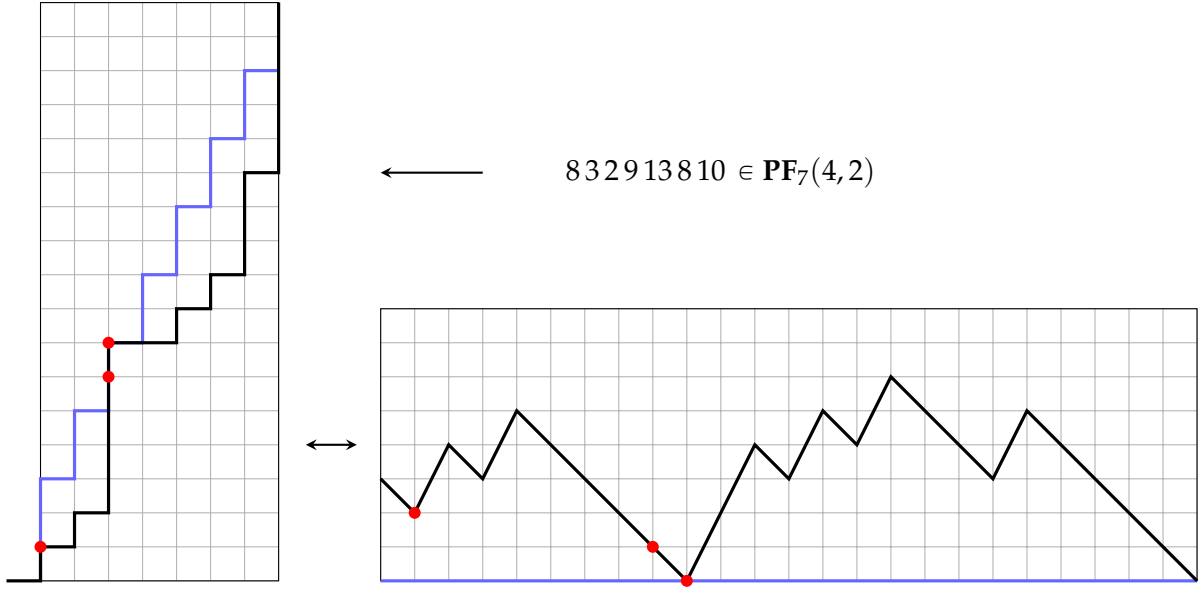
\begin{figure}
    \centering
    \begin{tikzpicture}[
      scale=0.45,
      >=stealth,
      grid/.style={gray!60, thin},
      boundary/.style={},
      pathA/.style={blue!60, very thick},
      pathB/.style={very thick},
      highlight/.style={red, dotted, very thick}
      ]

      \begin{scope}
        \draw[grid] (0,0) grid (7,17);
        \draw[boundary] (0,0) rectangle (7,17);
        \draw[pathA]
        (0,0) -- (0,3) -- (1,3) -- (1,5) -- (2,5)
        -- (2,7) -- (3,7) -- (3,9) -- (4,9) -- (4,11)
        -- (5,11) -- (5,13) -- (6,13) -- (6,15) -- (7,15);

        \draw[pathB]
        (-1,0) -- (0,0) -- (0,1) -- (1,1) -- (1,2) -- (2,2) -- (2,7) -- (3,7) -- (3,7) -- (4,7) -- (4,8) -- (5,8) -- (5,9) -- (6,9) -- (6,12) -- (7,12) -- (7,15)--(7,17);
        \fill[red] (0,1) circle (5pt);    \fill[red] (2,6) circle (5pt); \fill[red] (2,7) circle (5pt);
      \end{scope}

      \begin{scope}[xshift=7.8cm, yshift=4cm]
        \draw[<->, thick] (0,0) -- (1.5,0) node[midway,above] {};
      \end{scope}

      \begin{scope}[xshift=10cm]
        \draw[grid] (0,0) grid (24,8);
        \draw[boundary] (0,0) rectangle (24,8);
        \draw[pathA]
        (0,0) -- (24,0);
        \draw[pathB]
        (0,3)--(1,2) -- (2,4) -- (3,3) -- (4,5) -- (9,0) -- (11,4) -- (12,3) -- (13,5) -- (14,4) -- (15,6) -- (18,3) -- (19,5) -- (24,0);
        \fill[red] (1,2) circle (5pt);   \fill[red] (8,1) circle (5pt);   \fill[red] (9,0) circle (5pt);
      \end{scope}

      \begin{scope}[xshift=16cm, yshift=12cm]
        \node at (4, 0) {$8\,3\,2\,9\,13\,8\,10\,\in \PF_{7}(4, 2)$};
        \draw[->, thick] (-3,0) -- (-6,0) node[midway,above] {};
      \end{scope}

    \end{tikzpicture}
    \caption{From an $(a,b)$-parking function to a $b$-ballot path starting at $(0,a-1)$ with $a=4$ and $b=2$, where both paths decompose at each red node. }
    \label{fig:ballot-path}
  \end{figure}

  For convenience, we write $A \equiv A_{b}(x)$.
  We first consider the case $a = 1$.
  Suppose that $b > 1$.
  For a $b$-ballot path $P \in \ballot_{1, b}$, if it is not empty, it starts with an up step.
  Let $k$ be the number of initial up steps in $P$.
  We can then decompose $P$ by its first visit to height $bk - 1, bk - 2, \ldots, 1, 0$ as $P = U^k D P^{(bk - 1)} D P^{(bk - 2)} D \cdots P^{(1)} D P^{(0)}$.
  By the definitions of $\run$ and $\return$, we have $\run(P) = 1 + \run(P^{(bk - 1)})$ and $\return(P) = 1 + \return(P^{(0)})$.
  Note that as $b > 1$, the two subpaths $P^{(0)}$ and $P^{(bk - 1)}$ are distinct.
  Taking into account the weight $1 / k!$ for the $k$ initial up steps of $P$, we have
  \begin{align}
    \begin{split} \label{eq:parking-gf-a-1-exp}
      PF_{1, b}(x, q, t) &= 1 + \sum_{k \geq 1} \frac{qt x^k}{k!}  PF_{1, b}(x, q, 1)\,A^{bk - 2}\,PF_{1, b}(x, 1, t) \\
                        &= 1 + qt \left( \exp\left( xA^b \right) - 1 \right) \frac{PF_{1, b}(x, q, 1)\, PF_{1, b}(x, 1, t)}{A^{2}}.
    \end{split}
  \end{align}
  Substituting $q = t = 1$ in \Cref{eq:parking-gf-a-1-exp} leads to \Cref{eq:parking-gf-lambert}.
  Substituting \Cref{eq:parking-gf-lambert} back into \Cref{eq:parking-gf-a-1-exp} for the exponential expression, we have
  \begin{equation}
    \label{eq:parking-gf-a-1-frac}
    PF_{1, b}(x, q, t) = 1 + qt (A - 1) \frac{PF_{1, b}(x, q, 1)\, PF_{1, b}(x, 1, t)}{A^{2}}.
  \end{equation}
  Substituting $q = 1$ in \Cref{eq:parking-gf-a-1-frac} leads an equation involving only $PF_{1, b}(x, 1, t)$ and $A$, whose solution for $PF_{1, b}(x, 1, t)$ is
  \begin{equation}
    \label{eq:parking-gf-a-1-t}
    PF_{1, b}(x, 1, t) = \frac{A}{(1 - t) A + t}.
  \end{equation}
  We observe in \Cref{eq:parking-gf-a-1-frac} that $PF_{1, b}(x, q, t)$ is symmetric in $q, t$, thus
  \begin{equation}
    \label{eq:parking-gf-a-1-q}
    PF_{1, b}(x, q, 1) = \frac{A}{(1 - q) A + q}.
  \end{equation}
  Substituting \Cref{eq:parking-gf-a-1-t,eq:parking-gf-a-1-q} into \Cref{eq:parking-gf-a-1-frac}, we get the final expression of $PF_{1, b}(x, q, t)$ in \Cref{eq:parking-gf}.

  When $b = 1$, the two subpaths $P^{(0)}$ and $P^{(bk - 1)}$ may coincide when $k = 1$.
  By isolating this case, we have
  \begin{align}
    \begin{split} \label{eq:parking-gf-a-1-b-1-exp}
      PF_{1, 1}(x, q, t)
      &= 1 + \sum_{k \geq 2} \frac{qt x^k}{k!} PF_{1, 1}(x, q, 1) A^{k - 2} PF_{1, 1}(x, 1, t) + qtx PF_{1, 1}(x, q, t) \\
      &= 1 + qt (\exp(xA) - 1 - xA) \frac{PF_{1, 1}(x, q, 1) PF_{1, 1}(x, 1, t)}{A^{2}} + qtx PF_{1, 1}(x, q, t).
    \end{split}
  \end{align}
  Note that the substitution $q = t = 1$ in \Cref{eq:parking-gf-a-1-b-1-exp} leads to \Cref{eq:parking-gf-lambert} with $b = 1$, hence \Cref{eq:parking-gf-lambert} still holds for $b = 1$.
  Substituting \Cref{eq:parking-gf-lambert} into \Cref{eq:parking-gf-a-1-b-1-exp} leads to
  \begin{equation}
    \label{eq:parking-gf-a-1-b-1-frac}
    PF_{1, 1}(x, q, t) = 1 + qt (A - 1 - xA) \frac{PF_{1, 1}(x, q, 1) PF_{1, 1}(x, 1, t)}{A^{2}} + qtx PF_{1, 1}(x, q, t).
  \end{equation}
  Substituting $q = 1$ in \Cref{eq:parking-gf-a-1-b-1-frac} leads to an equation involving only $PF_{1, 1}(x)$ and $PF_{1, 1}(x, 1, t)$, of which \Cref{eq:parking-gf-a-1-t} is a solution with $b = 1$.
  We thus know that \Cref{eq:parking-gf-a-1-t} and by symmetry \Cref{eq:parking-gf-a-1-q} hold for $b = 1$.
  Substituting \Cref{eq:parking-gf-a-1-t,eq:parking-gf-a-1-q} into \Cref{eq:parking-gf-a-1-b-1-frac} and solving for $PF_{1, 1}(x, q, t)$, we get the expression of $PF_{1, 1}(x, q, t)$ in \Cref{eq:parking-gf}.

  We now consider the case $a = 2$, given by $b$-ballot paths $P \in \ballot_{2, b}$.
  We have two cases:
  \begin{itemize}
  \item $P$ starts with a down step, in which case $P = D P'$ with $P' \in \ballot_{1, b}$. We have $\run(P) = 1 + \run(P')$ and $\return(P) = 1 + \return(P')$, with the extra $1$ for $\run(P)$ due to $P$ starting with a down step, and that for $\return(P)$ due to $a > 1$;
  \item $P$ starts with an up step, in which case $P = P_1 D P_2$ with $P_1, P_2 \in \ballot_{1, b}$ and $P_1$ is non-empty. We have $\run(P) = \run(P_1)$ and $\return(P) = 1 + \return(P_2)$, the former due to $P_1$ being non-empty thus ending with a down step, and the latter due to $a > 1$.
  \end{itemize}
  Summing up the contribution of both cases, we have
  \begin{equation}
    \label{eq:parking-gf-a-2-expr}
    PF_{2, b}(x, q, t) = qt PF_{1, b}(x, q, t) + t (PF_{1, b}(x, q, 1) - 1) PF_{1, b}(x, 1, t).
  \end{equation}
  The $a = 1, b > 1$ and $a = b = 1$ cases of \Cref{eq:parking-gf} along with \Cref{eq:parking-gf-a-2-expr,eq:parking-gf-a-1-t,eq:parking-gf-a-1-q} gives the expression of $PF_{2, b}(x, q, t)$ in \Cref{eq:parking-gf}.
  
  For the general case $a \geq 3$, we decompose $b$-ballot paths $P \in \ballot_{a, b}$ by their first visit to height $a - 1, a - 2, \ldots, 0$ as $P = P^{(a - 1)} D P^{(a - 2)} D \cdots P^{(1)} D P^{(0)}$, where $P^{(i)} \in \ballot_{1, b}$ for all $i$.
  There are again two cases, similar to those for $a = 2$, and obtained through the same reasoning:
  \begin{itemize}
  \item $P^{(a - 1)}$ is empty, in which case $\run(P) = 1 + \run(P^{(a - 2)})$ and $\return(P) = 1 + \return(P^{(0)})$;
  \item $P^{(a - 1)}$ is non-empty, in which case $\run(P) = \run(P^{(a - 1)})$ and $\return(P) = 1 + \return(P^{(0)})$.
  \end{itemize}
  Note that the only difference for $a > 2$ is that $P^{(a - 2)}$ and $P^{(0)}$ are distinct.
  Summing up the contributions, we have
  \begin{align}
    \label{eq:parking-gf-a-expr}
    FP_{a, b}(x, q, t) &= qt\, PF_{1, b}(x, q, 1)\, A^{a - 3} PF_{1, b}(x, 1, t) + t\, (PF_{1, b}(x, q, 1) - 1) A^{a - 2} PF_{1, b}(x, 1, t).
  \end{align}
  Substituting \Cref{eq:parking-gf-a-1-t,eq:parking-gf-a-1-q} in \Cref{eq:parking-gf-a-expr}, all valid for $b \geq 1$, leads to the expression of $PF_{a, b}(x, q, t)$ in \Cref{eq:parking-gf}.
\end{proof}

\begin{remark} \label{rmk:parking-gf-lambert}
  Note that by applying Lagrange inversion on \Cref{eq:parking-gf-lambert} for $A_{b}(x)$ yields $A_{b}(x) = 1 + \sum_{n \geq 1} (bn + 1)^{n - 1} \frac{x^n}{n!}$.
  Using the same technique, we deduce \Cref{eq:PF,thm:prime:ab}:
  \begin{itemize}
  \item For \Cref{eq:PF}, from \Cref{thm:parking-gf}, we have
    \[
      PF_{a, b}(x) \equiv FP_{a, b}(x, 1, 1) = A_b(x)^a.
    \]
    Applying Lagrange inversion along with \Cref{eq:parking-gf-lambert}, we obtain \Cref{eq:PF}.
  \item For \Cref{thm:prime:ab}, the case $a > 1$ comes directly from \Cref{eq:PF}, as $\PPF_n(a,b) = \PF_n(a - 1, b)$ by \Cref{rmk:prime-parking}~(4).
    For $a = 1$, we observe that their generating function is exactly $PF_{1, b}(x, 1, 0)$, which by \Cref{thm:parking-gf} is equal to $A_b(x)^{a - 1}$, and we obtain the result using Lagrange inversion and \Cref{eq:parking-gf-lambert}.
  \end{itemize}
  The same can be done for $PF_{a, b}(x, q, t)$, as all expressions on the right-hand side of \Cref{eq:parking-gf} are rational fractions of $A_{b}(x)$.
\end{remark}

\begin{remark}
  Using the same technique, with a more involved decomposition, we may obtain similar expressions for $\u$-parking functions with $\u$ ultimately periodic considered in \cite[Section~6]{Cai2020}, refined by $(\pri, \run)$, using the symbolic method \cite{flajolet}. More precisely, as $\u$ is ultimately periodic, it is periodic after a certain finite $N$ elements, whose effect can be taken into account by finitely many generating functions. Then for the period $\u'$, for each of its elements we may write the decomposition by the first step, leading to a system similar to the one in the proof of \Cref{thm:parking-gf}. Its strongly connected component leads to an analogue of \Cref{eq:parking-gf-lambert}, but of the form $A = \exp(Q(x, A))$, where $Q$ is a polynomial with positive coefficients. Other refined generating functions will be rational fraction of $A$.
\end{remark}

\begin{remark}\label{rmk:cycle}
  We note that many interesting statistics on $(a, b)$-parking functions $\alpha$ of length $n$ can be easily translated to $b$-ballot paths.
  Examples include the preference profile
  \[
    p(\alpha) = (|\{ i \in [a + (n - 1)b] : m_i(\alpha) = k \}|)_{k \geq 1}, \text{ where } m_i(\alpha) = |\{ j \in [n] : \alpha_j = i \}|.
  \]
  and the multiset of forward differences \cite{forward-diff}
  \[
    \Delta(\alpha) = (|\{ i: \alpha_{i} - \alpha_{i - 1} \equiv \ell \pmod{M} \}|)_{\ell \geq 0}, \text{ with } M =
    \begin{cases}
      bn - 1 & \text{ if } a = 1; \\
      a - 1 + bn & \text{ if } a > 1.
    \end{cases}
  \]
  They are translated respectively to the profile of up steps aligned to each line of slope $b$, and to pairs of up steps with consecutive labels on lines with $x$-translation distance $\ell$ modulo $M$.
  As such statistics on $b$-ballot paths are stable under the rotation operation dictated by the corresponding cycle lemma in each case, using cycle lemma (or equivalently Pollak's argument~\cite{pollak}), their refined enumeration can be deduced from that on unconstrained paths, which is much simpler.
  For instance, for preference profile, it is clear that
  \[
    \sum_{\alpha \in \PF_n(a, b)} \prod_{j \geq 1} u_j^{m_i(\alpha)} \quad \propto_a \quad \frac{n!}{M} [z^n] \left( 1 + \sum_{j \geq 1} u_j \frac{z^j}{j!} \right)^M,
  \]
  with $M$ defined in the definition of the multiset of forward differences. Here, $\propto_a$ means that the proportional coefficient depends only on $a$.
\end{remark}

\section{Concluding remarks}\label{sec7}

In this paper, we establish several symmetries on vector parking functions, but the core arguments are carried out on bounded lattice paths.
To provide a unified framework for the run-like statistics, we introduced the composition runs. We showed that for any composition $\c$ satisfying $\c \le_c \vc(B)$ (componentwise order), $\run_\c$ is symmetric with $\return$ on $\CP_B$ (\Cref{thm:Comp}).
An immediate consequence is that any two such statistics $\run_\c$ and $\run_{\c'}$ with $\c, \c' \le_c \vc(B)$ are equidistributed on $\CP_B$ (\Cref{coro:equidist}).
While it follows from composing the involutions $\Psi_\c$ and $\Psi_{\c'}$, it would be interesting to construct a direct combinatorial explanation.

To answer Problem~\ref{prob:Dai}, we introduced the block return, which exhibits a remarkable symmetry with the run statistic on bounded lattice paths (Theorem~\ref{thm:Theta}).
When $\run$ is replaced by $\run_{\c}$, this symmetry---and even equidistribution---is lost.
We pose the following natural problem for further investigation. 

\begin{problem}
  Let $B$ be a boundary and $\c$ a composition with $\c \le_c \vc(B)$.
  If $\c \neq (1, \dots, 1)$, motivated by the symmetry between $\run$ and $\return$, can one associate to each composition $\c$ a statistic $\return_\c$  that is jointly symmetric with $\run_\c$ on $\P_B$?
\end{problem}

\section*{Acknowledgments}

Fang was partially supported by the French \textit{Agence nationale de recherche} project PAGCAP (ANR-21-CE48-0020).
Li and Lin were supported by the National Science Foundation of China (grants 12271301 \& 12322115) and the Fundamental Research Funds for the Central Universities. 
We used LLM only for research on literature and drafting the first version of some figures.

\printbibliography

\end{document}